\documentclass[11pt]{amsart}
\usepackage{amsmath,amscd,amsfonts,amssymb,amsthm,epsfig,mathrsfs,color,bm,shuffle}
\makeatletter
\let\@wraptoccontribs\wraptoccontribs
\makeatother
\usepackage[all]{xy}
\SelectTips{cm}{10}
\usepackage[colorlinks,citecolor=black,linkcolor=black,urlcolor=black,
  pdfstartview=FitH]{hyperref}
\hypersetup{    colorlinks,    citecolor=blue,    filecolor=black,
    linkcolor=blue,    urlcolor=black}
\usepackage{tikz}
\usepackage{ctable,array}

\usepackage{booktabs}

\newcommand{\Zoo}{\mathring{\mathring Z}}
\newcommand{\Zo}{\mathring Z}
\newcommand\xt[1][t]{[\hspace{-.12em}[ {#1} ]\hspace{-.12em}]}
\newcommand{\mdeg}[1][]{\mathrm{mdeg}_{#1}}

\DeclareMathOperator\sh{sh}

\usepackage[vcentermath]{youngtab}
\usepackage[thinlines]{easybmat}
\newcommand\bsm{\begin{smallmatrix}}
	\newcommand\esm{\end{smallmatrix}}
\newcommand\Grf{\mathbb G}
\usepackage{tikz-cd,longtable}
\DeclareMathOperator\Span{Span}

\newtheorem{theorem}{Theorem}[section]
\newtheorem{proposition}[theorem]{Proposition}
\newtheorem{lemma}[theorem]{Lemma}
\newtheorem{corollary}[theorem]{Corollary}
\newtheorem{conjecture}[theorem]{Conjecture}
\newtheorem{definition}[theorem]{Definition}
\newtheorem{example}[theorem]{Example}
\newtheorem{remark}[theorem]{Remark}
\newtheorem{question}[theorem]{Question}
\theoremstyle{definition}

\newcommand{\C}{\mathbb{C}}
\newcommand{\Cx}{\mathbb{C}^\times}
\newcommand{\N}{\mathbb{N}}
\newcommand{\Z}{\mathbb{Z}}
\newcommand{\D}{\mathbb{D}}
\newcommand{\K}{\mathcal{K}}
\newcommand{\R}{\mathbb{R}}

\newcommand{\vi}{\mathbf{i}}
\newcommand{\vj}{\mathbf{j}}
\newcommand{\vk}{\mathbf{k}}
\newcommand{\ON}{\mathbb C[N]}
\newcommand{\barD}{\overline{D}}

\newcommand{\GO}{G^\vee(\mathcal{O})}

\newcommand{\A}{{\mathbb A}}
\newcommand{\basefield}{\C}
\newcommand{\tR}{\mathfrak{t}^*_{\mathbb R}}

\newcommand{\ft}{{\mathfrak t}}
\newcommand{\fg}{{\mathfrak g}}
\newcommand{\ftreg}{{\ft^{\mathrm{reg}}}}
\newcommand{\dimvec}{\underrightarrow{\dim}\,}
\newcommand{\varep}{\varepsilon}

\DeclareMathOperator\Hom{Hom}

\DeclareMathOperator\im{im}
\DeclareMathOperator\wt{wt}
\DeclareMathOperator\ad{ad}
\DeclareMathOperator\Ad{Ad}

\DeclareMathOperator\Irr{Irr}

\newcommand\Perv{\mathrm P_{G^\vee(\mathcal O)}(\Gr)}
\newcommand\Vect{\mathrm{Vect}_{\mathbb C}}
\newcommand\Rep{\mathrm{Rep}(\overline G)}
\DeclareMathOperator\Conv{Conv}
\DeclareMathOperator\Seq{Seq}

\DeclareMathOperator\supp{supp}

\DeclareMathOperator\Sym{Sym}
\newcommand\PGL{{\mathbf{PGL}}}
\newcommand\SL{{\mathbf{SL}}}
\newcommand\id{{\mathrm{id}}}
\newcommand\IC{{\mathcal I}}
\newcommand\Gr{{\mathcal Gr}}

\newcommand\Pol{{\mathrm{Pol}}}

\newcommand\hH{{\overset{{\scriptscriptstyle \vee}}{\mathcal H}}{}}
\newcommand\PP{{\mathcal{PP}}}

\newcommand\iso\cong
\newcommand\shin{\shuffle}

\newcommand\defn[1]{{\bf #1}}

\begin{document}

\title[The Mirkovi\' c--Vilonen basis]{The Mirkovi\'c--Vilonen basis and \\ Duistermaat--Heckman measures}

\author{Pierre Baumann}
\email{p.baumann@unistra.fr}
\address{Institut de Recherche Math\'ematique Avanc\'ee\\
Universit\'e de Strasbourg et CNRS}

\author{Joel Kamnitzer}
\email{jkamnitz@math.utoronto.ca}
\address{Department of Mathematics\\
University of Toronto}

\author{Allen Knutson}
\email{allenk@math.cornell.edu}
\address{Department of Mathematics\\
Cornell University}
\contrib[With an appendix by]{Anne Dranowski, Joel Kamnitzer, and Calder Morton-Ferguson}

\begin{abstract}
	Using the geometric Satake correspondence, the Mirkovi\'c--Vilonen cycles in the affine Grasssmannian give bases for representations of a semisimple group $ G $.  We prove that these bases are ``perfect'', i.e. compatible with the action of the Chevelley generators of the positive half of the Lie algebra $ \mathfrak g $.  We compute this action in terms of intersection multiplicities in the affine Grassmannian.  We prove that these bases stitch together to a basis for the algebra $\C[N]$ of regular functions on the unipotent subgroup.  We compute the multiplication in this MV basis using intersection multiplicities in the Beilinson--Drinfeld Grassmannian, thus proving a conjecture of Anderson.
	
	In the third part of the paper, we define a map from $ \C[N] $ to a convolution algebra of measures on the dual of the Cartan subalgebra of $ \mathfrak g $.  We characterize this map using the universal centralizer space of $ G$.  We prove that the measure associated to an MV basis element equals the Duistermaat--Heckman measure of the corresponding MV cycle.  This leads to a proof of a conjecture of Muthiah.
	
	Finally, we use the map to measures to compare the MV basis and Lusztig's dual semicanonical basis.  We formulate conjectures relating the algebraic invariants of preprojective algebra modules (which underlie the dual semicanonical basis) and geometric invariants of MV cycles.  In the appendix, we use these ideas to prove that the MV basis and the dual semicanonical basis do not coincide in $SL_6 $.

\end{abstract}

\date{\today}
\maketitle

\setcounter{tocdepth}{1}
\tableofcontents

\newcommand\lie[1]{{\mathfrak #1}}
\newcommand\tensor\otimes
\newcommand\into\hookrightarrow
\newcommand\onto\twoheadrightarrow

\section{Introduction}\label{sec:intro}
\subsection{Biperfect bases and MV polytopes}\label{ssec:biperfect}
Let $ G $ denote a simple simply-connected complex algebraic group.  Going back to the work of Gel$'$fand--Zelevinsky \cite{GZ}, there has been great interest in finding special bases for irreducible representations $L(\lambda) $ of $ G $. 
Good bases restrict to bases of weight spaces and induce bases of tensor product multiplicity spaces.

Rather than work with individual representations, it is convenient to pass to the coordinate ring $ \C[N] $ (where $ N \subset B \subset G $ is the unipotent radical of a Borel), which contains all the irreducible representations.  Berenstein--Kazhdan \cite{BerensteinKazhdan} introduced the notion of perfect bases for $ L(\lambda) $ and $ \C[N]$; in this paper, we slightly modify their definition and work with \textbf{biperfect bases} for $\C[N]$.  Biperfect bases have good behaviour with respect to the left and right actions of the Chevalley generators $ \{ e_i \} \subset \lie{n} $ on $ \C[N] $ (see \S \ref{ss:PerfBas}).

For $ G = SL_2, SL_3, SL_4$, biperfect bases exist, are unique, and are given by explicit formulas.  However, for general $ G $, uniqueness does not hold, nor are explicit formulas available.  Some constructions of biperfect bases are known in general. The first general construction was Lusztig's dual canonical basis \cite{Lusztig}, aka Kashiwara's upper global basis \cite{Kashiwara91}.  In this paper, we will focus on the \textbf{Mirkovi\'c--Vilonen basis}, which is constructed for any $ G $ using the geometric Satake correspondence, and Lusztig's \textbf{dual semicanonical basis}, which is constructed for simply-laced $ G $ using preprojective algebra modules. 

Though uniqueness does not hold in general, a beautiful result of Berenstein--Kazhdan \cite{BerensteinKazhdan} shows that every biperfect basis has the same underlying combinatorics.  More precisely, any biperfect basis $ B $ comes with maps $ \tilde e_i, \tilde e^*_i : B \rightarrow B \cup \{0\} $ (for each $ i \in I$) which approximate the left and right actions of $ e_i $ on $ B$.  These maps $ \tilde e_i, \tilde e^*_i $ endow $ B $ with a bicrystal structure.  If $ B, B' $ are two biperfect bases, then there is a unique isomorphism of bicrystals $ B \rightarrow B' $ (Theorem \ref{th:UniqCrys}).  We write $ B(\infty) $ for the abstract bicrystal common to all biperfect bases.

\subsection{Bases and their polytopes}\label{ssec:bases}
Let $ G^\vee $ be the Langlands dual group and let
$ \Gr = G^\vee((t))/G^\vee[[t]] $ denote the affine Grassmannian of
this group.  Mirkovi\'c--Vilonen \cite{MirkovicVilonen} defined a family of cycles in $ \Gr $
which, under the geometric Satake correspondence, give bases for
irreducible representations of $ G $.  We will give a slight
modification of their construction in order to get a basis for
$ \C[N] $.  Let $ S_{\pm}^\mu := N^\vee_{\pm}((t)) L_\mu $ denote
semi-infinite orbits in $ \Gr$ (where $ L_\mu \in \Gr $ is the point defined
by the $ G^\vee$-coweight $ \mu $).  We will be concerned with the
intersection of opposite semi-infinite orbits.  For any
$ \nu \in Q_+ $, the positive root cone, the irreducible
components of $ \overline{S_+^0 \cap S_-^{-\nu}} $ are called
\textbf{stable MV cycles}.  These cycles index the MV basis
$ \{ b_Z \}$ for $ \C[N] $.

In addition to the MV basis, in this paper, we also study the dual semicanonical basis for $ \C[N] $, which was introduced by Lusztig \cite{Lusztig00} and further studied by Geiss--Leclerc--Schr\"oer \cite{GeissLeclercSchroer}.   This basis $ \{ c_Y \}$ is indexed by irreducible components $ Y $ of Lusztig's nilpotent varieties, which parameterize representations of the preprojective algebra $\Lambda$.  To each $\Lambda$-module $ M $, we can associate a vector $ \xi_M \in \C[N] $ and we define $ c_Y := \xi_M $ for a general point $ M \in Y $.

As the MV basis and dual semicanonical basis are both biperfect bases (Theorems \ref{th:MVBasisBiperf} and \ref{th:LusztigPerfect}), 
we get canonical bijections
$$
\{ b_Z \} \longrightarrow B(\infty) \longleftarrow \{ c_Y \}
$$

Suppose that $ b_Z $ and $ c_Y $ correspond under these bijections.  Outside of small rank, this doesn't imply that $ b_Z = c_Y $ as elements of $ \C[N] $, just that they represent the same element of the crystal.  However, this combinatorial relationship is very appealing, as can be seen by the following result which combines the work of the second author \cite{mvcrystal} and the first and second authors (together with Tingley) \cite{BaumannKamnitzerTingley}.

\begin{theorem} \label{th:polytopesIntro}
Let $  Z$ and $ Y $ be as above and let $ M $ be a general point of $ Y $.  Then we have an equality of polytopes.
$$ \Conv( \{ \mu \mid L_\mu \in Z \}) = - \Conv (\{ \dimvec N \mid N \subseteq M \text{ is a $ \Lambda$-submodule} \})$$
\end{theorem}

MV polytopes were first defined by Anderson \cite{Anderson03} as the left hand side of the above equality.  However, in retrospect, it is more natural to view them as purely combinatorial objects: they can either be defined using a condition on their 2-faces (as in \cite{mvcrystal}) or from the crystal $B(\infty) $ using Saito reflections (as in \S \ref{ss:MVpolytopes}).

The above equality of polytopes motivates the following question.
\begin{question} \label{ques:main}
Let $ Z, Y$ be as above.  Is there a relationship between the equivariant invariants of $Z$ and the structure of a general point of $ Y $?  How is this connected to the relationship between the basis vectors $ b_Z $ and $ c_Y $?
\end{question}

\subsection{The MV basis}
The first part of the paper is devoted to understanding the MV basis.  We prove the following results in Theorems \ref{th:MVBasisBiperf}, \ref{th:ActEiMV}, \ref{th:mult}.

\begin{theorem} \label{th:MVbasisIntro}
\begin{enumerate}
\item The MV basis $ \{ b_Z \} $ is a biperfect basis for $ \C[N] $.
\item For each $ i $, the action of $ e_i $ on $ b_Z $ is given by the intersection multiplicities appearing in the intersection of $ Z $ with a hyperplane section.
\item Given two MV cycles $ Z_1, Z_2 $, the product $ b_{Z_1} b_{Z_2} $ in $ \C[N] $ is given by the intersection multiplicities appearing in the intersection of the Beilinson--Drinfeld degeneration of $ Z_1 \times Z_2 $ with the central fibre.
\end{enumerate}
In particular, the structure constants for the action of $ e_i $ and for the multiplication are non-negative integers.
\end{theorem}

Part (i) of this theorem is a two-sided extension to $\mathbb C[N]$ of a like result from Braverman and Gaitsgory \cite[Proposition~4.1]{BravermanGaitsgory}. Part (iii) was conjectured by Anderson in \cite{Anderson03}.  We prove (i) and (ii) using a result of Ginzburg \cite{Ginzburg} and Vasserot \cite{Vasserot} concerning the action of the principal nilpotent under the geometric Satake correspondence.  We prove (iii) by carefully considering the fusion product in the Satake category, as defined by Mirkovi\'c--Vilonen \cite{MirkovicVilonen}.

Part (iii) of this theorem is closely related to an old result of Feigin--Finkelberg--Kuznetzov--Mirkovi\'c \cite{FFKM} and a more recent result of Finkelberg--Krylov--Mirkovi\'c \cite{FK} describing the algebra $ U(\mathfrak n) $ using Zastava spaces.

\subsection{Equivariant invariants of MV cycles}
\label{ssec:equivinvts}
In order to further our understanding of the Mirkovi\'c--Vilonen basis, we relate the basis vector $ b_Z \in \C[N] $ to equivariant invariants of the MV cycle $ Z$.  We begin by introducing a remarkable map from $ \C[N] $ to the space $\PP $ of piecewise polynomial measures on $\tR $. Given any sequence $ \vi = (i_1, \dots, i_p) \in I^p $ of simple roots, we define a measure $ D_\vi $ on $ \tR $ as follows.  First, we consider a linear map $ \pi_\vi : \R^{p+1} \rightarrow \tR $ taking the standard basis vectors to the negative partial sums  $ \alpha_{i_1} + \cdots + \alpha_{i_k} $.  Then we define $ D_\vi $  to be the pushforward of Lesbesgue measure on the $p$-simplex under $ \pi_\vi $ (see Figure \ref{fig:MeasureExamples}).

\begin{figure} 
\begin{tikzpicture}[scale = .4]
\shadedraw[top color=black, bottom color=white,shading angle = 240]
(0,0) node[left] {$-\alpha_1$} -- ++(-3,-5) node[below] {$-2\alpha_1 -\alpha_2$}
-- ++(4.5,2.5) -- ++(-1.5,2.5);
\shadedraw[top color=black, bottom color=lightgray,shading angle = 0]
++(-3,-5) -- ++(6,0) node[below] {$-\alpha_1-\alpha_2$} -- ++(-1.5,2.5) 
 -- ++(-4.5,-2.5);
\shadedraw[top color=black, bottom color=white,shading angle = 60]
++(3,-5) -- ++(3,5) -- ++(-4.5,-2.5) -- ++(+1.5,-2.5);
\shadedraw[top color=black, bottom color=lightgray,shading angle = 180]
(6,0) node[right] {$0$} -- ++(-6,0)  -- ++(1.5,-2.5) -- ++(4.5,2.5);
\draw[dashed] (6,0) -- ++(3,-5) -- ++(-6,0);
\end{tikzpicture}
\begin{tikzpicture}[scale = .4]
\shadedraw[top color=darkgray, bottom color=white,shading angle = -150]
(0,0) node[right] {$0$} -- ++(-9,-5) node[below] {$-2\alpha_1 -\alpha_2$} 
-- ++(6,0) -- ++(3,5);
\shadedraw[top color=darkgray, bottom color=white,shading angle = 120]
(0,0) -- ++(3,-5) node[below] {$-\alpha_2$}-- ++(-6,0)  -- ++(3,5);
\draw[dashed] (0,0) -- ++(-6,0) -- ++(-3,-5);
\end{tikzpicture}

 \caption{The $SL_3$ examples of $ D_\vi $ for $ \vi = (1,2,1) $ and $(2,1,1)$, 
 vertices labeled by their positions, with the shading to suggest the 
 (piecewise-linear function times Lebesgue) measure. } \label{fig:MeasureExamples}
\end{figure}
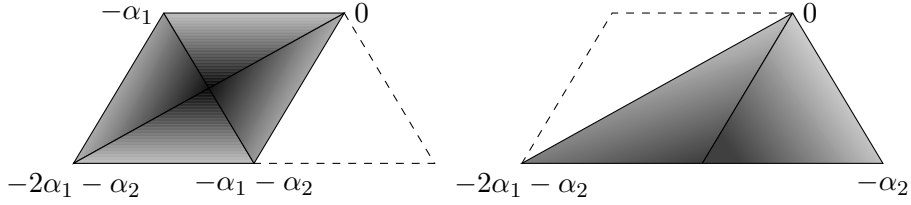

These measures $ D_\vi $ satisfy the shuffle relations under convolution (Lemma \ref{le:shuffle}) and so we obtain an algebra morphism $ D : \C[N] \rightarrow \PP $ by
$$
D(f) = \sum_{\vi} \langle e_{\vi}, f \rangle D_\vi
$$
where $ e_\vi = e_{i_1} \cdots e_{i_p} \in U(\mathfrak n) $ and where $ \langle \phantom{a}, \phantom{a} \rangle : U(\mathfrak n) \otimes \C[N] \rightarrow \C $ is the usual pairing.

Let $ \barD : \C[N] \rightarrow \C[\ftreg] $ be the map that sends a function $ f \in \C[N]_{-\nu} $, to the coefficient of $e^{-\nu} $ in the Fourier transform of $ D(f) $. Then $\barD$ is an algebra morphism, which can be described using the universal centralizer of the Lie algebra $ \mathfrak g $.  Namely, consider the map $ \ftreg \rightarrow N $ that associates to $ x \in \ftreg $  the unique $ n_x \in N $ such that $ \Ad_{n_x} (x+e) = x $, where $ e $ is a principal nilpotent.  We prove that the map $ \barD : \C[N] \rightarrow \C[\ftreg]$ is the algebra map dual to $x \mapsto n_x$ (Proposition \ref{pr:ExpanCh}).

On the other hand, associated to $ Z $, we have the \textbf{Duistermaat--Heckman measure} $ DH(Z)$, a measure (defined by Brion--Procesi \cite{BrionProcesi}, following ideas from symplectic geometry) on $ \ft_\R$ which captures the asymptotics of sections of equivariant line bundles on $ Z $ (see \S \ref{se:DH}).  This measure lives on the MV polytope of $ Z$ and its Fourier transform encodes the class of $ Z $ in the equivariant homology $H^{T^\vee}_\bullet(\Gr) $ (see Theorem \ref{th:FTDH}).  We also have the \textbf{equivariant multiplicity} $ \varep_{L_0}(Z) $ of $ Z$; this is a rational function which represents the equivariant homology class of $ Z $ in a neighbourhood of $ L_0 $.

Our second main result (Theorem \ref{th:DHequalsD}) relates the equivariant invariants of MV cycles to the above morphisms.
\begin{theorem} \label{th:DHDZintro}
Let $Z$ be an MV cycle.  After identifying $ \ft $ and $ \ft^* $, we have
$$
DH(Z) = D(b_Z) \quad\text{and}\quad \varep_{L_{-\nu}}(Z) = \barD(b_Z)
$$
\end{theorem}

We give two proofs of this Theorem.  Our first proof uses Theorem \ref{th:MVbasisIntro}(ii) and a general result of the third author \cite{Knutson} for computing Duistermaat--Heckman measures using hyperplane sections.  Our second proof uses the work of Yun--Zhu \cite{YunZhu}, which relates the equivariant homology of the affine Grassmannian to the universal centralizer.

As an application of this theorem, we prove a conjecture of Muthiah \cite{Mu}.  Let $ \lambda $ be a dominant weight and identify the zero weight space $ L(\lambda)_0 $ with $ H_{2\rho^\vee(\lambda)}(\overline{\Gr^\lambda \cap S^0_-}) $ by the geometric Satake correspondence.  Given an irreducible component $ Z \subset \overline{\Gr^\lambda \cap S^0_-} $, we can consider its equivariant multiplicity $ \varep_{L_0}(Z) \in \C(\ft^*)$ at the point $ L_0 \in \Gr$.  

Muthiah \cite{Mu} conjectured the following result and proved it when $ G =SL_n $ and $ \lambda \le n \omega_1 $.  We prove it for all $ G$ and $ \lambda $ (Theorem \ref{th:MuthiahMain}).

\begin{theorem} \label{th:Dinakar}
	The linear map $L(\lambda)_0 \rightarrow \C(\ft^*) $ defined by
$	 [Z] \mapsto \varep_{L_0}(Z)$ is  equivariant with respect to the actions of the Weyl group on both sides.
\end{theorem}

\subsection{Comparison of bases}
\label{ss:CompBases}
As we have ring homomorphisms $ D: \C[N] \rightarrow \PP $, $\barD: \C[N] \rightarrow \C[\ftreg] $, we immediately see the following result.
\begin{corollary} \label{co:intro}
Suppose that $ b_Z = c_Y $.  Then $ DH(Z) = D(c_Y) $ and $ \varep_{L_0}(Z) = \barD(c_Y) $.
\end{corollary}
This result is useful, since $ DH(Z)$ and $ \varep_{L_0}(Z)$ can be computed using the methods of computational commutative algebra.  On the other hand, $ D(c_Y) $ and $ \barD(c_Y) $ are relatively easy to compute using the following formula from Geiss--Leclerc--Schr\"oer \cite{GeissLeclercSchroer}:
$$  \langle e_{\vi}, \xi_M \rangle = \chi(F_\vi(M)) 
$$
where $ \chi(F_\vi(M))$ denotes the Euler characteristic of the variety of composition series of $M $ of type $ \vi $:
 $$F_{\vi}(M) = \bigl\{ 0 \subset M^1 \subset \cdots \subset M^m = M \bigm| M^k / M^{k-1} \cong S_{i_k} \bigr\} $$

In the appendix (written with Anne Dranowski and Calder Morton-Ferguson) we prove the following result (Theorem \ref{th:AppendixMain}).

\begin{theorem} \label{th:introAppendix}
Let $ G = SL_6 $ and let $ \nu = 2 \alpha_1 + 4 \alpha_2 + 4 \alpha_3 + 4 \alpha_4 + 2 \alpha_5$.
There exists a specific MV cycle $ Z $ of weight $ \nu$  with corresponding component $ Y $ of $ \Lambda(\nu) $, such that $ \barD(c_Y)= \barD(b_Z) + 2 \barD(b)$, where $ b $ is a vector which lies in both the MV and dual semicanonical bases.  In particular, we have $ b_Z \ne c_Y $.
\end{theorem}

This theorem suggests that we have $ c_Y = b_Z + 2b $ in $\C[N]$.  Remarkably, Geiss--Leclerc--Schr\"oer found the disagreement of the dual canonical and dual semicanonical bases at the same location.  Let $ v $ be the dual canonical basis element which corresponds to both $ b_Z $ and $ c_Y $ under the crystal isomorphisms.  Then from \cite[p.~196]{GeissLeclercSchroer}, we have $ c_Y = v + b $. Turning these equations around, we thus expect that
$$
c_Y = b_Z + 2b \quad\text{and}\quad v = b_Z + b .
$$
In \S \ref{ss:NonUniq}, we explain that similar equations indeed hold in a $D_4$ example.

In rank 2 cluster algebras, we have a similar trichotomy of bases (see for example Theorem 2.2 in \cite{Lampe}).  In this trichotomy, the MV basis seems to match Lee--Li--Zelevinsky's \cite{LLZ} greedy basis, which in turn coincides with Gross--Hacking--Keel--Kontsevich's theta basis  \cite{GHKK}, by the work of Cheung--Gross--Muller--Musiker--Rupel--Stella--Williams \cite{CGMMRSW}. The theta basis exists for any cluster algebra, in particular for $ \C[N]$.  Thus, the above calculation suggests that the MV basis for $ \C[N]$ coincides with the theta basis for this cluster algebra.

\subsection{Extra-compatibility}\label{ssec:extra}
In the Brion--Procesi definition we use, the Duistermaat--Heckman
measure is an $n\to\infty$ limit of sums of point measures.
Thus it is natural to look for this extra structure,
i.e. a finitely supported measure for each $n$, on the $ \Lambda$-module side
as well.

For any preprojective algebra module $M$ and $n \in \N $, we consider the space of (possibly degenerate) flags of submodules
$$ F_{n,\mu}(M) := \{ 0 \subseteq N_1 \subseteq \cdots \subseteq N_n \subseteq M : \sum \dimvec N_k = \mu \}$$
We prove the following result (Theorem \ref{th:limLambda}) by a direct calculation; it is an analogue of Theorem \ref{th:DHDZintro}, but more elementary.

\begin{theorem} \label{th:DxiMintro}
For any $ \Lambda$-module $M $, $D(\xi_M) $ describes the asymptotics (as $ n \to \infty$) of the function $ \mu \mapsto \chi(F_{n,\mu}(M))$.
\end{theorem}

Suppose that we have an MV cycle $Z $ and a component $ Y $ of Lusztig's nilpotent variety, such that $b_Z = c_Y $.  This implies that $ D(b_Z) = D(c_Y) $, which by Theorem \ref{th:DHDZintro} and \ref{th:DxiMintro} means that
$$
\lim_{n \to \infty} (\tau_n)_* [\Gamma(Z, \mathscr L^{\otimes n})] = \lim_{n\to \infty} (\tau_n)_* [H^\bullet(F_n(M))]
$$
where $M $ is a general point of $ Y$.  Here $ [\phantom{V}] $ denotes a class in the representation ring of $ T $, regarded as a linear combination of point masses.  The map $ \tau_n $ represents scaling by $ 1/n $ and the limits are taken in the space of distributions on $ \tR $.  Thus it is reasonable to expect equality before taking the limits. We say that $Z $ and $ Y $ are \textbf{extra-compatible}, if,
for all $ n \in \N $ and all weights $\mu$, we have
$$
\dim \Gamma(Z, \mathscr L^{\otimes n})_\mu = \chi(F_{n,-\mu}(M))
$$
For example, taking $n = 1$ gives
$$
\dim \Gamma(Z, \mathscr L)_\mu = \chi(\{ N \subseteq M : \dimvec N = -\mu \})
$$
which can be viewed as an upgrade of the equality of polytopes from Theorem \ref{th:polytopesIntro}.

We prove the following result (Theorem \ref{th:cominusExtra}) establishing extra-compatibility in a simple class of examples.
\begin{theorem}
If $ Z $ is a Schubert variety in a cominuscule flag variety and $M $ is the corresponding $\Lambda$-module, then $ Z $ and $ M $ are extra-compatible.
\end{theorem}

In the appendix, we give also some examples for $ G = SL_5, SL_6 $ of modules which satisfy the extra-compatibility condition for $ n = 1, 2 $.

\subsection{A general conjecture}\label{ssec:expectations}
Most $\Lambda$-modules $ M $ are not extra-compatibly paired with any
MV cycle; for example if $ G = SL_3$ and $M$ is the sum of the two simple
$\Lambda$-modules, then the rhombus $\Pol(M)$ is the
union of two MV polytopes, each a triangle.
However, for any $ \Lambda$-module $ M$, we expect that
there will be a corresponding coherent sheaf on the affine
Grassmannian, supported on a union of MV cycles.
To state our precise expectation, we introduce the following space
whose Euler characteristic coincides with $ F_{n,\mu}(M) $.
$$
\mathbb G_\mu(M[t]/t^n) := \bigl\{ N \subseteq M \otimes \C[t]/t^n \bigm| N \text{ is a $\Lambda\otimes \C[t]$-submodule, }\, \dimvec N = \mu \bigr\},
$$

We conjecture the following result (see \S \ref{se:conj} for more precise motivation).
\begin{conjecture} \label{conj:intro}
For any preprojective algebra module $ M $ of dimension vector $ \nu $, there exists a coherent sheaf $ \mathcal F_M $ supported on $ \overline{S_+^0 \cap S_-^{-\nu}} $ such that
$$
\Gamma(\Gr, \mathcal F_M \otimes \mathcal O(n)) \cong H^\bullet(\mathbb G(M[t]/t^n))
$$
as $T^\vee$-representations.
\end{conjecture}
For example, if $Z $ and $ M $ are extra-compatible, then we can take $ \mathcal F_M = \mathcal O_Z $.

This conjecture has two important relations with recent developments.  First, an earlier version of this conjecture motivated a number of recent works by Mirkovi\'c and his coauthors \cite{M,MYZ} on the subject of local spaces.

Second, quiver varieties and affine Grassmannian slices are related using the theory of symplectic duality as introduced by Braden--Licata--Proudfoot--Webster \cite{BLPW}.  Recently, Braverman--Finkelberg--Nakajima \cite{BFN} proved that affine Grassmannian slices are Coulomb branches associated to quiver gauge theories.  Using this result, in a forthcoming paper, Hilburn, Weekes and the second author \cite{HKW} will prove Conjecture \ref{conj:intro} for those $ M $ which come from a quiver path algebra.

More generally, the relationship between MV cycles and preprojective algebra modules studied in this paper admits a generalization to arbitrary symplectic dual pairs.  It would be very interesting to understand how the results presented here generalize to that setting.

\subsection*{Acknowledgements}
We would like to thank Ivan Mirkovi\'c for many valuable discussions about geometric Satake over the years and in particular for his interest in our conjectured relationship between preprojective modules and MV cycles. J.K. would like to thank Dennis Gaitsgory for helping him to understand the multiplication of MV cycles many years ago.  We also thank David Anderson, Alexander Braverman, Michel Brion, Pavel Etingof, Michael Finkelberg, Daniel Juteau, Bernard Leclerc, Ivan Losev, Dinakar Muthiah, Ben Webster, Alex Weekes, and Xinwen Zhu for helpful conversations, and Anne Dranowski and Calder Morton-Ferguson for their work on the appendix and for carefully reading the main text. Finally, we thank the referees for their careful reading of the manuscript and the additional references which they kindly provided.

P.B. was partially supported by the ANR, project ANR-15-CE40-0012.  J.K. was supported by NSERC, by the Sloan Foundation, and the Institut Henri Poincar\'e.  A.K. was supported by NSF grant DMS–1700372.

\begin{table}[htbp]
	\caption{Dramatis person\ae}
		\begin{tabular}{c p{10cm} l}
			$G$ & a simple simply-connected complex algebraic group & \S \ref{ss:Notation} \\
			$N, B, T, W$ & usual data associated to $G$ &  \S \ref{ss:Notation} \\
			$e_i \in \lie{n}$ & Chevalley generators &  \S \ref{ss:Notation} \\
			$\{\alpha_i\},\{\alpha_i^\vee\},I$ &the simple roots and coroots,
			with common index set $I$
			& \S \ref{ss:Notation} \\
			$P \supset P_+$& the weight lattice and dominant weights of $G$ & \S \ref{ss:Notation} \\
			$Q \supset Q_+$& the root lattice and $\N$-span of the positive roots
			& \S \ref{ss:Notation} \\
			$\wt,\varepsilon_i,\tilde e_i,\phi_i,\tilde f_i$ & crystal data & \S\ref{ss:PerfBas} \\
			$B(\infty)$ & the bicrystal of $ \C[N]$ & \S \ref{ss:UniqCrys} \\
			$L(\lambda)$ & the irreducible $G$-representation of highest weight $\lambda\in P_+$ & \S\ref{ss:BasesRep} \\
			$v_\lambda,\ v_\lambda^*$ & a h.w.\ vector and a linear form on $L(\lambda)\,$ s.t.\ $v_\lambda^*(v_\lambda)=1$ & \S\ref{ss:BasesRep} \\
			$\Psi_\lambda$ & an embedding $L(\lambda)\to\ON$ taking $ v \mapsto v_\lambda^*(\bm?v) $ & \S\ref{ss:BasesRep} \\
			$G^\vee, \Gr$ & the Langlands dual group and its affine Grassmannian $ \Gr $ & \S \ref{ss:GeomSatEquiv} \\
			$t^\mu, L_\mu$ & the points in $ G^\vee(\K)$, $Gr$ defined by the $G^\vee$-coweight $\mu \in P$ & \S\ref{ss:GeomSatEquiv} \\
			$\Gr^\lambda, \ S^\mu_\pm$ & the spherical and semi-infinite orbits in $\Gr$ & \S\ref{ss:GeomSatEquiv} \\
			$F = \oplus_\mu F_\mu$
			& the fiber functor $H^\bullet(\Gr,\bm?)$ and the weight functors & \S\ref{ss:GeomSatEquiv} \\
			$\Upsilon$
			& the embedding $\Gr \into \mathbb P(V(\Lambda_0))$ & \S\ref{ss:PrincNilp} \\
			$i(Z,X\cdot V)$ & the multiplicity of $Z$ in the intersection $X\cdot V$ &
			\S \ref{ss:SoMoNo} \\
			$Z$ & a typical MV cycle & \S\ref{ss:MVCycles} \\
			$\IC_\lambda$ & the intersection cohomology sheaf of $\overline{\Gr^\lambda}$
			& \S\ref{ss:MVCycles} \\
			$\mathcal Z(\lambda)$ & the set of MV cycles of type $ \lambda $ & \S\ref{ss:MVCycles} \\
			$\mathcal Z(\infty)$ & the set of stable MV cycles & \S\ref{ss:CohMVBases} \\
			$\{b_Z\}$ & the MV basis of $\C[N]$, indexed by stable MV cycles & \S \ref{ss:MVBasisBiperf} \\
			$ n_x $ & for $ x \in \ftreg $, the element of $ N $ such that $ \Ad_{n_x}(x) = e + x $ & \S \ref{ss:EltsNx} \\
			$ \barD $ & the corresponding algebra morphism $ \C[N] \rightarrow \C[\ftreg] $ & \S \ref{ss:EltsNx} \\
			$\Seq(\nu)$ & the set of sequences of simple roots with sum $ \nu $
			& \S\ref{ssec:sequences} \\
			$\barD_\vi$ & a rational function associated to a sequence $ \vi $ 
			&    \S \ref{ssec:sequences}  \\
			$\PP$ & an algebra of distributions on $ \tR$
			& \S\ref{ssec:measures} \\
			$ \pi_\vi : \R^{p+1} \rightarrow \ftreg $ & a linear map defined using the partial sums of a sequence $ \vi $ & \S\ref{ssec:measures} \\
			$D_\vi $ & a measure associated to a sequence $ \vi $  
			& \S\ref{ssec:measures} \\
			$D$ &
			an algebra map $ \ON  \longrightarrow \PP$ &
			 \S\ref{ssec:measures} \\
			$C$ & the universal centralizer space $ \{(b,x) : \Ad_b(\dot e+x) = \dot e+x \} $ & \S\ref{se:univcent}\\
			$\psi:C\to N$ & a morphism given by $(tn,x) \mapsto n$  & \S\ref{se:univcent}\\
			$ \tau_n $ & scaling of $ \tR $ by $ 1/n $ & \S \ref{ssec:DH} \\
			$ S \subset \C[\ft]$ & the multiplicative set generated by
			$ \mu \in P  \setminus \{0\}$ & \S\ref{ssec:FTDH} \\
			$\Lambda$ & the preprojective algebra & \S\ref{ssec:preproj} \\
			\ \ $\dimvec$ & the dimension vector of a $\Lambda$-module &\S\ref{ssec:preproj} \\
			$\Lambda(\nu)$\ \  & Lusztig's nilpotent variety &\S\ref{ssec:preproj} \\
			$ F_\vi(M) $ & the variety of composition series of type $ \vi $ & \S\ref{ssec:dualsemi} \\
			$\xi_M $ & an element of $\C[N]_{-\dimvec M}$
			associated to a $\Lambda$-module $M$ & \S\ref{ssec:dualsemi} \\
			$Y$ & a typical component of $ \Lambda(\nu)$ & \S\ref{ssec:dualsemi} \\
			$c_Y $ & dual semicanonical basis element
			associated to $Y$& \S\ref{ssec:dualsemi} \\
			$F_{n,\mu}(M)$ & the space of length $n +1$ flags in $M$ of
			total dimension $\mu$
			& \S\ref{ssec:measmod} \\
			$\mathbb G_\mu(M[t]/t^n)$ & the space of dimension $\mu$ submodules in
			$M \tensor \C[t]/t^n$
			&  \S \ref{se:conj} 
		\end{tabular}
	\label{tab:notation}
\end{table}

\part{Biperfect bases}

\section{General background}
\label{se:GenBack}

\subsection{Notation}
\label{ss:Notation}
Let $G$ be a simple simply-connected complex algebraic group. Let $B$
be a Borel subgroup with unipotent radical $N $ and let $T$ be a maximal
torus of $B$. Let $\mathfrak g, \mathfrak b, \mathfrak n, \mathfrak t$ denote their Lie algebras.

Let $P$ denote the character lattice of $T$, let $P_+$ denote the set
of dominant weights, and let $Q\subset P$ denote the root lattice. Let
$Q_+\subset Q$ denote the $\mathbb Z_{\geq0}$-span of the positive roots.
We define the \defn{dominance order} on $P$ by declaring that
$\lambda\geq\mu$ if $\lambda-\mu\in Q_+$.

Let $\{\alpha_i\}_{i\in I}$ denote the set of simple roots and let
$\{\alpha_i^\vee\}_{i\in I}$ denote the set of simple coroots.
The Cartan matrix of $G$ is $A=(a_{i,j})_{i,j\in I}$ where
$a_{i,j}=\langle\alpha_i^\vee,\alpha_j\rangle$. We define $\rho\in P$
to be the half-sum of the positive roots, and we define $\rho^\vee\in
\Hom_{\mathbb Z}(P,\frac12\mathbb Z)$ to be the half-sum of the positive coroots.

Let $W=N_G(T)/T$ be the Weyl group, generated by the simple reflections
$s_i$ for $i\in I$. For $i\in I$, we set $h_i=\alpha_i^\vee$ and we choose
root vectors $e_i$ and $f_i$ in $\mathfrak g$ of weights $\alpha_i$ and
$-\alpha_i$, respectively, normalized so that $[e_i,f_i]=h_i$. Then the
element $\overline s_i=\exp(-e_i)\exp(f_i)\exp(-e_i)$ is a lift of $s_i$ in
$N_G(T)$. These elements $\overline s_i$ satisfy the braid relations
\cite[Proposition~3]{Tits}, which allows us to lift any $w\in W$ to an element $\overline w\in N_G(T)$.

The enveloping algebra $U(\mathfrak n)$ of $\mathfrak n$ is generated
by the elements $e_i$; it is graded by~$Q_+$, with $\deg e_i=\alpha_i$.
As is customary, for any natural number $n$, we denote the $n$-th divided
power of $e_i$ by $e_i^{(n)}$.

The torus $T$ acts by conjugation on $N$, which endows the algebra $\ON$
of regular functions on $N$ with a weight grading
$$\ON=\bigoplus_{\nu\in Q_+}\ON_{-\nu}.$$
The group $N$ and its Lie algebra $\mathfrak n$ act on both
sides of $\ON$; our convention is that $f\cdot n=f(n\bm?)$ and
$n\cdot f=f(\bm?n)$ for each $(n,f)\in N\times\ON$.
Denoting by $1_N$ the unit element in $N$, we have $(a\cdot f)(1_N)=
(f\cdot a)(1_N)$ for each $(a,f)\in U(\mathfrak n)\times\ON$, and
the map $\langle a,f \rangle := (a\cdot f)(1_N)$ defines a perfect pairing
$U(\mathfrak n)\times\ON\to\mathbb C$. In particular the vector
space $\ON_{-\nu}$ is linearly isomorphic to the dual of
$U(\mathfrak n)_\nu$ for each $\nu\in Q_+$.

\subsection{Crystals} \label{ss:Crystals}
Following for instance~\cite[sect.~3]{KashiwaraSaito}, we recall that
a $G$-crystal is a set $B\not\ni 0$ endowed with maps
\begin{alignat*}2
\wt:B\to P,\quad
&\varepsilon_i:B\to\mathbb Z\cup\{-\infty\},\quad
&&\tilde e_i:B\to B\cup\{0\},\\
&\varphi_i:B\to\mathbb Z\cup\{-\infty\},\quad
&&\tilde f_i:B\to B\cup\{0\}
\end{alignat*}
for $i\in I$, satisfying the following axioms:
\begin{itemize}
\item
For each $b\in B$ and $i\in I$, $\varphi_i(b)=\langle h_i,\wt(b)\rangle+
\varepsilon_i(b)$.
\item
For each $b,b'\in B$ and $i\in I$, we have $b=\tilde e_ib'\
\Longleftrightarrow\ \tilde f_ib=b'$.
\item
For each $b\in B$ and $i\in I$ such that $\tilde e_ib\neq0$, we have
$\wt(\tilde e_ib)=\wt(b)+\alpha_i$, $\varepsilon_i(\tilde e_ib)=
\varepsilon_i(b)-1$ and $\varphi_i(\tilde e_ib)=
\varphi_i(b)+1$.
\item
For each $b\in B$ and $i\in I$, if $\varphi_i(b)=\varepsilon_i(b)=-\infty$,
then $\tilde e_ib=\tilde f_ib=0$.
\end{itemize}

A crystal $B$ is said to be \defn{upper semi-normal} if, for each $b\in B$ and
$i\in I$, there exists $ n \in\mathbb Z_{\geq0} $ such that $\tilde e_i^nb = 0$ and
$$\varepsilon_i(b)=\max\{n\in\mathbb Z_{\geq0}\mid\tilde e_i^nb\neq0\}.$$
A crystal $B$ is said to be \defn{semi-normal} if additionally,
for each $b\in B$ and each $i\in I$,  there exists $ n \in\mathbb Z_{\geq0} $ such that $\tilde f_i^nb = 0$ and
$$\varphi_i(b)=\max\{n\in\mathbb Z_{\geq0}\mid\tilde f_i^nb\neq0\}.$$

All the crystals that we consider in this paper are upper semi-normal.
The maps $\varepsilon_i$ and $\varphi_i$ are then determined by the rest
of the structure.

\subsection{Perfect bases}
\label{ss:PerfBas}
We look for bases of $\ON$ that enjoy a form of compatibility with
the left and right actions of $ \{ e_i \} \subset \mathfrak n$. The following definition
matches Berenstein and Kazhdan's one (\cite{BerensteinKazhdan},
Definition~5.30), with the addition of a specific normalization.

\begin{definition}
A linear basis $B$ of $\ON$ is \textbf{perfect} if it is endowed with
an upper semi-normal crystal structure such that:
\begin{itemize}
\item
The constant function equal to $1$ belongs to $B$.
\item
Each $b$ in $B$ is homogeneous of degree $\wt(b)$ with respect to the
weight grading
$$\ON=\bigoplus_{\nu\in Q_+}\ON_{-\nu}.$$
\item
For each $i\in I$ and $b\in B$, the expansion of $e_i\cdot b$ in the basis $B$
has the form
$$e_i\cdot b=\varepsilon_i(b)\;\tilde e_ib+\sum_{\substack{b'\in B\\
\varepsilon_i(b')<\varepsilon_i(b)-1}}a_{b'}\,b'$$
with $a_{b'}\in\mathbb C$.
\end{itemize}
\end{definition}

It follows from the definition that if a linear basis $B$ of $\ON$ is
perfect, then for each $i\in I$ and each $b\in B$ we have
$$n=\varepsilon_i(b)\ \Longrightarrow\ e_i^{(n)}\cdot b=\tilde e_i^nb\
\text{ and }\ e_i^{n+1}\cdot b=0.$$

For $i\in I$ and $n\in\mathbb Z_{\geq0}$, let us define
$$K_{i,n} := \{f\in\ON\mid e_i^{n+1}\cdot f=0\}.$$
Using the fact just above, we easily check that for any perfect basis
$B$ of $\ON$, we have
$$B\cap K_{i,n}=\{b\in B\mid\varepsilon_i(b)\leq n\}$$
and moreover this set is a basis of $K_{i,n}$.

To take into account the right action of $\mathfrak n$ on $\ON$, we
now introduce biperfect bases.

\begin{definition}
A linear basis $B$ of $\ON$ is \textbf{biperfect}\footnote{This is the  same as the notion of ``basis of dual canonical type'' from \cite{McNamara}, except that we have specialized $ q= 1 $, and we do not require that $ B$ be invariant under the involution $*$ discussed in Remark \ref{re:StarInv}.}  if it is perfect
and if it is endowed with a second upper semi-normal crystal structure
$(\wt,\varepsilon_i^*,\varphi_i^*,\tilde e_i^*,\tilde f_i^*)$ which
shares the same weight map as the first crystal structure and such that for each
$i\in I$ and $b\in B$,
$$b\cdot e_i=\varepsilon_i^*(b)\;\tilde e_i^*b+\sum_{\substack{b'\in B\\
\varepsilon_i^*(b')<\varepsilon_i^*(b)-1}}a^*_{b'}b'$$
with $a_{b'}\in\mathbb C$.
\end{definition}

We will refer to the data consisting of these two crystal structures
on $B$ as the \defn{bicrystal structure of $B$}. If $B$ is a biperfect basis
of $\ON$, then for each $i\in I$ and each $n\in\mathbb Z_{\geq0}$, the set
$\{b\in B\mid\varepsilon_i^*(b)\leq n\}$ is a basis of
$$K_{i,n}^*=\{f\in\ON\mid f\cdot e_i^{n+1}=0\}.$$

The algebra $\ON$ does have biperfect bases, but the explicit constructions
of such bases rely on geometric constructions or on categorification methods.
The first example of a biperfect basis is Lusztig's canonical basis, after
specialization at $q=1$ and then dualization; in other words, Kashiwara's
upper global basis, specialized at $q=1$ (see for instance
\cite[Theorems~1.6~and~7.5]{Lusztig90}, and \cite[\S 5.3]{Kashiwara93a}).
Another example, in the simply-laced case, is the dual of Lusztig's
semicanonical basis (the compatibility of the dual semicanonical basis
with the subspaces $K_{i,n}$ and $K_{i,n}^*$ is established in
\cite[\S 3]{Lusztig00}). The basis arising using the categorification
by representations of KLR algebras is also biperfect; this fact is
transparent from Khovanov and Lauda's first paper \cite{KhovanovLauda}
and is given a detailed proof in \cite{McNamara}.
Lastly, the geometric Satake correspondence also gives rise to a biperfect
basis of $\ON$, as we shall see in \S\ref{se:MVBasisRepr}.

In types $A_1$, $A_2$ and $A_3$, we not only have existence, but also
uniqueness and explicit formulas.

\begin{example}
Suppose $G=\SL_3(\mathbb C)$, with the standard choice for $B$, $T$ and
$N$. Then $\ON=\mathbb C[x,y,z]$ where $x$, $y$ and $z$ are the three
matrix entries of an upper unitriangular matrix
$$\begin{pmatrix}1&x&z\\0&1&y\\0&0&1\end{pmatrix}\in N.$$
The unique biperfect basis of $\ON$ is
$$B=\{x^az^b(xy-z)^c\mid(a,b,c)\in\mathbb Z_{\geq0}^3\}\cup
\{y^az^b(xy-z)^c\mid(a,b,c)\in\mathbb Z_{\geq0}^3\}.$$
The action from the left of the Chevalley generators is given by
$$e_1=\frac\partial{\partial x}\ \text{ and }\
e_2=\frac\partial{\partial y}+x\frac\partial{\partial z}.$$
One can check that in $B$, the operators $e_1$ and $e_2$ act with
coefficients in $\mathbb Z_{\geq0}$ and that the structure constants of
the multiplication belong to $\mathbb Z_{\geq0}$.
\end{example}

For the explicit formulas in type $A_3$, we refer to the paper
\cite{BerensteinZelevinsky93} by Berenstein and Zelevinsky, which
was the starting point of the theory of cluster algebras.

\subsection{Uniqueness of crystal}
\label{ss:UniqCrys}
Berenstein and Kazhdan proved that up to isomorphism, the crystal of
a perfect basis of $\ON$ is independent of the choice of the basis
(\cite{BerensteinKazhdan}, Theorem~5.37). The same is true for biperfect
bases.

\begin{theorem}
\label{th:UniqCrys}
Let $B$ and $C$ be two biperfect bases of $\ON$. Then there is a unique
bijection $B\cong C$ that respects the bicrystal structure.
\end{theorem}
\begin{proof}
We study the properties of the transition matrix $M=(m_{b,c})$ between
the two bases, defined by the equation
\begin{equation*}
c=\sum_{b\in B}m_{b,c}\,b
\tag*{($*$)}
\end{equation*}
for each $c\in C$.

Fix $i\in I$, take $c\in C$, and set $n=\max\{\varepsilon_i(b)\mid
b\in B\text{ such that }m_{b,c}\neq0\}$. Then each $b$ occurring in
the right-hand side of ($*$) belongs to $K_{i,n}$, and therefore
$c\in K_{i,n}$, that is, $\varepsilon_i(c)\leq n$. Applying $e_i^{(n)}$
to ($*$), we get
$$e_i^{(n)}\cdot c\ =\sum_{\substack{b\in B\\m_{b,c}\neq0\text{ and }
\varepsilon_i(b)=n}}m_{b,c}\;e_i^{(n)}\cdot b.$$
The terms $e_i^{(n)}\cdot b=\tilde e_i^nb$ that survive in the right-hand
side belong to $B$, hence are linearly independent, which implies that
the left-hand side is not zero, and therefore that $\varepsilon_i(c)=n$.

Let us define
$$B_{i,n}=\{b\in B\mid\varepsilon_i(b)=n\}\ \text{ and }\
C_{i,n}=\{c\in C\mid\varepsilon_i(c)=n\}.$$
The above analysis shows that the matrix $M$ is block upper triangular
with respect to the decompositions
$$B=B_{i,0}\cup B_{i,1}\cup B_{i,2}\cup\cdots\ \text{ and }\
C=C_{i,0}\cup C_{i,1}\cup C_{i,2}\cup\cdots$$
of the rows and the columns, and that the diagonal blocks of $M$ are
equal under the bijections
$$\tilde e_i^n:B_{i,n}\to B_{i,0}\ \text{ and }\
\tilde e_i^n:C_{i,n}\to C_{i,0}.$$

We now replace the index $i\in I$ by a sequence $\mathbf i=
(i_1,i_2,i_3,\ldots)$, in which each element of $I$ appears infinitely
many times. To an element $b\in B$ we associate its string parameters
in direction $\mathbf i$, namely the sequence $\mathbf n_{\mathbf i}(b)=
(n_1,n_2,n_3,\ldots)$ where
$$n_1=\varepsilon_{i_1}(b),\quad
n_2=\varepsilon_{i_2}(\tilde e_{i_1}^{n_1}b),\quad
n_3=\varepsilon_{i_3}(\tilde e_{i_2}^{n_2}\tilde e_{i_1}^{n_1}b),\ \ldots$$
The weights of the sequence of elements
$$b,\quad
\tilde e_{i_1}^{n_1}b,\quad
\tilde e_{i_2}^{n_2}\tilde e_{i_1}^{n_1}b,\quad
\tilde e_{i_3}^{n_3}\tilde e_{i_2}^{n_2}\tilde e_{i_1}^{n_1}b,\ \ldots$$
increase in $(-Q_+)$ with respect to the dominance order, so this
sequence becomes eventually constant, and its final value is in
$$\bigcap_{i\in I}B_{i,0}=B\cap\left(\;\bigcap_{i\in I}K_{i,0}\right)=\{1\}.$$
It follows that the map $b\mapsto\mathbf n_{\mathbf i}(b)$ is
injective, and we can therefore transfer the lexicographic order on
string parameters to a total order on $B$. Similarly, we define the
string parameters in direction $\mathbf i$ of an element of $C$ and
totally order $C$ accordingly.

Iterating our first argument, we see that the matrix $M$ is now upper
triangular, with all diagonal elements equal --- and in fact equal to
$1$ because of the normalization condition that $1$ belongs to both
$B$ and $C$. In particular, we obtain a bijection between $B$ and $C$
that preserves string parameters in direction $\mathbf i$. This certainly
means that this bijection preserves the weight map and the crystal
operations $\varepsilon_{i_1}$, $\tilde e_{i_1}$ and $\tilde f_{i_1}$.

Now we can change $\mathbf i$ and thus replace $i_1$ by any element of
$I$. The bijection does not change, because given two sets $B$ and $C$
and a matrix whose elements are indexed by $B\times C$, there is at
most one bijection $B\cong C$ such that there exists a total order on
$B$ (and hence $C$) making $M$ upper unitriangular. In other words,
if $M$ is an upper unitriangular matrix, and $P$ and $Q$ are permutation
matrices, and $PMQ^{-1}$ is upper triangular, then $P=Q$. (For this
last assertion, it is necessary to assume that the matrices have
finitely many rows and columns. This condition does not hold in our
situation, but we can reduce to it by restricting to weight subspaces.)

We now have proved the existence of a bijection $B\cong C$ that respects
the crystal structure $(\wt,\varepsilon_i,\varphi_i,\tilde e_i,\tilde f_i)$.
But we can similarly construct a bijection $B\cong C$ that respects the
starred crystal structure $(\wt,\varepsilon_i^*,\varphi_i^*,\tilde e_i^*,
\tilde f_i^*)$, and the argument in the previous paragraph shows that the
two bijections necessarily coincide.
\end{proof}

The abstract bicrystal structure shared by all biperfect bases of
$\ON$ is denoted by $B(\infty)$.

\subsection{Bases in representations}
\label{ss:BasesRep}
Given a dominant weight $\lambda$, we denote the simple $G$-module of
highest weight $\lambda$ by $L(\lambda)$. We fix a preferred choice of
a highest weight vector $v_\lambda$ in $L(\lambda)$.

\begin{definition}
A linear basis $B_\lambda$ of $L(\lambda)$ is \textbf{perfect} if it
is endowed with an upper semi-normal crystal structure such that:
\begin{itemize}
\item
The highest weight vector $v_\lambda$ belongs to $B_\lambda$.
\item
Each $b$ in $B_\lambda$ is homogeneous of degree $\wt(b)$ with respect to the
weight grading of $L(\lambda)$.
\item
For each $i\in I$ and $b\in B_\lambda$, the expansion of $e_i\cdot b$ in the basis $B_\lambda$
has the form
$$e_i\cdot b=\varepsilon_i(b)\;\tilde e_ib+\sum_{\substack{b'\in B_\lambda\\
\varepsilon_i(b')<\varepsilon_i(b)-1}}a_{b'}\,b'$$
with $a_{b'}\in\mathbb C$.
\end{itemize}
\end{definition}

\begin{remark} \label{rem:good}
\begin{enumerate}
\item 
Any perfect basis is a good basis in the sense of Berenstein and
Zelevinsky~\cite{BerensteinZelevinsky88}. It follows that any perfect basis $ B_\lambda $ of $ L(\lambda) $ restricts to bases for tensor product multiplicity spaces. Specifically,
given dominant weights $\mu$, $\nu$, the set
$$\{b\in B_\lambda\mid
\wt(b)=\nu-\mu\ \text{ and }\ \forall i\in I,\ \varepsilon_i(b)
\leq\langle h_i,\mu\rangle\}
$$
forms a basis for $ \Hom(L(\nu), L(\lambda) \otimes L(\mu)) $ where we use the inclusion
$$
\Hom(L(\nu), L(\lambda) \otimes L(\mu)) \hookrightarrow L(\lambda) \quad \phi \mapsto (1 \otimes v_\mu^*)(\phi(v_\nu))
$$
(see below for the definition of $ v_\mu^*$).
\item
Let $B_\lambda$ be a perfect basis of $L(\lambda)$ and let
$B_\lambda^\star$ be its dual basis with respect to a contravariant
form on $L(\lambda)$. Then for any Demazure module $W\subset L(\lambda)$,
the set $B_\lambda^\star\cap W$ is a basis of $W$. In fact,
Kashiwara's proof of the same result for the global crystal basis
(\cite[\S3.2]{Kashiwara93b}) only uses the axioms of a perfect basis
(up to duality). In the case of the semicanonical basis, this property
was observed by Savage (\cite[Theorem~7.1]{Savage}).
\end{enumerate}
\end{remark}
	
Let $v_\lambda^*:L(\lambda)\to\mathbb C$ be the linear form such that
$v_\lambda^*(v_\lambda)=1$ and $v_\lambda^*(v)=0$ for any weight vector
$v\in L(\lambda)$ of weight other than $\lambda$.
We define an $N$-equivariant map
$$\Psi_\lambda:L(\lambda)\to\ON$$
by $\Psi_\lambda(v)=v_\lambda^*(\bm? v)$. As a matter of fact,
$\Psi_\lambda$ is injective and that its image is
$$\im\Psi_\lambda=\bigcap_{i\in I}K_{i,\langle h_i,\lambda\rangle}^*$$
(one can deduce this from \cite[Theorem~21.4]{Humphreys72}
with the help of a contravariant form on $L(\lambda)$ ;
see also \cite[Proposition~5.1]{BerensteinZelevinsky96}).
It follows that if $B$ is a biperfect basis of $\ON$, then
$$B\cap(\im\Psi_\lambda)=\{b\in B\mid\forall i\in I,\ \varepsilon_i^*(b)
\leq\langle h_i,\lambda\rangle\}$$
is a basis of $\im\Psi_\lambda$.

\begin{proposition}
\label{pr:PerfBasTrf}
Let $B$ be a biperfect basis of $\ON$, let $\lambda\in P_+$, and
let $B_\lambda=\Psi_\lambda^{-1}(B)$. Then $B_\lambda$ inherits from
$B$ the structure of an upper semi-normal crystal, the weight map being
shifted by $\lambda$, and $B_\lambda$ is a perfect basis of $L(\lambda)$.
\end{proposition}
\begin{proof}
For each $b\in B_\lambda$, we set
$$\wt(b)=\wt(\Psi_\lambda(b))+\lambda,\quad
\varepsilon_i(b)=\varepsilon_i(\Psi_\lambda(b))\quad\text{and}\quad
\varphi_i(b)=\langle h_i,\wt(b)\rangle+\varepsilon_i(b).$$
If $\varepsilon_i(b)=0$, then we set $\tilde e_ib=0$. Otherwise,
$\tilde e_i(\Psi_\lambda(b))$ appears with a nonzero coefficient
in the expansion of $e_i\cdot\Psi_\lambda(b)$ in the basis $B$. Now
$e_i\cdot\Psi_\lambda(b)$ belongs to $(\im\Psi_\lambda)$, and this
subspace is spanned by the elements of $B$ that it contains. We conclude
that $\tilde e_i(\Psi_\lambda(b))\in B\cap(\im\Psi_\lambda)$, and
therefore we can define $\tilde e_ib\in B_\lambda$ by
$$\Psi_\lambda(\tilde e_ib)=\tilde e_i(\Psi_\lambda(b)).$$
The fact that $B$ is upper semi-normal then implies that $B_\lambda$ is
upper semi-normal. Lastly, we define $\tilde f_ib\in B_\lambda\cup\{0\}$
so that
$$\Psi_\lambda(\tilde f_ib)=\begin{cases}
\tilde f_i\Psi_\lambda(b)&\text{if $\tilde f_i(\Psi_\lambda(b))\in
(\im\Psi_\lambda)$,}\\
0&\text{otherwise.}
\end{cases}$$
\end{proof}

From Remark \ref{rem:good}(i), we immediately deduce the following corollary, which can be regarded as a generalization (from the canonical basis to arbitrary biperfect bases) of \cite[Corollary 3.4]{BZ01}.
\begin{corollary}
	Let $ B $ be a biperfect basis of $ \ON $.  For any $ \lambda, \mu, \nu \in P_+ $, the set
	$$
\{b\in B(\infty)\mid
\wt(b)=\nu-\mu - \lambda\ \text{ and }\ \forall i\in I,\ \varepsilon_i(b)
\leq\langle h_i,\mu\rangle, \ \varepsilon_i^*(b) \leq \langle h_i,\lambda\rangle \} 
$$
restricts (under $ \Psi_\lambda$  and the inclusion from Remark \ref{rem:good}(i)) to a basis for $ \Hom(L(\nu), L(\lambda) \otimes L(\mu)) $.
\end{corollary}

The proof of the following lemma relies on elementary
$\mathfrak{sl}_2$-theory and is left to the reader.

\begin{lemma}
\label{le:PerfBasRep}
Let $B_\lambda$ be a perfect basis of $L(\lambda)$. Then the crystal
$B_\lambda$ is semi-normal, and for each $i\in I$ and $b\in B_\lambda$, the
expansion of $f_i\cdot b$ in the basis $B_\lambda$ has the form
$$f_i\cdot b=\varphi_i(b)\;\tilde f_ib+\sum_{\substack{b'\in B_\lambda\\
\varphi_i(b')<\varphi_i(b)-1}}a_{b'}\,b'$$
with $a_{b'}\in\mathbb C$.
\end{lemma}

\begin{remark} \label{re:flag}
Let $\lambda\in P_+$. For each $w\in W$, the weight space
$L(\lambda)_{w\lambda}$ is one-dimensional.  We choose a basis vector for this weight space by defining $v_{w\lambda}=\overline wv_\lambda$. These elements can also
be defined by induction on the length of $w$: if $s_iw>w$, then
$$v_{s_iw\lambda}=f_i^{(n)}\cdot v_{w\lambda}$$
where $n=\langle h_i,w\lambda\rangle$. Using Lemma~\ref{le:PerfBasRep},
we see that these elements $v_{w\lambda}$ belong to each perfect basis
of $L(\lambda)$. Abusing slightly the standard terminology, we will call
\textbf{flag minors} the functions $\Psi_\lambda(w_{w\lambda})$; they
belong to each biperfect basis of $\ON$. (When $\lambda$ is minuscule,
they are the restrictions to $N$ of the \textbf{flag minors} from
\cite{BerensteinFominZelevinsky,BerensteinZelevinsky97}.)
\end{remark}

We observe that for any dominant weights $\lambda$ and $\mu$,
we have $\im\Psi_\lambda\subset\im\Psi_{\lambda+\mu}$. This
motivates the following definition.

\begin{definition}
\label{de:CohFamBas}
A \defn{coherent family of bases} is the datum of a basis $B_\lambda$ of
$L(\lambda)$ for each dominant weight $\lambda\in P_+$ such that
$$\Psi_\lambda(B_\lambda)\subset\Psi_{\lambda+\mu}(B_{\lambda+\mu})$$
for all $\lambda,\mu\in P_+$.
\end{definition}

A biperfect basis $B$ gives rise to a coherent family of perfect bases,
namely the datum of all the bases $\Psi_\lambda^{-1}(B)$. Conversely,
given a coherent family of perfect bases $(B_\lambda)$, the union
$$\bigcup_{\lambda\in P_+}\Psi_\lambda(B_\lambda)$$
is a perfect basis of $\ON$. We note that the crystal structures
automatically match, in the sense of Proposition~\ref{pr:PerfBasTrf}.

\subsection{Multiplication}
We can easily describe multiplication in $ \C[N] $ using the maps
$ \Psi_\lambda$.  First, recall that there is a unique $G$-equivariant map
$ m_{\lambda \mu} : L(\lambda) \otimes L(\mu) \rightarrow L(\lambda + \mu) $
which takes $ v_\lambda \otimes v_\mu $ to $ v_{\lambda + \mu} $.

The following result follows immediately from the definition of $ \Psi_\lambda $.
\begin{proposition} \label{pr:mult}
We have the commutativity $ m \circ (\Psi_\lambda \otimes \Psi_\mu) = \Psi_{\lambda+ \mu} \circ m_{\lambda \mu} $, where $ m : \C[N] \otimes \C[N] \rightarrow \C[N] $ is the multiplication map.
\end{proposition}

Thus if we have a coherent family of bases $ (B_\lambda)_{\lambda \in P_+} $ and the matrix of the maps $ m_{\lambda \mu }$ in these bases, then we will have a basis for $ \C[N] $ and the structure constants for multiplication in this basis.

\subsection{Non-uniqueness}
\label{ss:NonUniq}
In \cite{GeissLeclercSchroer}, Geiss, Leclerc and Schr\"oer present
examples where the canonical and the semicanonical bases are different.
Thus, biperfect bases are not unique in general. Let us have a closer
look at the simplest example (see \S 19.1 in \textit{loc.\ cit.}).
Here $G$ is of type $D_4$. We enumerate the vertices in the Dynkin
diagram as customary (the trivalent vertex has label $2$) and
we set $\lambda=\alpha_1+2\alpha_2+\alpha_3+\alpha_4$,
the highest root.

We consider two biperfect bases $C$ and $C'$ of $\ON$. Abusing the
notation, for each $b\in B(\infty)$, we denote by $C(b)$ and $C'(b)$
the elements of $C$ and $C'$ indexed by $b$. We focus on the subspace
$\ON_{-2\lambda}\cap(\im\Psi_{2\lambda})$. It is compatible with
both $C$ and $C'$; specifically both $\{C(b)\mid b\in\mathcal S\}$ and
$\{C'(b)\mid b\in\mathcal S\}$ are bases of this subspace, where
$$\mathcal S=\{b\in B(\infty)\mid\wt(b)=-2\lambda,\
\varepsilon_1^*(b)=\varepsilon_3^*(b)=\varepsilon_4^*(b)=0,\
\varepsilon_2^*(b)\leq2\}.$$
The table below presents the $12$ elements in $\mathcal S$. Here
$b_0$ is the element of $B(\infty)$ of weight $0$, and given a word
$abc\cdots$ in the alphabet $\{1,2,3,4\}$, the notation
$\tilde f_{abc\cdots}b_0$ is a shorthand for the element
$\tilde f_a\tilde f_b\tilde f_c\cdots b_0$ in $B(\infty)$.
Lastly, $\vec\varepsilon(b)$ is the tuple
$(\varepsilon_1(b),\varepsilon_2(b),\varepsilon_3(b),\varepsilon_4(b))$.

\bigskip
\begin{center}
\setlength\tabcolsep{1.4em}
\begin{tabular}{>{$}r<{$}>{$}c<{$}}
\toprule
b\hspace*{6ex}&\vec\varepsilon(b)\\
\midrule
b_1=\tilde f_{2134221342}\;b_0&(0,1,0,0)\\[3pt]
b_2=\tilde f_{1342221342}\;b_0&(1,0,1,1)\\[3pt]
b_3=\tilde f_{3422113422}\;b_0&(0,0,1,1)\\[3pt]
b_4=\tilde f_{1422133422}\;b_0&(1,0,0,1)\\[3pt]
b_5=\tilde f_{1322134422}\;b_0&(1,0,1,0)\\[3pt]
b_6=\tilde f_{4221133422}\;b_0&(0,1,0,1)\\[3pt]
\bottomrule
\end{tabular}
\qquad
\begin{tabular}{>{$}r<{$}>{$}c<{$}}
\toprule
b\hspace*{6ex}&\vec\varepsilon(b)\\
\midrule
b_7=\tilde f_{3221134422}\;b_0&(0,1,1,0)\\[3pt]
b_8=\tilde f_{1221334422}\;b_0&(1,1,0,0)\\[3pt]
b_9=\tilde f_{4422113322}\;b_0&(0,0,0,2)\\[3pt]
b_{10}=\tilde f_{3322114422}\;b_0&(0,0,2,0)\\[3pt]
b_{11}=\tilde f_{1122334422}\;b_0&(2,0,0,0)\\[3pt]
b_{12}=\tilde f_{2211334422}\;b_0&(0,2,0,0)\\[3pt]
\bottomrule
\end{tabular}
\end{center}

\bigskip
The proof of Theorem~\ref{th:UniqCrys} shows that in the expansion
$$C'(b')=\sum_{b\in B(\infty)}m_{b,b'}\;C(b)$$
of an element of $C'$ in the basis $C$, the coordinate $m_{b,b'}$
necessarily vanishes except when $\varepsilon_i(b)\leq\varepsilon_i(b')$
for each $i$. We then deduce from the table above that $C(b)=C'(b)$ for
all $b\in\{b_1,b_3,b_4,b_5,b_9,b_{10},b_{11}\}$, that $C(b_2)-C'(b_2)$
is a linear combination of $C(b_3)$, $C(b_4)$, $C(b_5)$, and that for
$b\in\{b_6,b_7,b_8,b_{12}\}$, the difference $C(b)-C'(b)$ is a scalar multiple
of $C(b_1)$. In fact $C(b)=C'(b)$ also holds for $b\in\{b_2,b_6,b_7,b_8\}$;
to prove this for $b=b_8$ for instance, one can note that
\begin{xalignat*}2
\varepsilon_2(b_1)&=\varepsilon_2(b_8)&
\varepsilon_4(\tilde e_2\,b_1)&=\varepsilon_4(\tilde e_2\,b_8)\\[4pt]
\varepsilon_3(\tilde e_4\tilde e_2\,b_1)&=
\varepsilon_3(\tilde e_4\tilde e_2\,b_8)&
\varepsilon_2(\tilde e_3\tilde e_4\tilde e_2\,b_1)&>
\varepsilon_2(\tilde e_3\tilde e_4\tilde e_2\,b_8)
\end{xalignat*}
and refine the previous argument (see \cite{Baumann12}, \S 2.5).
To sum up, $C$ and $C'$ only differ at the element indexed by $b_{12}$,
and $C(b_{12})-C'(b_{12})$ is a scalar multiple of $C(b_1)$.

Let us set $\eta=e_2(e_1e_3e_4)e_2^{(2)}(e_1e_3e_4)e_2$, an element
in $U(\mathfrak n)$. Then $\langle\eta,C(b_1)\rangle=1$ (see
\cite{Baumann12}, Theorem~5.2, case~III; note that the elements $b_1$
and $b_{12}$ are denoted by $b_{0,1}$ and $b_{2,0}$ in that paper), so
$$C(b_{12})-C'(b_{12})=\bigl\langle\eta,C(b_{12})-C'(b_{12})
\bigr\rangle\,C(b_1).$$
If $C$ is the dual semicanonical basis, then
$\langle\eta,C(b_{12})\rangle=2$. If $C'$ is the dual canonical
basis/upper global basis (specialized at $q=1$), then
$\langle\eta,C'(b_{12})\rangle=1$. And if $C''$ is the MV basis
of $\ON$ (see \S \ref{se:MVBasisAlg} below for the definition of
this basis), then $\langle\eta,C''(b_{12})\rangle=0$. We thus see that
these three bases are pairwise different. Specifically, we have the
relations advertised in \S \ref{ss:CompBases}:
\begin{equation}
\label{eq:D4rel}
C(b_{12})=C''(b_{12})+2C(b_1)\quad\text{and}\quad
C'(b_{12})=C''(b_{12})+C(b_1).
\end{equation}

We will not develop enough material in the present paper to be able to
provide a complete justification of the equation
$\langle\eta,C''(b_{12})\rangle=0$, but we can nonetheless sketch the
proof. Let $b_{13}=\tilde f_{1342}\;b_0$ and
$b_{14}=\tilde f_{221342}\;b_0$; then $C''(b_{13})$ and $C''(b_{14})$ are
two elements in  $(\im\Psi_\lambda)$, that is, two matrix coefficients
of the adjoint representation $L(\lambda)$. A rather straightforward
calculation then gives $\langle\eta,C''(b_{13})C''(b_{14})\rangle=2$. On
the other hand, using Theorem~\ref{th:mult}, one can expand the product
$$C''(b_{13})C''(b_{14})=2C''(b_1)+\sum_{i=2}^8C''(b_i)+C''(b_{12}).$$
(The forthcoming paper~\cite{BaumannGaussentLittelmann} will explain
how to calculate the required intersection multiplicities. The actual
computations are rather tedious ---  for the coefficient $2$ one must
deal with a variety of codimension~$10$ defined by $18$ equations ---
and are carried out with the help of the computer algebra system
Singular~\cite{Singular}.) Since the matrix coefficients of the action
of the Chevalley generators in the MV basis of a representation are
nonnegative (a consequence of Theorem~\ref{th:ActEiMV} below), we have
$\langle\eta,C''(b)\rangle\geq0$ for any $b\in\mathcal S$. This is enough
to ensure that $\langle\eta,C''(b_{12})\rangle=0$.

In the appendix (Theorem \ref{th:AppendixMain}), we prove the analog of the first equation in (\ref{eq:D4rel}), but only after applying the noninjective map $ \barD $ (see section \ref{ssec:equivinvts}).  There we also use computer algebra systems along with techniques specific to type A.

\section{More on biperfect bases}
\label{se:MoreBiperf}

\subsection{Crystal reflections}
\label{ss:CrysRefl}
A result of Kashiwara and Saito (\cite[Proposition~3.2.3]{KashiwaraSaito}),
extended by Tingley and Webster (\cite[Proposition~1.4]{TingleyWebster})
says that the bicrystal $B(\infty)$ is characterized by the following
conditions:
\begin{enumerate}
\item
Both crystals $(B(\infty),\wt,\varepsilon_i,\varphi_i,\tilde e_i,\tilde f_i)$
and $(B(\infty),\wt,\varepsilon_i^*,\varphi_i^*,\tilde e_i^*,\tilde f_i^*)$
are upper semi-normal.
\item
$\wt(b)\in -Q_+$ for each element $b\in B(\infty)$.
\item
There is a unique element $b_0\in B(\infty)$ such that $\wt(b_0)=0$.
\item
\label{it:CharBInfty4}
For each $i\in I$ and $b\in B$, we have $\tilde f_ib\neq0$ and
$\tilde f_i^*b\neq0$.
\item
\label{it:CharBInfty5}
For each $i\neq j$ and each $b\in B(\infty)$, we have
$\varepsilon_j(\tilde f_i^*b)=\varepsilon_j(b)$,
$\varepsilon_j^*(\tilde f_ib)=\varepsilon_j^*(b)$, and
$\tilde f_i\tilde f_j^*b=\tilde f_j^*\tilde f_ib$.
\item
\label{it:CharBInfty6}
If $i\in I$ and $b_1\in B(\infty)$ satisfy
$\varepsilon_i(b_1)=\varepsilon_i^*(b_1)=0$,
then $\langle h_i,\wt(b_1)\rangle\geq0$.
\item
\label{it:CharBInfty7}
Let $i\in I$ and $b\in B$. The subset of $B(\infty)$ generated by
$b$ under the action of the operators $\tilde e_i$, $\tilde f_i$,
$\tilde e_i^*$, $\tilde f_i^*$ contains a unique element $b_1$ such
that $\varepsilon_i(b_1)=\varepsilon_i^*(b_1)=0$ and has the form drawn in Figure \ref{fig:vii}, where the action of $\tilde f_i$ is indicated by the plain
arrows, the action of $\tilde f_i^*$ is indicated by the dotted
arrows, and $\langle h_i,\wt(b_1)\rangle$ is the width of the shape.
(The picture is drawn for $\langle h_i,\wt(b_1)\rangle=4$.)
\end{enumerate}
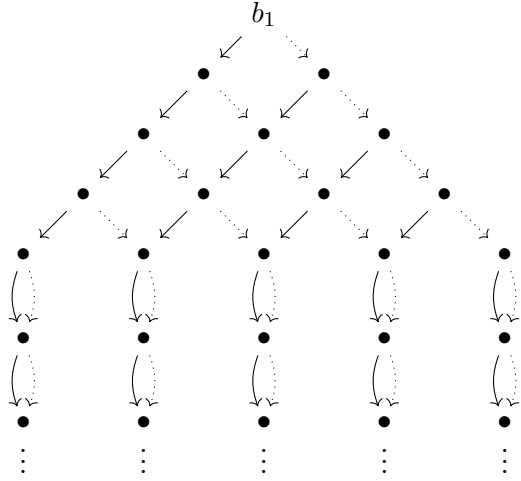
\begin{figure}[ht]  
\begin{center}
\begin{tikzpicture}[scale=.8]
\node (00) at (0,0){$b_1$};
\node (01) at (-1,-1){$\bullet$};
\node (11) at (1,-1){$\bullet$};
\node (02) at (-2,-2){$\bullet$};
\node (12) at (0,-2){$\bullet$};
\node (22) at (2,-2){$\bullet$};
\node (03) at (-3,-3){$\bullet$};
\node (13) at (-1,-3){$\bullet$};
\node (23) at (1,-3){$\bullet$};
\node (33) at (3,-3){$\bullet$};
\node (04) at (-4,-4){$\bullet$};
\node (14) at (-2,-4){$\bullet$};
\node (24) at (0,-4){$\bullet$};
\node (34) at (2,-4){$\bullet$};
\node (44) at (4,-4){$\bullet$};
\node (05) at (-4,-5.4){$\bullet$};
\node (15) at (-2,-5.4){$\bullet$};
\node (25) at (0,-5.4){$\bullet$};
\node (35) at (2,-5.4){$\bullet$};
\node (45) at (4,-5.4){$\bullet$};
\node (06) at (-4,-6.8){$\bullet$};
\node (16) at (-2,-6.8){$\bullet$};
\node (26) at (0,-6.8){$\bullet$};
\node (36) at (2,-6.8){$\bullet$};
\node (46) at (4,-6.8){$\bullet$};
\draw[->] (00) to (01);
\draw[->] (01) to (02);
\draw[->] (11) to (12);
\draw[->] (02) to (03);
\draw[->] (12) to (13);
\draw[->] (22) to (23);
\draw[->] (03) to (04);
\draw[->] (13) to (14);
\draw[->] (23) to (24);
\draw[->] (33) to (34);
\draw[->] (04) to[out=-110,in=110] (05);
\draw[->] (14) to[out=-110,in=110] (15);
\draw[->] (24) to[out=-110,in=110] (25);
\draw[->] (34) to[out=-110,in=110] (35);
\draw[->] (44) to[out=-110,in=110] (45);
\draw[->] (05) to[out=-110,in=110] (06);
\draw[->] (15) to[out=-110,in=110] (16);
\draw[->] (25) to[out=-110,in=110] (26);
\draw[->] (35) to[out=-110,in=110] (36);
\draw[->] (45) to[out=-110,in=110] (46);
\draw[->,dotted] (00) to (11);
\draw[->,dotted] (01) to (12);
\draw[->,dotted] (11) to (22);
\draw[->,dotted] (02) to (13);
\draw[->,dotted] (12) to (23);
\draw[->,dotted] (22) to (33);
\draw[->,dotted] (03) to (14);
\draw[->,dotted] (13) to (24);
\draw[->,dotted] (23) to (34);
\draw[->,dotted] (33) to (44);
\draw[->,dotted] (04) to[out=-70,in=70] (05);
\draw[->,dotted] (14) to[out=-70,in=70] (15);
\draw[->,dotted] (24) to[out=-70,in=70] (25);
\draw[->,dotted] (34) to[out=-70,in=70] (35);
\draw[->,dotted] (44) to[out=-70,in=70] (45);
\draw[->,dotted] (05) to[out=-70,in=70] (06);
\draw[->,dotted] (15) to[out=-70,in=70] (16);
\draw[->,dotted] (25) to[out=-70,in=70] (26);
\draw[->,dotted] (35) to[out=-70,in=70] (36);
\draw[->,dotted] (45) to[out=-70,in=70] (46);
\draw (06) ++(0,-.5) node{$\vdots$};
\draw (16) ++(0,-.5) node{$\vdots$};
\draw (26) ++(0,-.5) node{$\vdots$};
\draw (36) ++(0,-.5) node{$\vdots$};
\draw (46) ++(0,-.5) node{$\vdots$};
\end{tikzpicture}
\end{center}
\caption{Local structure of $ B(\infty)$.} \label{fig:vii}
\end{figure}

Note that (\ref{it:CharBInfty4}) is implied by (\ref{it:CharBInfty7})
and therefore not really needed.

\begin{remark}
These conditions imply that for any $b\neq b_0$, there exists $i\in I$
such that $\varepsilon_i(b)>0$. To show this, suppose that there exists
an element $b\neq b_0$ such that $\varepsilon_i(b)=0$ for all $i\in I$.
We may assume that among all possible elements, $b$ has been chosen to
be of maximal weight with respect to the dominance order. Since
$b\neq b_0$, the weight $\wt(b)$ is not dominant, and
(\ref{it:CharBInfty6}) implies the existence of $i\in I$ such that
$\varepsilon_i^*(b)>0$. So $b$ is on the upper right edge of the
shape drawn in Figure \ref{fig:vii}, but is not the top vertex.
Let $b_1=(\tilde e_i^*)^{\varepsilon_i^*(b)}b$ be the top vertex
of the shape. Certainly $\langle h_i,\wt(b_1)\rangle\neq0$, for
the shape has a positive width, and therefore $b_1\neq b_0$. By our
maximality condition, $b_1$ cannot satisfy the property imposed
on $b$, so there exists $j\in I$ such that $\varepsilon_j(b_1)>0$.
Necessarily $j\neq i$, and (\ref{it:CharBInfty5}) implies that
$\varepsilon_j(b)=\varepsilon_j(b_1)>0$, a contradiction.
(This argument comes from \cite[p.~16]{KashiwaraSaito}.)
\end{remark}

Using this, one easily recovers the following result due to Saito
(\cite[Corollary~3.4.8]{Saito}).
\begin{theorem}
Let $i\in I$.  The map
$$\sigma_i:\{b\in B(\infty)\mid\varepsilon_i^*(b)=0\}\to\{b\in B(\infty)
\mid\varepsilon_i(b)=0\}$$
given by $\sigma_i(b)=(\tilde f_i^*)^{\varphi_i(b)}
(\tilde e_i)^{\varepsilon_i(b)}(b)$ is bijective and
$\sigma_i^{-1}(b)=(\tilde f_i)^{\varphi_i^*(b)}
(\tilde e_i^*)^{\varepsilon_i^*(b)}(b)$.
\end{theorem}
Specifically $\sigma_i$ maps the upper left edge of the shape in
Figure \ref{fig:vii} to the upper right edge so that $\wt(\sigma_i(b))
=s_i\wt(b)$. Note that $\sigma_i(b)=(\tilde e_i)^{\varepsilon_i(b)}
(\tilde f_i^*)^{\varphi_i(b)}(b)$.

\subsection{Biperfect bases and Weyl group action}
\label{ss:BiperfWeyl}
The automorphism
$\Ad_{\overline s_i}$ of the Lie algebra $\mathfrak g$ extends
to an automorphism $T_i$ of the enveloping algebra $U(\mathfrak g)$.
We set $U^+=U(\mathfrak n)$, a subalgebra of $U(\mathfrak g)$.
Certainly $T_i$ restricts to a linear isomorphism
$$U^+\cap T_i^{-1}(U^+)\xrightarrow\simeq T_i(U^+)\cap U^+.$$

The following theorem generalizes to any biperfect basis a property
known for the dual canonical and the dual semicanonical bases (see
\cite[Theorem~1.2]{Lusztig96} and \cite[\S1.2]{Baumann11}).

\begin{theorem}
Let $B$ be a biperfect basis of $\ON$, let $i\in I$ and $b\in B$ be
such that $\varepsilon_i^*(b)=0$, and let $u\in U^+\cap T_i^{-1}(U^+)$.
Then
$$\langle\sigma_i(b),u\rangle=\langle b,T_i(u)\rangle.$$
\end{theorem}

The equation $\varepsilon_i(\sigma_i(b))=0$ implies that
$\sigma_i(b)$ annihilates $U^+e_i$. Taking into account the
decomposition $U^+=(U^+\cap T_i^{-1}(U^+))\oplus U^+e_i$,
we see that the theorem provides a completely algebraic characterization of
$\sigma_i(b)$ (previously defined only combinatorially).

\begin{proof}
Let $i\in I$ and $b\in B$ as in the statement. Set $b'=\sigma_i(b)$,
$n=\varepsilon_i(b)$, $p=\varphi_i(b)$. Then
$p=\varepsilon_i^*(b')=\varepsilon_i^*\bigl((\tilde f_i)^n(b')\bigr)$,
which implies
$$b=(\tilde e_i^*)^p(\tilde f_i)^n(b')=
\bigl((\tilde f_i)^n(b')\bigr)\cdot e_i^{(p)}.$$

We choose a dominant weight $\lambda\in P_+$ such that
$\langle h_i,\lambda\rangle=p$ and
$\langle h_j,\lambda\rangle\geq\varepsilon_j^*(b')$ for any $j\neq i$.
We adopt the notation of \S \ref{ss:BasesRep}; in particular the
simple $\mathfrak g$-module $L(\lambda)$ comes with the
$\mathfrak n$-equivariant embedding $\Psi_\lambda:L(\lambda)\to\ON$.
Then $b'\in \im\Psi_\lambda$, the set $B_\lambda=\Psi_\lambda^{-1}(B)$
is a perfect basis of $L(\lambda)$, and we can write
$b'=\Psi_\lambda(\overline{b'})$ for a certain element
$\overline{b'}\in B_\lambda$.

Now $\varepsilon_i(\overline{b'})=\varepsilon_i(b')=0$, so
$e_i\cdot\overline{b'}=0$ in the module $L(\lambda)$. Further
$$\langle h_i,\wt(\overline{b'})\rangle=
\langle h_i,\lambda+s_i(\wt(b))\rangle=
p-\langle h_i,\wt(b)\rangle =p-(p-n)=n,$$
so $\overline s_i\cdot\overline{b'}=f_i^{(n)}\cdot\overline{b'}$.
From Lemma~\ref{le:PerfBasRep}, it follows that
$\overline s_i\cdot\overline{b'}=(\tilde f_i)^n(\overline{b'})$,
and therefore
$$(\tilde f_i)^n(b')=
\Psi_\lambda\bigl((\tilde f_i)^n(\overline{b'})\bigr)=
\Psi_\lambda(\overline s_i\cdot\overline{b'}).$$

Let $\omega$ be the involutive antiautomorphism of $\mathfrak g$
such that $\omega(e_i)=f_i$ and $\omega(h_i)=h_i$ for each $i\in I$.
With respect to $\omega$, there is a unique contravariant form $(,)$ on
$L(\lambda)$ such that $(v_\lambda,v_\lambda)=1$ (see for instance
\cite[\S1.6]{Jantzen}). The embedding $\Psi_\lambda$ is given
by $\Psi_\lambda(v)=(v_\lambda,\bm?\,v)$; in particular
$(\tilde f_i)^n(b')=(v_\lambda,\bm?\;\overline s_i\cdot\overline{b'})$.
Then
$$b=(v_\lambda,e_i^{(p)}\,\bm?\;\overline s_i\cdot\overline{b'})=
(f_i^{(p)}\cdot v_\lambda,\bm?\;\overline s_i\cdot\overline{b'})=
(\overline s_i\cdot v_\lambda,\bm?\;\overline s_i\cdot\overline{b'})=
(v_\lambda,\overline s_i^{-1}\,\bm?\;\overline s_i\cdot\overline{b'}).$$
Evaluating this equation on $T_i(u)$, where $u\in U^+\cap T_i^{-1}(U^+)$,
we get $\langle b,T_i(u)\rangle=\langle b',u\rangle$, as desired.
\end{proof}

\subsection{MV polytopes} \label{ss:MVpolytopes}
Recall that $ B(\infty) $ denotes the abstract bicrystal common to all biperfect bases.  The theory of MV polytopes provides a convenient combinatorial model for $ B(\infty) $.
These polytopes also will also serve as the support for the measures to be introduced later in part~\ref{pa:Measures}.

It is simplest to introduce MV polytopes using the crystal reflections
$\sigma_i$. Specifically, we extend the definition of the crystal
reflections to all of $B(\infty)$ by defining
$\hat{\sigma_i}(b)=\sigma_i(\tilde e_i^{\max}b)$, where as usual
$\tilde e_i^{\max}b$ means $\tilde e_i^{\varepsilon_i(b)}b$.
These operators satisfy the braid relations
$$\underbrace{\hat{\sigma_i}\hat{\sigma_j}\hat{\sigma_i}\cdots}_{m_{i,j}
\text{ factors}}=\underbrace{\hat{\sigma_j}\hat{\sigma_i}\hat{\sigma_j}
\cdots}_{m_{i,j}\text{ factors}}$$
where $m_{i,j}$ is $2$, $3$, $4$, $6$ depending on $a_{ij}a_{ji}$
being $0$, $1$, $2$, $3$. To any $w\in W$ one can then attach an operator
$\hat\sigma_w$ on $B(\infty)$ so that
$\hat\sigma_w=\hat\sigma_{i_1}\cdots\hat\sigma_{i_\ell}$ for any reduced
word $w=s_{i_1}\cdots s_{i_\ell}$. Further, one can show that the map
$\hat \sigma_{w_0} $ takes every $b\in B(\infty)$ to the unique element
$b_0$ of weight~$0$.

(These facts follow from the work of Saito \cite{Saito}. Indeed, each
reduced decomposition $w_0=s_{i_1}\cdots s_{i_N}$ gives rise to a
bijection $B(\infty)\to\mathbb N^N$ called the Lusztig datum in direction
$(i_1,\ldots,i_N)$ (see \cite[\S 2]{Lusztig}). Under this bijection the action of $\hat s_{i_1}$ on
$\mathbb N^N$ has as its effect to drop the first coordinate and to insert a
zero on the right.)

We define the \textbf{MV polytope} $\Pol(b)$ of an element $b\in B(\infty)$
as the convex hull of the weights
$$\mu_w(b):=w\wt(\hat\sigma_{w^{-1}}(b))$$
for $w\in W$;
this polytope lies in $P\otimes_{\mathbb Z}\mathbb R$.
The vertices of $\Pol(b)$ are the points $ \mu_w(b) $.
The edges of $\Pol(b)$ are of the form $(\mu_w(b),\mu_{ws_i}(b))$ and
point in root directions; indeed if $ws_i>w$, then
$$\mu_{ws_i}-\mu_w=\varepsilon_i(\hat \sigma_{w^{-1}}(b))\;w\alpha_i.$$

\begin{remark} \label{rem:MVpoly}
In \cite{mvcrystal}, the second author gave an explicit combinatorial definition of MV polytopes, defined a bicrystal structure on this set of polytopes, and proved that this bicrystal is isomorphic to $ B(\infty) $.  Examining the definition of this crystal structure and its relationship with Lusztig data, it follows that we defined here the same set of polytopes with the same bijection with $ B(\infty) $.
\end{remark}

\part{Mirkovi\'c--Vilonen cycles}

\section{Background on the geometric Satake equivalence}
\label{se:RecGeomSat}

\subsection{The geometric Satake equivalence}
\label{ss:GeomSatEquiv}
Let $G^\vee$ denote the Langlands dual group of $G$. This reductive group
scheme comes with a maximal torus $T^\vee$ whose cocharacter lattice is
$P$. Our choice of positive and negative roots provide a pair of opposite
Borel subgroups $B^\vee_+$ and $B_-^\vee$ in $G^\vee$; we denote their
unipotent radicals by $N^\vee_+$ and $N_-^\vee$.

Let $\mathcal O=\mathbb C[[t]]$ be the ring of formal series and let
$\mathcal K=\mathbb C((t))$ be its fraction field of Laurent series. As a set,
the affine Grassmannian $\Gr$ of $G^\vee$ is the homogeneous space
$G^\vee(\mathcal K)/G^\vee(\mathcal O)$. It is the set of
$\mathbb C$-points of a reduced projective ind-scheme over $\mathbb C$;
see~\cite[\S 13.2.12--19]{Kumar} and \cite{Zhu} for a thorough
introduction to this object.

A weight $\mu\in P$ is a cocharacter of $T^\vee$; therefore it gives
a homomorphism of groups $\mathcal K^\times\to T^\vee(\mathcal K)$.
We denote by $t^\mu\in T^\vee(\mathcal K)\subset G^\vee(\mathcal K)$
the image of $t\in\mathcal K^\times$ under this homomorphism, and by
$L_\mu$ the image of $t^\mu$ in $\Gr$. These points $L_\mu$ are the
fixed points for the action of $T^\vee(\mathbb C)$ on $\Gr$.

Given $\lambda\in P_+$, we denote by $\Gr^\lambda$ the
$G^\vee(\mathcal O)$-orbit of $L_\lambda$ in $\Gr$. This is a smooth
variety of dimension $2\rho^\vee(\lambda)$.
The Cartan decomposition in $G^\vee(\mathcal K)$ implies that
$$\Gr=\bigsqcup_{\lambda\in P_+}\Gr^\lambda.$$
Each $\Gr^\lambda$ can be viewed as a (parabolic) Schubert cell; its closure is
obtained by adding the orbits $\Gr^\mu$ with $\mu\in P_+$ such that
$\mu\leq\lambda$.

Lusztig observed in \cite{Lusztig83} that a great deal of information about
the representation $L(\lambda)$ of $G$ is encoded in the geometry of
$\overline{\Gr^\lambda}$; for instance, the dimension of $L(\lambda)$
is equal to the dimension of the intersection homology of
$\overline{\Gr^\lambda}$.

Lusztig's insight can be regarded as a categorification of the classical
Satake isomorphism, where $G^\vee(\mathcal O)$-biinvariant compactly
supported functions on $G^\vee(\mathcal K)$ are replaced by
$G^\vee(\mathcal O)$-equivariant perverse sheaves on $\Gr$ with
coefficients in $\mathbb C$. Specifically, consider the category
$\Perv$ of such sheaves with finite dimensional support.  It is possible to endow
$\Perv$ with a convolution product along with suitable associativity
and commutativity constraints.  The total cohomology
provides a fiber functor $F$ from $\Perv$ to the category $\Vect$ of
finite dimensional $\mathbb C$-vector spaces. By Tannakian reconstruction,
$F$ induces an equivalence of categories from $\Perv$
to the category $\Rep$ of finite dimensional representations of a group
scheme $\overline G$ over $\mathbb C$ such that the diagram
$$\xymatrix@R=15pt@C=10pt{\Perv\ar[rr]^\simeq\ar[dr]_F&&\Rep\ar[dl]\\
&\Vect&}$$
commutes, where the right downward arrow is the forgetful functor.
The group scheme $\overline G$ is algebraic,
connected, reductive, and its root datum is inverse to the
root datum of $G^\vee$; in other words $\overline G$ is isomorphic to~$G$.

This program was carried out by Ginzburg~\cite{Ginzburg},
Beilinson--Drinfeld~\cite{BeilinsonDrinfeld} and
Mirkovi\'c--Vilonen~\cite{MirkovicVilonen}.
We refer the reader to the latter paper for the proof.

\subsection{Weight functors}
\label{ss:WeightFunct}
As a subgroup of $G^\vee(\mathcal K)$, the torus $T^\vee(\mathbb C)$
acts on $\Gr$. The regular dominant weight $\rho$ defines a
homomorphism $\mathbb C^\times\to T^\vee(\mathbb C)$, so provides
a $\mathbb C^\times$-action on $\Gr$. In their proof, Mirkovi\'c
and Vilonen define weight functors using the hyperbolic localization functors
defined by this action~\cite{Braden}. We recall part of their
construction.

Given $\mu\in P$, we denote by $S^\mu$ the $N^\vee_+(\mathcal K)$-orbit
through $L_\mu$ and by $S_-^\mu$ the $N_-^\vee(\mathcal K)$-orbit
through the same point. Then for each $L$ in $S_+^\mu$, respectively
$S_-^\mu$, we have
$$\lim_{a\to0}\rho(a)\cdot L=L_\mu,\quad\text{respectively }
\lim_{a\to\infty}\rho(a)\cdot L=L_\mu.$$
The Iwasawa decomposition in $G^\vee(\mathcal K)$ implies that
$$\Gr=\bigsqcup_{\mu\in P}S^\mu=\bigsqcup_{\mu\in P}S_-^\mu\,;$$
it follows that the points $L_\mu$ are the fixed points for
our $\mathbb C^\times$-action and that $S^\mu$ and $S_-^\mu$ are the
attractive and repulsive varieties around $L_\mu$.

Given $\mu\in P$, we define subsets $\overline{S^\mu}$ and
$\overline{S_-^\mu}$ by
$$\overline{S^\mu}=\bigsqcup_{\nu\leq\mu}S^\nu\quad\text{and}\quad
\overline{S_-^\mu}=\bigsqcup_{\nu\geq\mu}S_-^\nu.$$
These are closed subsets (\cite[Proposition~3.1]{MirkovicVilonen}).
Therefore we can factorize the inclusion map
$s_\mu:S_-^\mu\hookrightarrow\Gr$ as the composition of an open immersion
$\mathring s_\mu$ and a closed immersion $\overline s_\mu$ as follows
$$\xymatrix@C=36pt{S_-^\mu\ar[r]^{\mathring s_\mu}\ar@/_1.4pc/[rr]_{s_\mu}
&\overline{S_-^\mu}\ar[r]^{\overline s_\mu}&\Gr.}$$

For any weight $\mu\in P$ and any sheaf $\mathcal A\in\Perv$, the
cohomology $H^\bullet(S_-^\mu,\;s_\mu^{\;!}\mathcal A)$ is concentrated
in degree $k=2\rho^\vee(\mu)$ and we have a diagram
$$\xymatrix@C=40pt{H^k(\overline{S_-^\mu},\;\overline s_\mu^{\;!}\mathcal A)
\ar[r]^(.54){(\overline s_\mu)_!}\ar[d]_{(\mathring s_\mu)_*}^{\simeq}&
H^k(\Gr,\mathcal A)\\H^k(S_-^\mu,\;s_\mu^{\;!}\mathcal A)&}$$
where the vertical arrow is an isomorphism. Further, for each
$k\in\mathbb Z$, the maps $(\overline s_\mu)_!$ provide a decomposition
$$H^k(\Gr,\mathcal A)=\bigoplus_{\substack{\mu\in P\\[1pt] 2\rho^\vee(\mu)=k}}H^k\bigl(\overline{S_-^\mu},\;\overline s_\mu^{\;!}
\mathcal A\bigr)$$
by \cite{MirkovicVilonen}, Theorem~3.6. As a consequence, the
fiber functor $F=H^\bullet(\Gr,\bm?)$ from \S \ref{ss:GeomSatEquiv}
decomposes as a direct sum of weight functors
$$F=\bigoplus_{\mu\in P}F_\mu$$
defined by
$$F_\mu=H^{2\rho^\vee(\mu)}
\bigl(\overline{S_-^\mu},\;\overline s_\mu^{\;!}\;\bm?\bigr)=
H^\bullet\bigl(S_-^\mu,\;s_\mu^{\;!}\;\bm?\bigr).$$

Since this decomposition is compatible with the convolution product,
it defines a homomorphism $T\to\overline G$ that identifies $T$ with
a maximal torus of $\overline G$ (\cite[p.~122]{MirkovicVilonen}).

\subsection{Action of the principal nilpotent}
\label{ss:PrincNilp}
To understand how $G$ acts on the spaces $F(\mathcal A)$, we need to fix the
isomorphism $\overline G\cong G$. For this, we use an idea of Ginzburg
\cite{Ginzburg}.

Let $\mathfrak g^\vee$ be the Lie algebra of $G^\vee$ and let
$q:P\to\mathbb Q$ be the $W$-invariant quadratic form such that
$q(\alpha)=1$, if $\alpha$ is a short root of $G$.  Let $B:P\times P\to\mathbb Q$ be the polar form of
$q$ and let $\iota:P\to\mathfrak t$ be the map $\mu\mapsto B(\mu,\bm?)$.
The invariance of $q$ under the Weyl group implies that
$\iota(\alpha_i)=q(\alpha_i)\alpha_i^\vee$ for each $i\in I$.

From this data,
we can construct an affine Kac--Moody Lie algebra
$\widehat{\mathfrak g^\vee}$, as explained in~\cite[chapter~6]{Kac}.
With the standard notation set up in this reference, the dual of the
Cartan subalgebra of $\widehat{\mathfrak g^\vee}$ can be written as
$\mathfrak t\oplus\mathbb C\delta^\vee\oplus\mathbb C\Lambda_0$.
For $\mu\in P$, we set
$\pi(\mu)=\Lambda_0-\iota(\mu)-q(\mu)\delta^\vee$, an element
in~$\tilde{\mathfrak t}$.

Let $V$ denote the integrable representation of $\widehat{\mathfrak g^\vee}$
of highest weight $\Lambda_0$. This representation determines a homomorphism
$G^\vee(\mathcal K)\to\PGL(V)$~(\cite[Proposition~13.2.4]{Kumar}).  This can be
lifted to a representation of a central extension
$E(G^\vee(\mathcal K))$ of $G^\vee(\mathcal K)$ by $\mathbb C^\times$
(\cite[Proposition~13.2.8]{Kumar}). Moreover, the cocycle that defines
this extension involves the tame symbol (\cite[Theorem~12.24]{Garland});
this cocycle is trivial on $\mathcal O$, so this extension splits over
$G^\vee(\mathcal O)$, giving a diagram
$$\xymatrix{&&&G^\vee(\mathcal O)\ar[dl]_s\ar@{^{(}->}[d]&\\
1\ar[r]&\mathbb C^\times\ar[r]^(.34)i&E(G^\vee(\mathcal K))
\ar[r]^(.57)p&G^\vee(\mathcal K)\ar[r]&1.}$$

Let $v$ denote the highest weight vector of $V$. It is invariant under
the group $s(G^\vee(\mathcal O))$, so the map $g\mapsto gv$ defines an
embedding $\Upsilon:\Gr\into\mathbb P(V)$. This embedding is a morphism
of ind-varieties (\cite[\S2]{Slodowy}). We thereby obtain a (very ample)
$G^\vee(\mathcal K)$-equivariant line bundle
$\mathscr L=\Upsilon^*\mathscr O(1)$, which incidentally is known to
generate the Picard group of the identity component of $\Gr$.

 Formula~(6.5.4) in \cite{Kac} implies the
following statement (compare with \cite[(3.2)]{MirkovicVilonen}).
\begin{proposition}
	\label{pr:KacFormula}
	For each $\mu\in P$, the line $\Upsilon(L_\mu)$ is contained in the
	$\pi(\mu)$ weight space of $V$.
\end{proposition}

The cohomology algebra $H^\bullet(\Gr,\mathbb C)$ acts by the cup-product
on $F(\mathcal A)=H^\bullet(\Gr,\mathcal A)$ for any object
$\mathcal A\in\Perv$, and the action is natural in $\mathcal A$. In
particular, the cup-product with the first Chern class
$c_1(\mathscr L)$ is an endomorphism $\overline e$ of the functor $F$.
By Lemma~5.1 in \cite{YunZhu}, this element $c_1(\mathscr L)$ is primitive
in $H^\bullet(\Gr,\mathbb C)$, which implies that for any sheaves
$\mathcal A$, $\mathcal B$ in $\Perv$ we have
$$\overline e_{\mathcal A*\mathcal B}=
\overline e_{\mathcal A}\otimes\id_{F(\mathcal B)}
+\id_{F(\mathcal A)}\otimes\overline e_{\mathcal B}$$
under the isomorphism $F(\mathcal A*\mathcal B)\cong F(\mathcal A)
\otimes F(\mathcal B)$. It follows that $\overline e$ belongs to the
Lie algebra $\overline{\mathfrak g}$ of $\overline G$.

Now let $\overline h\in\overline{\mathfrak g}$ be the element that
acts as the multiplication by the cohomological degree on each vector
space $F(\mathcal A)$. Clearly we have $[\overline h,\overline e]=
2\overline e$. For any $\mathcal A\in\Perv$, the hard Lefschetz
theorem guarantees the existence of an endomorphism
$\overline f_{\mathcal A}$ of the vector space $F(\mathcal A)$
such that $(\overline e_{\mathcal A},\overline h_{\mathcal A},
\overline f_{\mathcal A})$ is a $\mathfrak{sl}_2$ triple. Certainly
$\overline f_{\mathcal A}$ is unique, hence natural in $\mathcal A$,
and we conclude that there is a unique element $\overline f\in
\overline{\mathfrak g}$ such that $(\overline e,\overline h,
\overline f)$ is an $\mathfrak{sl}_2$ triple (see
\cite[Theorem~5.3.23]{Zhu}).

By the end of \S \ref{ss:WeightFunct}, $T$ is a maximal torus of
the group $\overline G$, so we may decompose $\overline{\mathfrak g}$
into root subspaces with respect to the adjoint action of $T$. The root
system is then the root system of $G$. Further, $\overline h$
identifies with the element $2\rho^\vee\in\mathfrak t$, so
$\langle\overline h,\alpha_i\rangle=2$ for each simple root
$\alpha_i$ of $G$. By \cite[chapitre~8, \S11, Proposition~8]{Bourbaki},
we can then write $\overline e=\sum_{i\in I}\overline e_i$ where
each $\overline e_i$ is a nonzero root vector of weight $\alpha_i$.

With all these ingredients in hand, we can fix the isomorphism
$\overline G\cong G$ by identifying each simple root vector
$\overline e_i\in\overline{\mathfrak g}$ with its counterpart
$q(\alpha_i)e_i\in\mathfrak g$.

\section{The Mirkovi\'c--Vilonen basis in representations}
\label{se:MVBasisRepr}

\subsection{Some more notation}
\label{ss:SoMoNo}

In this section we recall standard facts and notation about sheaves and
cycles.  Throughout this paper, we will consider sheaves of $ \C$-vector spaces.  Similarly, singular cohomology, homology, and $K$-theory will always be considered with $\C$-coefficients.

Suppose that $X$ is a complex irreducible algebraic variety of dimension
$d$.  We denote the
constant sheaf on $X$ with stalk $\C$ by $\C_X$. The Verdier dual of
$\C_X$ is the dualizing sheaf $\D_X$ on $X$; it can be defined as either
$f^!\C_{pt}$ where $f:X\to pt$ is the constant map, or as the sheafification of the complex of presheaves
$U\mapsto C_{-\bullet}(X,X\setminus U)$ of relative singular chains.
The singular cohomology of $X$ is identified with $H^\bullet(X,\C_X)$; the Borel--Moore homology of $X$ (constructed from possibly infinite singular chains with locally finite support) is identified with $H_\bullet(X):=H^{-\bullet}(X,\D_X)$.

 Let $ D_c(X) $ denote the bounded derived category of constructible sheaves of $\C$-vector spaces on $ X $.  We have the (contravariant) Verdier duality functor $ \D : D_c(X) \rightarrow D_c(X) $ defined by $ R\mathcal Hom(-, \D_X) $.

\subsubsection{Intersection cohomology sheaf}
 The open subset of regular points $X^{\mathrm{reg}}$ is a real
connected oriented manifold of dimension $2d$, so
$H_{2d}(X^{\mathrm{reg}})$ is a one dimensional vector space spanned
by the fundamental class of $X^{\mathrm{reg}}$. Since $X\setminus
X^{\mathrm{reg}}$ is a pseudomanifold of dimension $\leq2d-2$, the
restriction map $H_{2d}(X)\to H_{2d}(X^{\mathrm{reg}})$ is an
isomorphism; we denote by $[X]$ the class in $H_{2d}(X)$ that restricts
to the fundamental class of $X^{\mathrm{reg}}$ and refer to $[X]$ as the
fundamental class of $X$. The same notation $[X]$ will also be used to
denote the image (proper pushforward) in $H_{2d}(Y)$ of this class
under a closed immersion $X\to Y$.

As a topological pseudomanifold, $X$ admits a filtration
with even real-dimensional strata, which allows to define unambiguously the sheaf of intersection chains $IC(X)$ w.r.t.\ the middle perversity; it
restricts to the shifted local system $\C_S[d]$ on the open stratum $S$.
Local sections of $IC(X)$ are (possibly infinite) singular chains that
satisfy specific conditions relative to how they meet the lower dimensional
strata of $X$; forgetting these conditions gives a map 
\begin{equation}
\label{eq:ICtoD}
	IC(X)\to\D_X[-d]
\end{equation}
which restricts on $S$ to the isomorphism
$$IC(X)\bigl|_S\equiv\C_S[d]\xrightarrow\simeq
\D_S[-d]\equiv\D_X[-d]\bigl|_S$$
given by the orientation (see \cite[\S5.1]{GoreskyMacPherson}).

\subsubsection{Cup and cap products}\label{sss:cupcap}
Let $ Z $ be a locally closed subset of $ X $ and let $ i: Z \rightarrow X $ denote the inclusion.  For any object $ \mathcal A \in D_c(X)$ on $ X$, we write $ H^\bullet_Z(X, \mathcal A) := H^\bullet(Z, i^! \mathcal A) $.  In particular, when $ \mathcal A = \C_X $, then we write $ H^\bullet_Z(X) := H^\bullet(Z, i^! \mathcal \C_X) $.  This is isomorphic to the singular cohomology $ H^\bullet(X, X \setminus Z) $.

Now let $ u \in H_Z^p(X) $.  So $ u : \C_Z \rightarrow i^! \C_X[p] $ and by adjunction, we can regard  $ u$ as a map  $ u :  i_! \basefield_Z \rightarrow \basefield_X [p]$.

For any $ \mathcal A \in D_c(X) $, we define its \defn{cup product} with $ u $ to be the resulting map $  i_! i^* \mathcal{A}  \rightarrow \mathcal{A}[p] $ given by applying $ \stackrel L\otimes \mathcal A $ to the map $ u $.  Taking compactly supported global sections gives us
$$
u \cup : H^k_c(Z, i^*\mathcal A) \rightarrow H^{k+p}_c(X,\mathcal A)
$$

Similarly for any $ \mathcal A \in D_c(X) $, we define its \defn{cap product} with $ u $ to be the resulting map $  \mathcal{A}  \rightarrow i_* i^! \mathcal{A} [p]$ (given by applying $ R\mathscr{H}om( - , \mathcal A)[p] $ to the map $ u $).  Taking global sections gives us
$$
u \cap :
H^{-k}(X, \mathcal A) \rightarrow H^{-k+p}_Z(X,\mathcal A)
$$

If we take $ \mathcal{A} = \D_X $, we obtain the usual cap product map in Borel--Moore homology,
\begin{equation*}
u \cap : H_k(X) \rightarrow H_{k-p}(Z).
\end{equation*}

For any $ \mathcal A \in D_c(X) $, if we take the cup product map
$$
i_! i^* \mathcal{A}  \xrightarrow{u \cup} \mathcal{A}[p]
$$
and apply Verdier duality, we obtain a map
$$
\D(\mathcal A)[-p] \xrightarrow{\D(u \cup)} i_* i^! \D(\mathcal{A})
$$
which coincides with the map $ u \cap $ shifted by $-p$,
by~\cite[(2.6.7)]{KashiwaraSchapira}.

These cup and cap product maps are compatible with pullback.  Let us explain this compatibility in the case of cap product.  Let $ f : Y \rightarrow X $ be a morphism and let $ W = f^{-1}(Z) $.  Then we can form $ f^*u \in H_{W}^p(Y) $ and given any $ \mathcal A \in D_c(Y) $, we have the following commutative diagram
\begin{equation} \label{eq:projection}
\xymatrix{
	H^{-k}(Y, \mathcal A) \ar^{(f^*u) \cap }[r]  \ar^\cong[d] & H^{-k+p}_W(Y, \mathcal A) \ar^\cong[d] \\
	H^{-k}(X, f_* \mathcal A) \ar^{u \cap }[r]  & H^{-k+p}_Z(X, f_* \mathcal A)
}
\end{equation}
where the vertical isomorphisms come from the composition of push-forwards (and base change for the right hand vertical arrow).

\subsubsection{Cap products and cycles}
An effective way to compute cap products with Chern classes is to reduce
to calculations of intersection multiplicities. We quickly recall the
definitions and the basic results from Fulton's book \cite{Fulton} in
the very specific setup that we will need. We consider a fibre square
of $\mathbb C$-schemes
$$\xymatrix{W\ar[r]\ar[d]&V\ar[d]\\D\ar[r]_i&Y}$$
with $i$ the inclusion of an effective Cartier divisor and $V$ an
irreducible variety of dimension $k$. We assume that $V$ is not contained
in the support of $D$. Let $Z$ be an irreducible component of $W$;
it is a subvariety of $V$ of codimension $1$. The \textbf{multiplicity}
of $Z$ in the product $D\cdot V$ is defined to be the length of the
module $\mathcal O_{W,V}/(f)$ over the local ring $\mathcal O_{W,V}$
of $V$ along $W$, where $f$ is a local equation of $D|_V$ on an affine
open subset of $V$ which meets $Z$. Following \cite[chap.~7]{Fulton},
this multiplicity is denoted by $i(Z,D\cdot V)$.

The inclusion $i$ is a regular embedding of codimension $1$, hence
has an orientation class $u\in H^2(Y,Y\setminus D)$. Concretely
(see \cite[\S19.2]{Fulton}), $D$ is the zero-locus of the canonical
section $s:Y\to\mathcal O_Y(D)$ and $u$ is the pullback by $s$ of the
Thom class of the line bundle $\mathcal O_Y(D)$. Now consider a fibre square 
$$\xymatrix{X'\ar[r]\ar[d]&Y'\ar[d]^g\\D\ar[r]_i&Y}$$
Following the above discussion we have the cap product
$$H_k(Y')\xrightarrow{(g^*u)\cap}H_{k-2}(X').$$
The following result follows from Theorem~19.2 in \cite{Fulton}.

\begin{proposition} \label{pr:Fulton}
Retain the above notation and assume that 
 $V$ is an irreducible
subvariety of $Y'$ of dimension $k$, not contained in $g^{-1}(D)$. 
Then
$$(g^*u)\cap[V]=\sum_Zi(Z,D\cdot V)\,[Z]$$
the sum being taken over all irreducible components of $V\cap g^{-1}(D)$.
\end{proposition}

\subsection{Mirkovi\'c--Vilonen cycles}
\label{ss:MVCycles}
Let $\lambda\in P_+$, fixed for the whole section. As shown
by Mirkovi\'c and Vilonen (\cite[Theorem~3.2]{MirkovicVilonen}),
given $\mu\in P$, the intersection $\Gr^\lambda\cap S_-^\mu$, when
non-empty, has pure dimension $\rho^\vee(\lambda-\mu)$.
We define an \textbf{MV cycle} $Z$ of type $\lambda$ and weight $\mu$ to be
an irreducible component of $\overline{\Gr^\lambda\cap S_-^\mu}$.
Equivalently, an MV cycle of type $\lambda$ and weight $\mu$ is an
irreducible component of $\overline{\Gr^\lambda}\cap\overline{S_-^\mu}$
of dimension $\rho^\vee(\lambda-\mu)$.
We denote the set of these cycles by $\mathcal Z(\lambda)_\mu$
and define
$$\mathcal Z(\lambda):=\bigsqcup_{\mu\in P}\mathcal Z(\lambda)_\mu,$$
the set of all MV cycles of type $\lambda$.

Braverman and Gaitsgory endow $\mathcal Z(\lambda)$ with the structure
of an upper semi-normal $G$-crystal~\cite{BravermanGaitsgory}. Their
definition involves a geometric construction, but one can provide the
following purely combinatorial short characterization
(\cite[proof of Proposition~4.3]{BaumannGaussent}):
\begin{itemize}
\item
Let $\mu\in P$ and $Z\in\mathcal Z(\lambda)_\mu$. We set $\wt(Z)=\mu$.
The closed subset $Z$ is $\mathbb C^\times$-invariant with respect to the
action defined in \S \ref{ss:WeightFunct} and meets the repulsive
cell $S_-^\mu$, so $L_\mu\in Z$. For each $i\in I$, we can then define
$$\varepsilon_i(Z)=\max\{n\in\mathbb Z_{\geq0}\mid
L_{\mu+n\alpha_i}\in Z\}\;\text{ and }\;
\varphi_i(Z)=\varepsilon_i(Z)+\langle\alpha_i^\vee,\mu\rangle.$$
\item
Let $\mu\in P$, $i\in I$ and $(Z,Z')\in\mathcal Z(\lambda)^2$. Then
$$Z'=\tilde e_iZ\ \Longleftrightarrow\ \Bigl(Z'\subseteq Z,\;\
\wt(Z')=\wt(Z)+\alpha_i\;\text{ and }\;
\varepsilon_i(Z')=\varepsilon_i(Z)-1\Bigr).$$
\end{itemize}

We denote the intersection cohomology sheaf of the Schubert variety
$\overline{\Gr^\lambda}$ by $\IC_\lambda$. The geometric Satake
equivalence maps this perverse sheaf to the simple $G$-module
$L(\lambda)$. In other words, under the identification $\overline G
\cong G$ specified at the end of \S \ref{se:RecGeomSat}, there
is an isomorphism $F(\IC_\lambda)\cong L(\lambda)$, unique up to a
scalar. For each $\mu\in P$, the subspace of $L(\lambda)$ of weight
$\mu$ identifies with
$$F_\mu(\IC_\lambda)=H^k\bigl(\overline{S_-^\mu},\;\overline s_\mu^{\;!}
\,\IC_\lambda\bigr) = H^k_{\overline{S_-^\mu}}\bigl(\Gr, \IC_\lambda\bigr)$$
where $k=2\rho^\vee(\mu)$.

By (\ref{eq:ICtoD}), we have a map of sheaves $\IC_\lambda\to\D_{\overline{\Gr^\lambda}}[-d]$,
where $d=\dim\Gr^\lambda=2\rho^\vee(\lambda)$.
Denoting the inclusion of the open stratum
by $j:\Gr^\lambda\to\overline{\Gr^\lambda}$, we then get a commutative diagram
$$\xymatrix{\IC_\lambda\ar[r]\ar[d]&\D_{\overline{\Gr^\lambda}}[-d]\ar[d]\\
j_*j^*\,\IC_\lambda\ar[r]^-\simeq&j_*j^*\,\D_{\overline{\Gr^\lambda}}[-d]}$$
where the vertical arrows are adjunction maps. The bottom arrow is
an isomorphism because both $\IC_\lambda[-d]$ and
$\D_{\overline{\Gr^\lambda}}[-2d]$ restrict to the trivial local system
over $\Gr^\lambda$.

Base change in the Cartesian square
$$\xymatrix{\Gr^\lambda\cap\overline{S_-^\mu}\ar[r]^{j'}
\ar[d]_{\overline s'_\mu}&\overline{\Gr^\lambda}\cap\overline{S_-^\mu}
\ar[r]\ar[d]&\overline{S_-^\mu}\ar[d]^{\overline s_\mu}\\
\Gr^\lambda\ar[r]_j&\overline{\Gr^\lambda}\ar[r]&\Gr}$$
gives the isomorphism
$\overline s_\mu^{\;!}\;j_*j^*\,\D_{\overline{\Gr^\lambda}}\cong
j'_*\;\D_{\Gr^\lambda\cap\overline{S_-^\mu}}$. Applying the
functor $H^k\bigl(\overline{S_-^\mu},\;\overline s_\mu^{\;!}\,\bm?\bigr)$
to the commutative diagram above then yields
$$\xymatrix{F_\mu(\IC_\lambda)=H^k\bigl(\overline{S_-^\mu},\;
\overline s_\mu^{\;!}\,\IC_\lambda\bigr)\ar[r]\ar@<25pt>[d]_\simeq
&H^{k-d}\bigl(\overline{S_-^\mu},\;\overline s_\mu^{\;!}\,
\D_{\overline{\Gr^\lambda}}\bigr)\ar@{=}[r]\ar[d]
&H_{d-k}\bigl(\overline{\Gr^\lambda}\cap\overline{S_-^\mu},\,
\mathbb C\bigr)\ar[d]^\simeq\\
\rule{40pt}{0pt}H^k\bigl(\overline{S_-^\mu},\;\overline
s_\mu^{\;!}\,j_*j^*\,\IC_\lambda\bigr)\ar[r]^-\simeq
&H^{k-d}\bigl(\overline{S_-^\mu},\;j'_*\,\D_{\Gr^\lambda\cap
\overline{S_-^\mu}}\bigr)\ar@{=}[r]
&H_{d-k}\bigl(\Gr^\lambda\cap\overline{S_-^\mu},\,\mathbb C\bigr).}$$
Here the left vertical arrow is an isomorphism, as shown by Mirkovi\'c
and Vilonen (\cite[proof of Proposition~3.10]{MirkovicVilonen}).
The right vertical arrow is also an isomorphism, because each irreducible
component of $\overline{\Gr^\lambda}\cap\overline{S_-^\mu}$ of dimension
$(d-k)/2=\rho^\vee(\lambda-\mu)$ meets the open stratum
$\Gr^\lambda$. In fact, these irreducible components are precisely
the MV cycles of type~$\lambda$ and of weight $\mu$.

We denote by $[Z]\in H_{d-k}\bigl(\overline{\Gr^\lambda}\cap
\overline{S_-^\mu},\,\mathbb C\bigr)$ the fundamental class of such an
MV cycle~$Z$.
The set $\{[Z]\mid Z\in\mathcal Z(\lambda)_\mu\}$ is then a basis of
the weight space $F_\mu(\IC_\lambda)$. Gathering these bases for all
possible weights, we obtain a basis of $F(\IC_\lambda)$ indexed by
$\mathcal Z(\lambda)$, which we can transport to $L(\lambda)$ by
normalizing the isomorphism $F(\IC_\lambda)\cong L(\lambda)$ in such
a way that the highest weight vectors $[\{L_\lambda\}]$ and $v_\lambda$
match.

The basis of $L(\lambda)$ obtained in this manner is called the
\textbf{MV basis}. (By analogy with the case of the global basis and
following Kashiwara's terminology, we should more accurately call it the
upper MV basis.) The following result is (up to duality) Proposition~4.1
in \cite{BravermanGaitsgory}.

\begin{theorem}  \label{th:MVBasPerf}
  The MV basis is perfect. 
\end{theorem}

We will give our own proof of this result, which provides some
more refined information needed in the sequel.
The first step in the proof is carried out in \S\ref{ss:ActEiMV},
where we establish a formula that expresses in geometrical terms the
action of a Chevalley generator $e_i$ on a basis element $[Z]$. This
formula has the form
$$e_i\cdot[Z]=\sum_{Z'\in\mathcal Z(\lambda)}a_{Z'}[Z'],$$
where the coefficient $a_{Z'}$ is nonzero only if
$$\wt(Z')=\wt(Z)+\alpha_i\;\text{ and }\;Z'\subseteq Z.$$
Thus, for $[Z']$ to actually appear in $e_i\cdot[Z]$, it is necessary
that $L_{\wt(Z')+\varepsilon_i(Z')\alpha_i}\in Z$, which in turn implies
that $\varepsilon_i(Z)\geq\varepsilon_i(Z')+1$. If moreover the latter
relation is an equality, then necessarily $Z'=\tilde e_iZ$, by the
characterization of the crystal structure on $\mathcal Z(\lambda)$
given above. At this point, it remains to show that if
$\tilde e_iZ\neq0$, then the coefficient $a_{\tilde e_iZ}$ is equal to
$\varepsilon_i(Z)$. We perform this computation in \S \ref{ss:CompLeadCoef}.

\begin{remark}
As shown by Berenstein and Kazhdan~\cite{BerensteinKazhdan}, the
crystal of a perfect basis of the highest weight module $L(\lambda)$ is
independent of the choice of the basis. Therefore the crystals of the MV
basis and of the upper global basis of $L(\lambda)$ (specialized at $q=1$)
are isomorphic. This observation provides another proof of Braverman
and Gaitsgory's theorem~\cite{BravermanGaitsgory} that states that the
crystals $\mathcal Z(\lambda)$ are isomorphic to Kashiwara's normal
crystals~$B(\lambda)$.
\end{remark}

\subsection{Action of $e_i$ on an MV cycle}
\label{ss:ActEiMV}
Recall the notation set up in \S \ref{ss:PrincNilp} and the statement of Proposition \ref{pr:KacFormula}.

Fix $\mu\in P$ and pick a linear form $f$ on $V$ which is
nonzero on the line $\Upsilon(L_\mu)$ and which vanishes on all weight
subspaces of $V$ of weight other than $\pi(\mu)$. Let $D\subseteq
\Gr$ be the Cartier divisor defined as the intersection of $\Gr$
with the hyperplane in $\mathbb P(V)$ defined by $f$.
Proposition~3.1 in \cite{MirkovicVilonen} tells us that
$$D\cap\overline{S_-^\mu}=\bigcup_{i\in I}\overline{S_-^{\mu+\alpha_i}}.$$

\begin{theorem}
\label{th:ActEiMV}
Let $\lambda\in P_+$, let $i\in I$, and let $Z\in\mathcal Z(\lambda)_\mu$.
Let
$$e_i\cdot[Z]=\sum_{Z'\in\mathcal Z(\lambda)}a_{Z'}[Z']$$
be the expansion of the left-hand side in the MV basis of $L(\lambda)$.
Then
$$q(\alpha_i)a_{Z'}=
\begin{cases}
i(Z',D\cdot Z)&\text{if $\wt(Z')=\wt(Z)+\alpha_i$ and $Z'\subseteq Z$,}\\
0&\text{otherwise.}
\end{cases}$$
\end{theorem}

\begin{proof}
Regarding $\Gr$ as the zero section of the total space $L$ of the line
bundle $\mathscr L=\Upsilon^*\mathscr O(1)$, we can consider the Thom class
$\tau\in H^2_\Gr(L)$. Regarding $f$ as a continuous map from $\Gr$ to $L$
such that $f(\Gr\setminus D)\subseteq L\setminus\Gr$, we can form
$f^*\tau\in H^2_D(\Gr)$. With these notations, each perverse sheaf
$\mathcal A\in\Perv$ gives rise to a diagram
$$\xymatrix{H^\bullet_{\overline{S_-^\mu}}(\Gr,\mathcal A)\ar[r]
\ar[d]_{\cup f^*\tau}&H^\bullet(\Gr,\mathcal A)
\ar[d]_{\cup f^*\tau}\ar[dr]^(.65){\cup c_1(\mathscr L)}&\\
H_{D\cap\overline{S_-^\mu}}^\bullet(\Gr,\mathcal A)\ar[r]&
H_D^\bullet(\Gr,\mathcal A)\ar[r]&H^\bullet(\Gr,\mathcal A),}$$
which commutes following \cite[II.10.2 and II.10.4]{Iversen}.

Now let $\lambda\in P_+$ and set $\mathcal A=\IC_\lambda$,
$k=2\rho^\vee(\mu)$, and
$d=2\rho^\vee(\lambda)$. Similarly to the isomorphism
$$H^k_{\overline{S_-^\mu}}(\Gr,\,\IC_\lambda)\xrightarrow{\simeq}
H^{k-d}_{\overline{S_-^\mu}}(\Gr,\,\D_{\overline{\Gr^\lambda}})
=H_{d-k}\bigl(\overline{\Gr^\lambda}\cap\overline{S_-^\mu},\,
\mathbb C\bigr)$$
obtained in \S\ref{ss:MVCycles}, we have an isomorphism
$$H^{k+2}_{D\cap\overline{S_-^\mu}}(\Gr,\,\IC_\lambda)
\xrightarrow\simeq
H^{k+2-d}_{D\cap\overline{S_-^\mu}}(\Gr,\,\D_{\overline{\Gr^\lambda}})
=H_{d-k-2}\bigl(D\cap\overline{\Gr^\lambda}\cap\overline{S_-^\mu},\,
\mathbb C\bigr).$$
We then get a commutative diagram
$$\xymatrix@C=36pt@R=28pt{
H^k_{\overline{S_-^\mu}}(\Gr,\,\IC_\lambda)\ar[d]_{\cup f^*\tau}
\ar[r]^(.46)\simeq&H^{k-d}_{\overline{S_-^\mu}}
\bigl(\Gr,\,\mathbb D_{\overline{\Gr^\lambda}}\bigr)\ar[d]_{\cup f^*\tau}
\ar@{=}[r]&
H_{d-k}\bigl(\overline{\Gr^\lambda}\cap\overline{S_-^\mu},\,\mathbb C\bigr)
\ar[d]^{\cap f^*\tau}\\
H^{k+2}_{D\cap\overline{S_-^\mu}}(\Gr,\,\IC_\lambda)\ar[r]^(.46)\simeq
&H^{k+2-d}_{D\cap\overline{S_-^\mu}}
\bigl(\Gr,\,\mathbb D_{\overline{\Gr^\lambda}}\bigr)\ar@{=}[r]&
H_{d-k-2}\bigl(D\cap\overline{\Gr^\lambda}\cap\overline{S_-^\mu},
\,\mathbb C\bigr).}$$

Let $Z\in\mathcal Z(\lambda)_\mu$ and let $[Z]$ be its fundamental
class in $H_{d-k}\bigl(\overline{\Gr^\lambda}\cap\overline{S_-^\mu},
\,\mathbb C\bigr)$. The two commutative diagrams above and the
explanations in \S\ref{ss:PrincNilp} show that $[Z]\cap f^*\tau$
is the result of the action on $[Z]$ of the principal nilpotent
$\tilde e$. On the other hand, $f^*\tau$ is the orientation class
$u$ of the regular embedding $D\to\Gr$, so $[Z]\cap f^*\tau$
is the homology class of the cycle $D\cdot Z$, by Proposition \ref{pr:Fulton} applied to the fibre square
$$\xymatrix{D\cap\overline{\Gr^\lambda}\cap\overline{S_-^\mu}\ar[d]\ar[r]
&\overline{Gr^\lambda}\cap\overline{S_-^\mu}\ar[d]\\D\ar[r]&\Gr.}$$

Now any irreducible component of $D\cap Z$ must be contained in
$\overline{\Gr^\lambda}\cap\overline{S_-^{\mu+\alpha_i}}$ for some
$i\in I$; being of dimension
$\dim Z-1=\rho^\vee(\lambda-\mu-\alpha_i)$, it is
then of the form $Z'$ with $Z'\in\mathcal Z(\lambda)_{\mu+\alpha_i}$.
We eventually obtain
$$\overline e\cdot[Z]=\sum_{i\in I}
\sum_{Z'\in\mathcal Z(\lambda)_{\mu+\alpha_i}}a_{Z'}[Z']$$
where $a_{Z'}=i(Z',D\cdot Z)$ if $Z'\subseteq Z$ and $a_{Z'}=0$
otherwise. The claimed formula follows by isolating the contributions
of the different summands in $\overline e=\sum_{i\in I}q(\alpha_i)e_i$.
\end{proof}

\subsection{Computation of the leading coefficient}
\label{ss:CompLeadCoef}

\begin{proposition}
\label{pr:CalcMult}
Adopt the notation of Theorem~\ref{th:ActEiMV} and assume that
$\tilde e_iZ\neq0$. Then
$$i(\tilde e_iZ,D\cdot Z)=q(\alpha_i)\,\varepsilon_i(Z).$$
\end{proposition}
\begin{proof}
Let $s_i^\vee$ be a lift in $G^\vee$ of the simple reflection $s_i$.
The weight $s_i\rho:\mathbb C^\times\to T^\vee(\mathbb C)$ defines an
action of $\mathbb C^\times$ on $\Gr$. With respect to this action,
the repulsive cell around the fixed point $L_\nu$ is the subset
$s_i^\vee(S_-^{s_i\nu})$.

Let $x_i^\vee$ be the additive one-parameter subgroup corresponding
to the simple root $\alpha_i^\vee$ of $(G^\vee,T^\vee)$; it defines
a homomorphism $\mathcal K\to G^\vee(\mathcal K)$. Let $N_{-,i}^\vee$
be the unipotent radical of the parabolic subgroup of $G^\vee$
generated by $N_-^\vee$ and by the image of $x_i^\vee$. The subgroup
generated by $N_{-,i}^\vee$ and the image of $x_i^\vee$ is then
the maximal unipotent subgroup $s_i^\vee\,N_-^\vee(\mathcal K)\,
(s_i^\vee)^{-1}$ of $G^\vee(\mathcal K)$. We can lift it to the
central extension $E(G^\vee(\mathcal K))$ so as to make it act on $V$,
and the embedding $\Upsilon:\Gr\to\mathbb P(V)$ is equivariant for this
action.

After these preliminaries, let us genuinely start the proof. Let
\begin{gather*}
m=\varepsilon_i(Z),\quad\nu=\mu+m\alpha_i,\quad
r=\langle\alpha_i^\vee,\nu\rangle,\quad Z'=\tilde e_i^mZ,\\[2pt]
\dot Z=Z\cap\Gr^\lambda\cap s_i^\vee(S_-^{s_i\nu}),\quad
\dot Z'=Z'\cap\Gr^\lambda\cap s_i^\vee(S_-^{s_i\nu}),\\
\mathbb C\bigl[t^{-1}\bigr]^+_m=\bigl\{a_{-m}t^{-m}+a_{1-m}t^{1-m}+
\cdots+a_{-1}t^{-1}\bigm|(a_{-m},a_{1-m},\ldots,a_{-1})\in\mathbb
C^m\bigr\}.
\end{gather*}
From \cite[Proposition~4.5~(ii)]{BaumannGaussent}, we see that the
assignment $(p,y)\mapsto x_i^\vee(pt^r)y$ defines a homeomorphism
$F:\mathbb C\bigl[t^{-1}\bigr]^+_m\times\dot Z'\to\dot Z$;
it is even an isomorphism of algebraic varieties. The cycle
$\tilde e_iZ$ is a divisor in $Z$; its local equation in the open
subset~$\dot Z$ is $a_{-m}=0$.

Our goal now is to evaluate the `equation' $f\circ\Upsilon$ of the
divisor $D$ at a point $F(p,y)$. (We put quotation marks around the
word `equation' because $f\circ\Upsilon$ is a section of $\mathscr L$,
not a function.)

We write $y\in\dot Z'$ in the form $y=gL_\nu$ with
$g\in N_{-,i}^\vee(\mathcal K)$. We write $\Upsilon(L_\nu)=\mathbb Cv_0$
with $v_0\in V$ of weight $\pi(\nu)$ and we decompose $v=gv_0$ as a
sum of weight vectors $v=v_0+\cdots+v_N$ with $v_j$ is of weight
$\pi(\nu)-\beta_j-c_j\delta^\vee$, where $\beta_0=c_0=0$ and
$\beta_j\notin\mathbb Z\alpha_i^\vee$ for $j\neq0$. We put
$p=a_{-m}t^{-m}+a_{1-m}t^{1-m}+\cdots+a_{-1}t^{-1}$ and we expand
the product
$$x_i^\vee(pt^r)v=\prod_{\ell=1}^m
\Biggl(\sum_{n_\ell\geq0}\frac{a_{-\ell}^{n_\ell}}{n_\ell!}\;
X_{\alpha_i^\vee+(r-\ell)\delta^\vee}^{n_\ell}\Biggr)v,$$
where $X_{\alpha_i^\vee+q\delta^\vee}\in\widehat{\mathfrak g^\vee}$
is the derivative at zero of the additive subgroup
$a\mapsto x_i^\vee(at^q)$ of $G^\vee(\mathcal K)$.
The linear form $f$ vanishes on
$$\Biggl(\prod_{\ell=1}^m
X_{\alpha_i^\vee+(r-\ell)\delta^\vee}^{n_\ell}\Biggr)v_j$$
except if
$$\sum_{\ell=1}^mn_\ell(\alpha_i^\vee+(r-\ell)\delta^\vee)
+\pi(\nu)-\beta_j-c_j\delta^\vee=\pi(\mu),$$
which can be rewritten as a system of two equations
$$\Biggl(\sum_{\ell=1}^mn_\ell\Biggr)\alpha_i^\vee
-m\,q(\alpha_i)\alpha_i^\vee-\beta_j=0\quad\text{and}\quad
\sum_{\ell=1}^m(r-\ell)n_\ell-c_j=m(r-m)\,q(\alpha_i).$$
The first one requires that $\beta_j\in\mathbb Z\alpha_i^\vee$,
hence that $j=0$; the condition then becomes
$$\sum_{\ell=1}^mn_\ell=m\,q(\alpha_i)\;\text{ and }\;
\sum_{\ell=1}^m\ell n_\ell=m^2\,q(\alpha_i),$$
which in turn is equivalent to $n_1=\cdots=n_{m-1}=0$ and
$n_m=m\,q(\alpha_i)$. To sum up, we have
$$f\bigl(x_i^\vee(pt^r)v\bigr)=\frac{a_{-m}^n}{n!}\;
f\bigl(X_{\alpha_i^\vee+(r-m)\delta^\vee}^n\;v_0\bigr),$$
with $n=m\,q(\alpha_i)$.

Now let $g$ be another linear form on $V$, which is nonzero on the line
$\Upsilon(L_\nu)$ and which vanishes on all weight subspaces of $V$ of
weight other than $\pi(\nu)$. Then $g\circ\Upsilon$ does not identically
vanish on $\dot Z$, because
$$g\bigl(x_i^\vee(pt^r)v\bigr)=g(v_0)\neq0.$$
Therefore the rational function $(f/g)\circ\Upsilon$ has no poles
on $\dot Z$ and has value
$$a_{-m}^n\;\frac{f\bigl(X_{\alpha_i^\vee+(r-m)\delta^\vee}^n\;
v_0\bigr)}{n!\;g(v_0)}$$
at the point $F(p,y)$. The second factor does not depend on $(p,y)$,
hence is a nonzero constant. We thus see that the local equation
$(f/g)\circ\Upsilon$ of the divisor $D$ vanishes along $\tilde e_iZ$
with multiplicity $n=q(\alpha_i)\,\varepsilon_i(Z)$, as asserted.
\end{proof}

\section{The Mirkovi\'c--Vilonen basis of $\ON$}
\label{se:MVBasisAlg}

\subsection{Stabilization}
\label{ss:CohMVBases}
Let $\nu\in Q_+$. Then the subset $S_+^0\cap S_-^{-\nu}$ of $\Gr$ is
non-empty and has pure dimension $\rho^\vee(\nu)$.
We define a \textbf{stable MV cycle} of weight $-\nu$ to be an
irreducible component of $\overline{S_+^0\cap S_-^{-\nu}}$, and we
denote the set of these cycles by $\mathcal Z(\infty)_{-\nu}$.
We further define the set of all stable MV cycles
$$\mathcal Z(\infty)=\bigsqcup_{\nu\in Q_+}\mathcal Z(\infty)_{-\nu}.$$

Let $\lambda\in P_+$. By \cite[Proposition~3]{Anderson03},
for any weight $\mu\in P$, the irreducible components of
$\overline{\Gr^\lambda\cap S_-^\mu}$ are the irreducible components
of $\overline{S_+^\lambda\cap S_-^\mu}$ that are contained in
$\overline{\Gr^\lambda}$. Additionally, the action of $t^\lambda$ on
$\Gr$ induces an isomorphism $\overline{S_+^0\cap S_-^{-\nu}}\xrightarrow
\simeq\overline{S_+^\lambda\cap S_-^{\lambda-\nu}}$. It follows that
the assignment $Z\mapsto t^\lambda Z$ provides a bijection
$$\bigl\{Z\in\mathcal Z(\infty)\bigm|t^\lambda Z\subseteq
\overline{\Gr^\lambda}\bigl\}\quad\xrightarrow\simeq\quad\mathcal Z(\lambda).$$

Recall that for each $\lambda\in P_+$, we defined in \S
\ref{ss:MVCycles} the MV basis $\{[Z]\mid Z\in\mathcal Z(\lambda)\}$
of the representation $L(\lambda)$.
\begin{proposition}
\label{pr:MVCohFam}
The MV bases of the
simple representations $L(\lambda)$ form a coherent family of perfect
bases in the sense of Definition~\ref{de:CohFamBas}.
More precisely, for each $Z\in\mathcal Z(\infty)$, there exists a unique element
$b_Z\in\ON$ such that for any $\lambda\in P_+$, we have
$$t^\lambda Z\subseteq\overline{\Gr^\lambda}\ \Longrightarrow\
b_Z=\Psi_\lambda\bigl([t^\lambda Z]\bigr).$$
\end{proposition}

\begin{proof}
For each $Z\in\mathcal Z(\infty)$, there exists $\lambda\in P_+$ such that
$t^\lambda Z\subseteq\overline{\Gr^\lambda}$
(\cite[Propositions~4 and~7]{Anderson03}).
The crux of the proof is to show that $\Psi_\lambda\bigl([t^\lambda Z]\bigr)$
does not depend on the choice of $\lambda$.

Let $\lambda,\mu\in P_+$. The orbit $N^\vee_+(\mathcal O)\cdot L_\lambda$
is dense in $\Gr^\lambda$ (\cite[proof of Theorem~3.2]{MirkovicVilonen}), so
$$
t^\mu\cdot\overline{\Gr^\lambda}
\ =\
t^\mu\,\overline{N^\vee_+(\mathcal O)\cdot L_\lambda}\ \subseteq \
\overline{N^\vee_+(\mathcal O)\,t^\mu\cdot L_\lambda}\ =\
\overline{N^\vee_+(\mathcal O)\cdot L_{\lambda+\mu}} \ =\
\overline{\Gr^{\lambda+\mu}}.$$
As a consequence,
$$\bigl\{Z\in\mathcal Z(\infty)\bigm|t^\lambda Z\subseteq
\overline{\Gr^\lambda}\bigl\}\,\subseteq
\bigl\{Z\in\mathcal Z(\infty)\bigm|t^{\lambda+\mu}Z\subseteq
\overline{\Gr^{\lambda+\mu}}\bigl\},$$
which shows that the assignment $Z\mapsto t^\mu Z$ defines an injection
$\mathcal Z(\lambda)\hookrightarrow\mathcal Z(\lambda+\mu)$.

Let $\psi:L(\lambda)\to L(\lambda+\mu)$ be the linear extension of this
injection; in other words, $\psi$ is the linear map that sends an
element $[Z]$ of the MV basis of $L(\lambda)$ to the element $[t^\mu Z]$
of the MV basis of $L(\lambda+\mu)$. By construction, $\psi$ raises the
weight by $\mu$ and maps $v_\lambda$ to~$v_{\lambda+\mu}$. We claim that
it intertwines the actions of $N$ on $L(\lambda)$ and $L(\lambda+\mu)$.

To see this, let $ \nu \in P $, let $f$ be a linear form on $V$ which is nonzero on the
line $\Upsilon(L_\nu)$ and which vanishes on all weight subspaces of $V$
of weight other than $\pi(\nu)$, and let $D\subseteq\Gr$ be the Cartier
divisor defined as the intersection of $\Gr$ with the hyperplane in
$\mathbb P(V)$ of equation~$f$. Then the linear form $g$ on $V$ defined
by $g(v)=f(t^{-\mu}v)$ is nonzero on the line $\Upsilon(L_{\nu+\mu})$ and
vanishes on all weight subspaces of $V$ of weight other than
$\pi(\nu+\mu)$, and $t^\mu D$ is the Cartier divisor of equation $g$.   Thus for any $ Z \in \mathcal Z(\lambda)_\nu $ and $ Z' \in \mathcal Z(\lambda)_{\nu + \alpha_i} $, we have
$$ i(t^\mu Z', (t^\mu D) \cdot (t^\mu Z)) =i(Z',  D \cdot  Z), $$
which implies  (Theorem~\ref{th:ActEiMV}) that $ \psi $ intertwines
the actions of each $ e_i $.

Thus, the map
$$(\Psi_\lambda-\Psi_{\lambda+\mu}\circ\psi):L(\lambda)\to\ON$$
is an homomorphism of $N$-modules, which lowers the weight by $\lambda$
and annihilates $v_\lambda$. Its image is therefore an
$N$-invariant subspace of the augmentation ideal
$$\bigoplus_{\nu\in Q_+\setminus\{0\}}\ON_{-\nu}
=\bigl\{f\in\ON\bigm|f(1_N)=0\bigr\}$$
of $\ON$. Consequently, this image is zero, and therefore
$\Psi_\lambda=\Psi_{\lambda+\mu}\circ\psi$. We conclude that for any
$Z\in\mathcal Z(\infty)$ satisfying
$t^\lambda Z\subseteq\overline{\Gr^\lambda}$, we have
$$\Psi_\lambda\bigl([t^\lambda Z]\bigr)=
\Psi_{\lambda+\mu}\bigl([t^{\lambda+\mu}Z]\bigr).$$
\end{proof}

Thus, the elements $b_Z$ constructed in Proposition~\ref{pr:MVCohFam}
form a perfect basis of $\ON$, which we call the MV basis of $\ON$, for it
is obtained by gluing the MV bases of the representations $L(\lambda)$.
By Proposition~\ref{pr:PerfBasTrf}, the crystal structure on the indexing
set $\mathcal Z(\infty)$ can be characterized by its restrictions to the
sets $\mathcal Z(\lambda)$, where it must coincide (up to a shift in the
weight map) with the structure that we used in \S \ref{se:MVBasisRepr}.
Comparing the constructions in \cite{BravermanGaitsgory} and
 \cite{BravermanFinkelbergGaitsgory}, we see that this crystal
 structure is the one defined in this latter reference.

\subsection{Biperfectness}
\label{ss:MVBasisBiperf}
Our aim in this section is to prove the following result.
\begin{theorem}
\label{th:MVBasisBiperf}
The MV basis of $\ON$ is biperfect.
\end{theorem}

As a first step in the proof, we need to endow $\mathcal Z(\infty)$ with
operators $(\varepsilon_i^*,\varphi_i^*,\tilde e_i^*,\tilde f_i^*)$ that
provide the structure of a bicrystal. We do this with the help of a
weight preserving involution, which we construct as follows.

We choose an involutive antiautomorphism $g\mapsto g^\tau$ of $G^\vee$
that fixes pointwise the torus $T^\vee$ and exchanges any root subgroup
with its opposite root subgroup. The automorphism
$g\mapsto(g^\tau)^{-1}$ of the group $G^\vee(\mathcal K)$ leaves
$G^\vee(\mathcal O)$ stable, hence induces an automorphism of $\Gr$
which we denote by $x\mapsto x^\dagger$. It is easy to see that $L_\mu^\dagger = L_{-\mu} $ and that
$(S^\mu)^\dagger=S_-^{-\mu}$ for each weight $\mu\in P$. Given $\nu\in Q_+$
and $Z\in\mathcal Z(\infty)_{-\nu}$, we set $Z^*=(t^\nu Z)^\dagger$,
another element in $\mathcal Z(\infty)_{-\nu}$. The map $Z\mapsto Z^*$
is involutive. We can now set
$$\varepsilon_i^*(Z)=\varepsilon_i(Z^*),\quad
\varphi_i^*(Z)=\varphi_i(Z^*),\quad
\tilde e_i^*Z=(\tilde e_i(Z^*))^*,\quad
\tilde f_i^*Z=(\tilde f_i(Z^*))^*$$
for each $i\in I$ and $Z\in\mathcal Z(\infty)$.

\begin{remark}
\label{rk:RefBFG}
The involution $Z\mapsto Z^*$ on the set $\mathcal Z(\infty)$
was first considered by Braverman, Finkelberg and
Gaitsgory~\cite{BravermanFinkelbergGaitsgory}, who show that
$\mathcal Z(\infty)\cong B(\infty)$ as bicrystals.
Theorem~\ref{th:MVBasisBiperf} provides an independent proof of the
existence of such an isomorphism.
\end{remark}

We extend the assignment $b_Z\mapsto b_{Z^*}$ to a linear bijection
of the vector space $\ON$, which we denote by $f\mapsto f^*$.

\begin{lemma}
\label{le:UEquivar}
Let $i\in I$ and let $u$ be the endomorphism of the vector space $\ON$
such that $u(f)=(e_i\cdot(f^*))^*$ for each $f\in\ON$. Then $u$ commutes
with the left action of $\mathfrak n$ on $\ON$.
\end{lemma}

\begin{proof}
We recall the formula that gives the left action of the Chevalley
generators $e_j$ on the basis elements of $\ON$. Fix $\nu\in Q_+$, pick
a linear form $f$ on the vector space $V$ which is nonzero on the line
$\Upsilon(L_{-\nu})$ and which vanishes on all weight subspaces of $V$
of weight other than $\pi(-\nu)$, and define $D$ to be the Cartier divisor
defined as the intersection of $\Gr$ with the hyperplane in $\mathbb P(V)$
of equation $f$. Then for each $j\in I$ and each
$Z\in\mathcal Z(\infty)_{-\nu}$ we have
\begin{equation} \label{eq:ActionStableMVcycles}
q(\alpha_j)e_j\cdot b_Z=\sum_{Z'\in\mathcal Z(\infty)_{-\nu+\alpha_j}}
i(Z',D\cdot Z)\,b_{Z'}.
\end{equation}

The definition of the involution $*$ leads to a similar formula for
the endomorphism $u$. As a matter of fact, for each
$Z\in\mathcal Z(\infty)_{-\nu}$ and
$Z'\in\mathcal Z(\infty)_{-\nu+\alpha_i}$, we have
$$i((Z')^*,D\cdot Z^*)=
i\bigl((t^{\nu-\alpha_i}Z')^\dagger,D\cdot(t^\nu Z)^\dagger;\Gr\bigr)=
i\bigl((t^{-\alpha_i}Z')^\dagger,(t^\nu D)\cdot Z^\dagger;\Gr\bigr).$$
Defining a Cartier divisor $D'$ on $\Gr$ by the formula
$D'=(t^\nu D)^\dagger$, we then get
$$q(\alpha_i)u(b_Z)=\sum_{Z'\in\mathcal Z(\infty)_{-\nu+\alpha_i}}
i(t^{-\alpha_i}Z',D'\cdot Z)\,b_{Z'}.$$

Using \cite[Corollary~2.4.2]{Fulton}, we obtain
$$e_j\cdot u(b_Z)=\frac1{q(\alpha_i)q(\alpha_j)}\;
\sum_{Z'\in\mathcal Z(\infty)_{-\nu+\alpha_i+\alpha_j}}
i(t^{-\alpha_i}Z',D\cdot D'\cdot Z)\,b_{Z'}=
u(e_j\cdot b_Z)$$
for each $j\in I$ and each $Z\in\mathcal Z(\infty)_{-\nu}$.

\end{proof}

\begin{remark}
\label{rk:GeomInterpAction}
The left action of the Chevalley generator $e_j$ on a basis element $b_Z$
is obtained by intersecting $Z$ with the divisor $D$ so as to jettison the
``bottom'' part of~$Z$. Similarly, the action of $u$ on $b_Z$ amounts to
intersecting $Z$ with the divisor $D'$ so as to jettison the ``top'' part of
$Z$. Thus, the lemma merely reformulates in a representation theoretic
language the general fact that $D\cdot(D'\cdot Z)=D'\cdot(D\cdot Z)$.
\end{remark}

\begin{proposition}
The involution $f\mapsto f^*$ of $\ON$ exchanges the left
and the right actions of the Chevalley generators $e_i$.
\end{proposition}
\begin{proof}
Let $i\in I$ and let $u$ be the endomorphism of $\ON$ defined in
Lemma~\ref{le:UEquivar}. By construction, $u(\ON_{-\nu})\subseteq
\ON_{-\nu+\alpha_i}$ for each $\nu\in Q_+$. Therefore the dual
of $u$ can be restricted to an endomorphism $v$ of the graded dual
of $\ON$, namely $U(\mathfrak n)$, and $v$ is of degree $\alpha_i$.
Lemma~\ref{le:UEquivar} implies that $v$ commutes with the right action
of $\mathfrak n$ on $U(\mathfrak n)$, so $v$ is the left multiplication
by an element of $U(\mathfrak n)$. For degree reasons, this element is
of the form $\lambda e_i$ with $\lambda\in\mathbb C$. Thus, $u$ is the
right action of $\lambda e_i$ on $\ON$.

The set $\mathcal Z(\infty)_{-\alpha_i}$ contains just one element
--- indeed $\overline{S_+^0\cap S_-^{-\alpha_i}}$ is a Riemann sphere, hence is
irreducible. Denote this element by $Z_i$; it is fixed by the
involution $*$. Then the basis element $b_{Z_i}$ spans the weight
subspace $\ON_{-\alpha_i}$, and by construction of the pairing between
$U(\mathfrak n)$ and $\ON$, we have $e_i\cdot b_{Z_i}=b_{Z_i}\cdot e_i$
(see \S \ref{ss:Notation}; in fact, this is the constant function
on $N$ equal to~$1$). Now on the one hand we have
$u(b_{Z_i})=b_{Z_i}\cdot(\lambda e_i)$, and on the other hand, since
$b_{Z_i}=(b_{Z_i})^*$, we have
$$u(b_{Z_i})=(e_i\cdot b_{Z_i})^*=e_i\cdot b_{Z_i}.$$
Therefore $\lambda=1$, and we conclude that $u$ is the right action of
$e_i$ on $\ON$.

Thus, the involution $f\mapsto f^*$ exchanges the left and the right
actions of each Chevalley generator $e_i$.
\end{proof}

\begin{proof}[Proof of Theorem~\ref{th:MVBasisBiperf}]
Our constructions ensure that the involution $f\mapsto f^*$ exchanges
the crystal structures $(\wt,\varepsilon_i,\varphi_i,\tilde e_i,\tilde f_i)$
and $(\wt,\varepsilon_i^*,\varphi_i^*,\tilde e_i^*,\tilde f_i^*)$, as well
as the left and right actions of the Chevalley generators. Since the MV
basis of $\ON$ is perfect, it is biperfect.
\end{proof}

\begin{remark}
\label{re:StarInv}
The involution $f\mapsto f^*$ of $\ON$ is dual to the involutive
algebra antiautomorphism of $U(\mathfrak n)$ that fixes the Chevalley
generators $e_i$. One can also easily show that $f\mapsto f^*$ is the
automorphism of the algebra $\ON$ induced by the automorphism
$n\mapsto(-1)^{\rho^\vee}n^{-1}(-1)^{-\rho^\vee}$ of the variety $N$,
where $(-1)^{\rho^\vee}$ is the evaluation at $-1$ of the cocharacter
$\rho^\vee$ of $T$.
\end{remark}

\subsection{MV polytopes from MV cycles}\label{ss:MVfromMV}
To each $T$-invariant closed subvariety $ Z\subset \Gr $ (for example an MV cycle or stable MV cycle), we define
$$ \Pol(Z) := \Conv \{ \mu \mid L_\mu \in Z \}$$
For any such $ Z $ and any $ \gamma \in P $, note that $ \Pol(t^\gamma Z) = \gamma + \Pol(Z) $.  

The MV basis is biperfect (Theorem~\ref{th:MVBasisBiperf}) and is indexed by the set $\mathcal Z(\infty)$ of stable MV cycles, so by \S\ref{ss:UniqCrys} we get a canonical bijection $\mathcal Z(\infty)\cong B(\infty)$.  The construction in \S \ref{ss:MVpolytopes} allows us to represent elements in $B(\infty) $ by MV polytopes.  Thus, we get a resulting bijection from $\mathcal Z(\infty)$ onto the set of MV polytopes.  By Remark \ref{rem:MVpoly} and Theorem 4.7 from \cite{mvcrystal}, this bijection is given by $ Z \mapsto \Pol(Z) $.

\section{Multiplication}
\subsection{Generalities on cosheaves}\label{ss:cosheaves}
In this section, we will work with cosheaves, which we define as Verdier duals of sheaves\footnote{The reader may find this terminology a bit strange since this differs from the more common use of ``cosheaf''.  We justify our use of this word in two ways.  First, $ i_x^! \mathcal A$ is often called ``costalk'' in the literature.  Second, our category of cosheaves is equivalent to the usual category of cosheaves by Proposition 7.13 of \cite{Curry}.}.  Let $ X $ be an irreducible complex algebraic variety of dimension $ d$.  

\begin{definition}
\begin{itemize}
\item A {\bf (constructible) cosheaf} on $ X $ is an object $\mathcal A $ of $ D_c(X) $ which is isomorphic to $ \D(\mathcal F) $ for some sheaf $ \mathcal F $.
\item The {\bf costalk} of a cosheaf $ \mathcal A $ at a point $ x $ is the vector space $ i_x^! \mathcal A $.
\item A cosheaf $\mathcal A $ is {\bf coconstant} if it is isomorphic to the Verdier dual of a constant sheaf.  It is {\bf coconstant along} $ V$ if $ i_V^!(\mathcal A) $ is coconstant as a cosheaf on $ V $.  (Here $ V \subset X $ is a constructible subset and $ i_V :V \rightarrow X $ is the inclusion.)
\end{itemize}
\end{definition}

The cosheaves on $ X $ form an abelian category $CoSh_c(X)$ which is the heart of a $t$-structure on $ D_c(X) $; this $t$-structure is obtained by applying Verdier duality to the standard $t$-structure.  We have cohomology functors $ \hH^k : D_c(X) \rightarrow CoSh_c(X) $ which are defined by $ \hH^k(\mathcal A) = \D(\mathcal H^{-k}(\D(\mathcal A))) $,  where $ \mathcal H^k $ denotes the usual cohomology functor.

For any morphism $ f : X \rightarrow Y $,  $f^! $ is exact with respect to this $t$-structure and thus $ f^! $ commutes with the cohomology functors $ \hH^k $.  This also implies that for any cosheaf $ \mathcal A $, the costalk $ i_x^! \mathcal A $ is just a vector space (rather than a complex).

Cosheaves are very useful when studying the homology of fibres of a
morphism of varieties.  

\begin{lemma} \label{le:BMfibres}
Let $ \pi : Y \rightarrow X $ be a morphism.  For any $k \in \Z$,  define a cosheaf $ \mathcal G $ on $ X $ by
$$
\mathcal G := \hH^{-k}(\pi_* \D_Y)
$$
Then for any $ x \in X $, $i_x^! \mathcal G \cong H_k(\pi^{-1}(x)) $.
\end{lemma}

\begin{proof}
We use base change and exactness of $ i_x^! $ (with respect to the cosheaf $ t$-structure) to see that
$$
i_x^!(\mathcal G) = \hH^{-k}( i_x^! \pi_* \D_X) = H^{-k}( \pi^{-1}(x), \D_{\pi^{-1}(x)}) = H_k(\pi^{-1}(x))
$$
\end{proof}

\subsubsection{Cosheaves on the line}
We will particularly be working with cosheaves on $ \A = \A^1 $ which are coconstant along $ U = \A \smallsetminus \{0 \} $.  Let $ i = i_0 : \{0\} \rightarrow \A $ and $ j : U \rightarrow \A $ be the inclusions of the origin and its complement.  We will also fix a point $ x \in U $ and write $ i_x : \{x \} \rightarrow \A $ for the inclusion.

Let $ u \in H_{\{0\}}^2(\A) $ be the usual relative orientation class.  As above, we will think of $ u $ as a map $ i_* \C_{\{0\}} \rightarrow \C_\A[2] $. For any vector space $ V $, we have a constant sheaf $V_\A $ with stalks $ V $, and the cup product gives an isomorphism
$$
V = H_c^0(\{0\}, V_{\{0\}}) \xrightarrow{u \cup} H_c^2(\A, {V}_\A).
$$

\begin{proposition} \label{pr:SheavesLine}
Let $ \mathcal F $ be a sheaf on $ \A $ which is constant along $U$.
\begin{enumerate}
\item We have isomorphisms
$$ i_x^* \mathcal F \cong \mathcal F(U), \ i_0^* \mathcal F \cong \mathcal F(\A)$$
and hence we have a restriction map $ r : i_0^* \mathcal F \rightarrow i_x^* \mathcal F$.
\item There is an isomorphism $ i_x^* \mathcal F \cong H_c^2(\A, \mathcal F) $ making the diagram
\begin{equation} \label{eq:squarewant}
\xymatrix{
i^*_0 \mathcal F \ar^{r}[r] \ar^\cong[d] & i_x^* \mathcal F \ar^\cong[d]\\
H^0_c(\{0\}, i^*_0 \mathcal F) \ar[r]^-{u \cup} & H^2_c(\A, \mathcal F)
}
\end{equation}
commute.
\end{enumerate}
\end{proposition}

\begin{proof}
For the purposes of this proof, let $ V = i^*_0 \mathcal F $ and $ W = i_x^* \mathcal F $.

Part (i) is immediate; in fact, for any connected open set $ U \subset \A $, we have
$$
\mathcal F (U) = \begin{cases} V \text{ if $ 0 \in U $ } \\ W \text{ if $ 0 \notin U $ } \end{cases}
$$
In particular $ \mathcal F (\A) = V $ and $ \mathcal F (U) = W $ and so by restriction, we have a linear map $ r : V \rightarrow W $.

For part (ii), consider the standard short exact sequence of sheaves
$$
0 \rightarrow  j_! j^! \mathcal F \rightarrow \mathcal F \rightarrow i_* i^* \mathcal F \rightarrow 0
 $$
which in our case becomes
$$
0 \rightarrow j_! {W}_{U} \rightarrow \mathcal F \rightarrow i_* {V}_{\{0\}} \rightarrow 0
$$
Since $ i $ is closed embedding, $ i_* = i_! $ and so $H_c^k(\A,  i_* {V}_{\{0\}} ) =  H_c^k(\{0\}, V_{\{0\}}) = 0 $ for $ k > 0 $.  Thus we get an isomorphism
$$ W = W \otimes H_c^2(U) =  H_c^2(U, {W}_U) = H_c^2(\A, j_! {W}_U) \cong H_c^2(\A, \mathcal F)
$$

Now that we have an isomorphism $ i_x^* \mathcal F \cong H^2_c(\A, \mathcal F)$, it remains to verify the commutativity of the square (\ref{eq:squarewant}).  To that end, consider the map of sheaves $ {V}_\A \rightarrow \mathcal F $, which is the identity on connected open sets containing 0 and the map $ r $ on connected open sets not containing 0.  By functoriality, we obtain a commutative diagram
\begin{equation*}
\xymatrix{
j_! {V}_U = j_! j^! {V}_\A \ar[r] \ar^r[d]&  {V}_\A \ar[d] \\
j_! {W}_U = j_! j^! \mathcal F \ar[r] &  \mathcal F
}
\end{equation*}
 Applying $H_c^2(\A, -) $, we obtain the commutative square
\begin{equation} \label{eq:square}
\xymatrix{
V = H_c^2(U, {V}_U) \ar^=[r] \ar^r[d] & V = H_c^2(\A, {V}_\A) \ar[d] \\
 W =  H_c^2(U, {W}_U) \ar[r]^-\cong & H_c^2(\A, \mathcal F)
}
\end{equation}

Now by the naturality of cup product, we get a commutative square
\begin{equation*}
\xymatrix{
i_* i^* V_\A \ar^{u \cup}[r] \ar^\cong[d] & V_\A[2] \ar[d] \\
i_* i^* \mathcal F \ar^{u \cup}[r] & \mathcal F[2].
}
\end{equation*}
Applying $H^0_c(\A, -) $, we obtain a factoring of $ u \cup $  as
$$
V = H^0_c(\{0\}, i^*_0 \mathcal F) \xrightarrow{\cong} H^2_c(\A, {V}_\A) \rightarrow H^2_c(\A, \mathcal F)
$$
Combined with the square (\ref{eq:square}), the result follows.
\end{proof}

Applying Verdier duality implies a similar result for cosheaves on $ \A $.
\begin{proposition} \label{pr:CosheavesLine}
Let $ \mathcal G $ be a cosheaf on $ \A $ which is coconstant along $U$.
\begin{enumerate}
\item We have isomorphisms
$$ i_x^! \mathcal G \cong H^0_c(U, i_U^!\mathcal G), \ i_0^! \mathcal G \cong H^0_c(\A, \mathcal G)$$
and a corestriction map $ r^\vee : i_x^! \mathcal G \rightarrow i_0^! \mathcal G$.
\item There is an isomorphism $  i_x^! \mathcal G \cong H^{-2}(\A, \mathcal G)  $
making the diagram
\begin{equation*}
\xymatrix{
i^!_x \mathcal G \ar^{r^\vee}[r] \ar^\cong[d] & i_0^! \mathcal G \ar^\cong[d]\\
H^{-2}(\A, \mathcal G) \ar^{u \cap}[r] & H_{\{0\}}^0(\A, \mathcal G)
}
\end{equation*}
commute.
\end{enumerate}
\end{proposition}

Now, we will see how we can use these results to describe degeneration of cycles in Borel--Moore homology.

Let $ f : Y \rightarrow \A $ be a morphism of varieties.  Fix an integer $ n$ such that all fibres of $ f $ have dimension $ \le n - 1 $ (usually $ n = \dim Y $, but not necessarily).  We write $Y_0 = f^{-1}(0), Y_x = f^{-1}(x), Y_U = f^{-1}(U) $.  Assume also that we are given an isomorphism $ Y_U \cong Y_x \times U $ compatible with the projection to~$U $.
Let $ \mathcal G = \hH^{-2n+2}(f_* \D_Y) $.

\begin{proposition} \label{pr:CosheafFamily}
We have the following results concerning $ \mathcal G $.
\begin{enumerate}
\item $\mathcal G $ is coconstant along $ U $.
\item We have isomorphisms $ i_x^! \mathcal G \cong H_{2n-2}(Y_x) $ and $ i_0^! \mathcal G \cong H_{2n-2}(Y_0) $.
\item For $ p < -2n+2$, we have $ \hH^p (f_* \D_Y) = 0 $ and thus we have a morphism $ \mathcal G \rightarrow f_* \D_Y[-2n+2] $.  This morphism induces an isomorphism $$ H^{-2}(\A, \mathcal G) \xrightarrow{\cong} H^{-2}(\A, f_* \D_Y[-2n+2]) =  H^{-2n}(Y, \D_Y) = H_{2n}(Y) $$

\item  With respect to  the isomorphism $H^{-2}(\A, \mathcal G) \cong H_{2n}(Y)$ and the isomorphism $ i_x^! \mathcal G \cong H_{2n-2}(Y_x) $, the isomorphism $  i_x^! \mathcal G \cong H^{-2}(\A, \mathcal G) $ from Proposition \ref{pr:CosheavesLine}(ii) is given on cycles as $ [Z] \mapsto [\overline{Z \times U}] $.
 \item
 The corestriction map is compatible with the cap product in homology as follows
 \begin{equation*}
 \xymatrix{
 & i_x^! \mathcal G \ar^{r^\vee}[r] \ar_{\cong}[dl] &  i_0^! \mathcal G \ar^\cong[d] \\
 H^{-2}(\A, \mathcal G) \ar^{\cong}[r] & H_{2n}(Y) \ar[r]^-{(f^*u) \cap }   & H_{2n-2}(Y_0)
 }
\end{equation*}
\end{enumerate}
\end{proposition}

\begin{proof}
\begin{enumerate}
\item Since $j^! $ is exact on the cosheaf $t$-structure, $ j^! \mathcal G = \hH^{-2n+2}(f_* \D_{Y_U})  $ and because $Y_U \cong Y_x \times U $, we see that $ j^! \mathcal G  = \D_U \otimes H_{2n-2}(Y_x) $ is coconstant.
\item These isomorphisms follow from Lemma \ref{le:BMfibres}.
\item For any $ p$ and any $ y \in \A $, we have $ i_y^! \hH^{p} (f_* \D_Y) = H_{-p}(f^{-1}(y)) $.  Since the fibre dimension is $ n - 1$, we see that this homology vanishes if $ p < -2n+2 $ and thus $ \hH^p (f_* \D_Y) = 0 $ for such $ p $.  Thus we deduce the existence of a morphism $  \mathcal G \rightarrow f_* \D_Y[-2n+2] $.  We have a spectral sequence starting with $H^q(\A, \hH^p(f_* \D_Y)) $ and converging to $ H^{p+q}(Y, \D_Y) $.  Since $ H^q(\A, \hH^p(f_* \D_Y)) $ vanishes for $ q < -2 $ and for $ p < -2n+2 $, we see that the only contribution to $ H^{-2n}(Y, \D_Y) $ can come from $ H^{-2}(\A, \mathcal G) $ and thus  $ H^{-2}(\A, \mathcal G) \cong H_{2n}(Y) $.
\item We begin by repeating part of the proof of Proposition \ref{pr:SheavesLine} in the Verdier dual setting.  Consider the morphism $ \mathcal G \rightarrow j_* j^* \mathcal G $.  As $ j^* = j^! $, we obtain the morphism $ \mathcal G \rightarrow j_* \D_U \otimes H_{2n-2}(Y_x) $.  Applying $H^{-2}(\A, -) $, we obtain the isomorphism
\begin{equation} \label{eq:homfibre}
H^{-2}(\A, \mathcal G) \cong H^{-2}(\A, j_* \D_U \otimes H_{2n-2}(Y_x) ) = H_{2}(U) \otimes H_{2n-2}(Y_x) = H_{2n-2}(Y_x)
\end{equation}
Now we have the commutative square
\begin{equation*}
\xymatrix{
\mathcal G \ar[r] \ar[d] & j_*j^* \mathcal G = j_* \D_U \otimes H_{2n-2}(Y_x) \ar[d] \\
f_* \D_Y [-2n+2] \ar[r] & j_* j^* f_* \D_Y[-2n+2] = j_* f_* \D_{Y_U}[-2n+2]
}
\end{equation*}
Applying $H^{-2}(\A, -) $, we obtain the commutative rectangle of isomorphisms
\begin{equation*}
\xymatrix{
H^{-2}(\A, \mathcal G) \ar[r] \ar[d] & H^{-2}(\A,  j_* \D_U \otimes H_{2n-2}(Y_x)) \ar[r]^-\cong &  H_{2n-2}(Y_x) \ar[d] \\
H^{-2n}(Y, \D_Y) = H_{2n}(Y) \ar[r] & H^{-2n}(Y_U, \D_{Y_U}) = H_{2n}(Y_U) \ar[r]^-\cong  & H_{2n}(Y_x \times U)
}
\end{equation*}
In this diagram, the bottom left horizontal arrow is given by pullback in Borel--Moore homology, so on cycles it is given by $ [X] \mapsto [X \cap Y_U]$.  The right vertical arrow is given on cycles by $ [Z] \mapsto [Z \times U] $.  Thus if we trace from the top right to the bottom left (following the inverses of the arrows along the bottom), we see the map $ [Z] \mapsto [\overline{Z \times U}]$.

 \item We apply the projection formula (\ref{eq:projection}) to the sheaf $ \D_Y $.  We obtain the commutative square.

   \begin{equation}  \xymatrix{
  H_{2n}(Y) = H^{-2n}(Y, \mathcal \D_Y) \ar[r]^-{(f^*u) \cap }  \ar^\cong[d] & H^{-2n+2}_{Y_0}(Y, \mathcal \D_Y) = H_{2n-2}(Y_0) \ar^\cong[d] \\
   H^{-2}(\A, \mathcal G) \cong H^{-2n}(\A, f_* \D_Y) \ar[r]^-{u \cap }  & H^{-2n+2}_{\{0\}}(\A, f_* \mathcal \D_Y) \cong H_{\{0\}}^0(\A,\mathcal G) \\
   }
  \end{equation}
 The result now follows from part (ii) of Proposition \ref{pr:CosheavesLine}.
 \end{enumerate}
\end{proof}

\subsection{The Beilinson--Drinfeld Grassmannian}\label{ssec:BD}

We will need to recall the definition and properties of the Beilinson--Drinfeld Grassmannian.  All the definitions and results here are taken from \cite{MirkovicVilonen}.  The original definitions are due to Beilinson--Drinfeld \cite{BeilinsonDrinfeld}.

For any $ x \in \C $, let $\mathcal{O}_x = \C[[t-x]]$ denote the completion of the local ring of $ \A = \A^1 $ at $ x $ and $\K_x = \C((t-x))$ its fraction field.  By the Beauville--Laszlo Theorem, the $\C$-points of the affine Grassmannian $ \Gr_x := G(\K_x) / G (\mathcal{O}_x) $ can be identified with the set of pairs $ (P, \sigma) $ where $ P $ is a principal $ G $-bundle on $ \A^1 $ and $ \sigma $ is a trivialization of $ P $ over $ \A \smallsetminus \{x\} $.

Since we have chosen the local coordinate $ t -x $, we get an isomorphism $ \mathcal{O}_x \rightarrow \mathcal{O} $ and thus $ \Gr_x \rightarrow \Gr $.

Now, we consider a two-point version of the Beilinson--Drinfeld Grassmannian.  For simplicity, we will fix one point to be $ 0 $ and allow the other point to vary.   The set of $ \C$-points of $ \Gr_\A $ is given by
\begin{align*}
\Gr_\A = \{ (P, \sigma, x) :  &\ P  \text{ is a principal } G \text{ bundle on } \A ,  x \in \A,\\
 &\text { and }  \sigma \text{ is a trivialization of } P \text{ over } \A \smallsetminus \{ 0, x \} \}
\end{align*}
(See 
\cite[(5.1)]{MirkovicVilonen} for a precise description of the $ R$-points of $ \Gr_\A $ for each $\C$-algebra $ R $.)

 We have an obvious map $ \pi: \Gr_{\A} \rightarrow \A$.  Let $ U = \A^1 \smallsetminus \{0 \} $.  The following is diagram (5.9) in \cite{MirkovicVilonen}.
\begin{lemma}
There are isomorphisms
$$ \pi^{-1}(U) = \Gr_U \cong \Gr \times \Gr \times U, \quad \pi^{-1}(0) \cong \Gr $$
\end{lemma}

The action of $ T $ by left multiplication on $ \Gr $ extends to an action on $ \Gr_{\A} $ preserving all fibres.  From the fibre perspective, this is simply the diagonal action on $ \Gr \times \Gr $.

Following \cite{MirkovicVilonen}, we introduce a global version of the semi-infinite cells.  For $ \mu \in P $, define $ S^\mu_{-,\A} $ to be the subvariety of $ \Gr_{\A} $ with fibres $ S_-^\mu $ over $0$ and fibres $ \displaystyle{\cup_{\mu_1 + \mu_2 = \mu} S_-^{\mu_1} \times S_-^{\mu_2}} $ over $ x \in U $.


We write $ s_\mu $ for the inclusion of $ \overline{S^\mu_{-,\A}} $ into $\Gr_{\A}$.

We will also need the global version of the $\GO $ orbits. From \cite{Zhu}, for any pair $ \lambda_1, \lambda_2 \in P_+$, there exists a variety $ \overline{\Gr^{\lambda_1, \lambda_2}_\A} \subset \Gr_\A  $, whose fibres are $ \overline{\Gr^{\lambda_1}} \times \overline{\Gr^{\lambda_2}} $ away from $ 0 $ and $ \overline{\Gr^{\lambda_1 + \lambda_2}} $ over $ 0 $.

\subsection{The fusion product of perverse sheaves}\label{ss:fusion}
We will now define the fusion tensor product on $ P_{\GO}(\Gr) $ following \cite{MirkovicVilonen}.

Consider the diagram
\begin{equation*}
\xymatrix{
\Gr \ar[r]^i \ar[d] & \Gr_\A \ar[d]^\pi &  \ar[l]_-{j}  \Gr_U \cong \Gr \times \Gr \times U \ar[d] \\
\{0 \} \ar[r] &\A &\ar[l] U
}
\end{equation*}

\begin{definition}
Let $ \mathcal A_1, \mathcal A_2 \in P_{\GO}(\Gr) $.  The {\bf fusion product} (see \cite[(5.8)]{MirkovicVilonen}) is defined by
$$
\mathcal{A}_1 * \mathcal{A}_2 := i^!( j_{!*} (p_{12}^! (\mathcal{A}_1 \boxtimes \mathcal{A}_2)[-1]))[1]
$$
where $ p_{12} : \Gr \times \Gr \times U \rightarrow \Gr \times \Gr $ is the projection onto the first two factors.
\end{definition}

Following \cite{MirkovicVilonen} and \cite[\S  8.3]{BR}, we will explain the compatibility of the fusion product with the weight functors.

Fix $ \mathcal{A}_1, \mathcal{A}_2 \in P_{\GO}(\Gr) $ and let
$$
\mathcal{B} := j_{!*} (p_{12}^! (\mathcal{A}_1 \boxtimes \mathcal{A}_2)[-1])
$$
and for each $ \mu \in P $ let
$$
 \mathcal{F}_\mu  :=  \hH^{2 \rho^\vee(\mu) + 1} \pi_*  s_\mu^! \mathcal{B}
 $$
a cosheaf on $ \A $.  The following result follows by direct computation.

\begin{proposition}
\begin{enumerate}
\item The cosheaf $ \mathcal{F}_\mu $ is coconstant along $ U $ with costalk at $ x \in U $ given by
$$
i_x^! \mathcal F_\mu =  \bigoplus_{\mu_1 + \mu_2 = \mu} F_{\mu_1}(\mathcal A_1) \otimes F_{\mu_2}(\mathcal A_2). $$
\item The costalk at $ 0 $ is given by
$$
i_0^! \mathcal F_\mu = F_\mu(\mathcal A_1 * \mathcal A_2)
$$
\item The cosheaf $ \mathcal F_\mu $ is actually coconstant along all of $ \A $ and so the corestriction map $ r^\vee : i_x^! \mathcal F_\mu \rightarrow i_0^! \mathcal F_\mu$ provides an isomorphism
$$
 \bigoplus_{\mu_1 + \mu_2 = \mu} F_{\mu_1}(\mathcal A_1) \otimes F_{\mu_2}(\mathcal A_2) \rightarrow  F_\mu(\mathcal A_1 * \mathcal A_2)
 $$
\end{enumerate}
\end{proposition}

In fact, this cosheaf $\mathcal F_\mu$ is $\mathcal L[2]$, where $\mathcal L$ is the local system $\mathcal L_\mu^{2\rho^\vee(\mu)}(\mathcal A_1,\mathcal A_2)$ defined in [MV], (6.22), and pulled back to $ \A \subset \A^2 $.

\subsection{The multiplication map}\label{ssec:mult}
Let $ \lambda_1, \lambda_2 $ be two dominant weights and let  $\lambda = \lambda_1 + \lambda_2$.  We have a morphism
$$ m_{\lambda_1 \lambda_2} : \IC_{\lambda_1} * \IC_{\lambda_2} \rightarrow \IC_\lambda $$
which becomes $m_{\lambda_1 \lambda_2} : L(\lambda_1) \otimes L(\lambda_2) \rightarrow L(\lambda) $ under the geometric Satake isomorphism.

Take $(\mathcal A_1,\mathcal A_2) = (\IC_{\lambda_1},\IC_{\lambda_2})$ in the setup above. Note that
$$
\mathcal{B} := j_{!*} (p_{12}^! (\IC_{\lambda_1} \boxtimes \IC_{\lambda_2})[-1])
$$
is actually the IC sheaf of $ \overline{\Gr^{\lambda_1, \lambda_2}_\A}$.

Let $ \mu \in P $. We will compare $ \mathcal F_\mu := \hH^{2 \rho^\vee(\mu) + 1} \pi_*  s_\mu^! \mathcal{B}$ to the cosheaf
$$ \mathcal G := \hH^{2\rho^\vee(\mu) +1} \pi_* s_\mu^! (\D_{\overline{Gr^{\lambda_1, \lambda_2}_\A}}[-2\rho^\vee(\lambda) -1]) = \hH^{-2\rho^\vee(\lambda - \mu)}  \pi_* (\D_{\overline{Gr^{\lambda_1, \lambda_2}_\A} \cap \overline{S^\mu_{-,\A}}}). $$

We can apply Proposition \ref{pr:CosheafFamily} to the map $ \pi : \overline{Gr^{\lambda_1, \lambda_2}_\A} \cap \overline{S^\mu_{-,\A}} \rightarrow \A $ and thus to the sheaf $ \mathcal G $.

So from Proposition \ref{pr:CosheafFamily} (i) and (ii), it follows that $ \mathcal G $ is coconstant along $ U $ and its costalks are as follows
$$ i_0^! \mathcal G \cong H_{2\rho^\vee(\lambda - \mu)}(\overline{\Gr^{\lambda}} \cap \overline{S_-^\mu}) \cong F_\mu(\IC_\lambda) $$
and for $ x \in U $,
\begin{align*}
i_x^! \mathcal G &= \bigoplus_{\mu_1 + \mu_2 = \mu} H_{2\rho^\vee(\lambda_1 - \mu_1)}(\overline{\Gr^{\lambda_1}} \cap \overline{S_-^{\mu_1}}) \otimes H_{2\rho^\vee(\lambda_2 - \mu_2)}(\overline{\Gr^{\lambda_2}} \cap \overline{S_-^{\mu_2}}) \\
&\cong \bigoplus_{\mu_1 + \mu_2 = \mu} F_{\mu_1}(\IC_{\lambda_1}) \otimes F_{\mu_2}(\IC_{\lambda_2})
\end{align*}

\begin{lemma} \label{th:mlem}
The following diagram commutes
\begin{equation*}
\xymatrix@C=60pt@R=30pt{
\displaystyle{\bigoplus_{\mu_1 + \mu_2 = \mu}} F_{\mu_1}(\IC_{\lambda_1}) \otimes F_{\mu_2}(\IC_{\lambda_2}) \ar^\cong[d] \ar[r]^(.66){m_{\lambda_ 1\lambda_2}} & F_\mu(\IC_\lambda) \ar^\cong[d] \\
i_x^! \mathcal G \ar[r]^{r^\vee} & i_0^! \mathcal G
}
\end{equation*}
\end{lemma}

\begin{proof}
From (\ref{eq:ICtoD}), we have a map
\begin{equation} \label{eq:eq1}
\mathcal B \rightarrow \D_{\overline{\Gr^{\lambda_1, \lambda_2}_\A}}[-2\rho^\vee(\lambda) -1].
\end{equation}

Let $ i : \Gr \rightarrow \Gr_\A $ denote the inclusion of the central fibre as before.  Then $ i^! \mathcal B[1] = \IC_{\lambda_1} * \IC_{\lambda_2} $ by definition.  On the other hand, $ \IC_{\lambda_1} * \IC_{\lambda_2} \cong \IC_{\lambda} \oplus \mathcal A $ where $ \mathcal A $ is supported on smaller orbits.  Thus we have a projection $ \IC_{\lambda_1} * \IC_{\lambda_2} \rightarrow \IC_{\lambda} $.  If we apply $ i^! [1]$ to the map (\ref{eq:eq1}), we obtain $
\IC_{\lambda_1} * \IC_{\lambda_2} \rightarrow \D_{\overline{\Gr^{\lambda}}}[-2\rho^\vee(\lambda)]$ which we can factor as
\begin{equation} \label{eq:factor}
\IC_{\lambda_1} * \IC_{\lambda_2} \rightarrow \IC_{\lambda} \rightarrow \D_{\overline{\Gr^{\lambda}}}[-2\rho^\vee(\lambda)]
\end{equation}
because
$$ \Hom(\mathcal A,  \D_{\overline{\Gr^{\lambda}}}[-2\rho^\vee(\lambda)]) = \Hom(\C_{\overline{\Gr^{\lambda}}}[2\rho^\vee(\lambda)], \D(\mathcal A)) = H^{-2\rho^\vee(\lambda)}(\overline{\Gr^{\lambda}}, \D(\mathcal{A})) = 0.
$$
For the last equality, note that $ \D(\mathcal{A}) $ is a direct sum of perverse sheaves, each supported on some $ \overline{\Gr^\mu} \subsetneq \overline{\Gr^{\lambda}} $, and $ \dim \overline{\Gr^\mu} < 2\rho^\vee(\lambda)$.

On the other hand, if we apply $ \hH^{2\rho^\vee(\mu)+1} \pi_*  \circ s_\mu^! $ to (\ref{eq:eq1}) we obtain a map of cosheaves $\mathcal F_\mu \to \mathcal G$, whence a commutative diagram
\begin{equation} \label{eq:eq2}
\xymatrix{
i_x^! \mathcal F_\mu \ar[r]^{r^\vee} \ar[d]&  \ar[d] i_0^! \mathcal F_\mu \\
i_x^! \mathcal G \ar[r]^{r^\vee} & i_0^! \mathcal G.
}
\end{equation}

The factoring (\ref{eq:factor}) means that the map $ i_0^! \mathcal F_\mu \rightarrow i_0^! \mathcal G $ factors as
$$
H_{\overline{S^\mu_-}}^{2\rho^\vee(\mu)}(\Gr, \IC_{\lambda_1} * \IC_{\lambda_2}) \rightarrow H_{\overline{S^\mu_-}}^{2\rho^\vee(\mu)}(\Gr, \IC_{\lambda}) \rightarrow H_{\overline{S^\mu_-}}^{2\rho^\vee(\mu)}(\Gr, \D_{\overline{\Gr^{\lambda}}}[-2\rho^\vee(\lambda)]).
$$

Thus, (\ref{eq:eq2}) can be rewritten as
\begin{equation*}
\xymatrix{
\displaystyle{\bigoplus_{\mu_1 + \mu_2 = \mu}} F_{\mu_1}(\IC_{\lambda_1}) \otimes F_{\mu_2}(\IC_{\lambda_2}) \ar[r] \ar^-\cong[dd]  \ar[dr]_-{m_{\lambda_1 \lambda_2}}& F_\mu(\IC_{\lambda_1} * \IC_{\lambda_2}) \ar[d] \\
& F_\mu(\IC_\lambda) \ar^\cong[d] \\
 i_x^! \mathcal G \ar^{r^\vee}[r]  & i_0^! \mathcal G
}
\end{equation*}
and the result follows.
\end{proof}

\subsection{Multiplication on the level of cycles}

Now we will translate Lemma \ref{th:mlem} to the cycle level.  As before, let $ \lambda_1, \lambda_2 \in P_+ $, let $ \lambda := \lambda_1 + \lambda_2$, let $ \mu_1, \mu_2 \in P $, and let $ \mu := \mu_1 + \mu_2 $.
\begin{lemma}
Let $ Z_1 \in \mathcal Z(\lambda_1)_{\mu_1},Z_2 \in \mathcal Z(\lambda_2)_{\mu_2}$.  Consider $$ Z_1 \times Z_2 \times U \subset \overline{\Gr^{\lambda_1, \lambda_2}_\A} \cap \overline{S^\mu_{-, \A}} $$  \\ Then in
$ F_\mu(\IC_\lambda) = H_{2\rho^\vee(\lambda - \mu)}\left(\overline{\Gr^{\lambda}} \cap \overline{S_-^{\mu}}\right) $,
we have an equality
$$ m_{\lambda_1, \lambda_2}\left([Z_1] \otimes [Z_2]\right) = \sum_Z i \bigl(Z, \pi^{-1}(0) \cdot \overline{Z_1 \times Z_2 \times U}\bigr) [Z]
$$
where the sum ranges over $ Z \in \mathcal Z(\lambda)_\mu$.
\end{lemma}

\begin{proof}
By Proposition \ref{pr:CosheafFamily} (v) and Lemma \ref{th:mlem}, we obtain the commutative diagram
$$
\xymatrix@C=50pt@R=30pt{
\displaystyle{\bigoplus_{\mu_1 + \mu_2 = \mu}} F_{\mu_1}(\IC_{\lambda_1}) \otimes F_{\mu_2}(\IC_{\lambda_2}) \ar[d]^\cong \ar[r]^(.66){m_{\lambda_1 \lambda_2}} & F_\mu(\IC_\lambda) \ar[d]^\cong \\
H_{2\rho^\vee(\lambda - \mu)+2}(\overline{Gr^{\lambda_1, \lambda_2}_\A} \cap \overline{S^\mu_{-,\A}}) \ar[r]^(.52){\pi^*(u) \cap } & H_{2\rho^\vee(\lambda - \mu)}(\overline{Gr^{\lambda_1+ \lambda_2}} \cap \overline{S_-^\mu})
}
$$
where as before $ u \in H^2_{\{0\}}(\A) $ denotes the usual orientation class.  Now, by Proposition \ref{pr:CosheafFamily} (iv), the element $ [Z_1] \otimes [Z_2] $ in the top left is sent to $ [\overline{Z_1 \times Z_2 \times U}] $ in the bottom left.  

Now, we apply the setup from Proposition \ref{pr:Fulton} to $ Y = \mathbb A^1, D = \{0\}, X' = \overline{\Gr^\lambda} \cap \overline{S^\mu_-}, Y' = \overline{Gr_\A^{\lambda_1, \lambda_2}} \cap \overline{S^\mu_{-, \mathbb A}} $.  This gives the desired result.
\end{proof}

By Proposition \ref{pr:mult}, this immediately implies the following result concerning stable MV cycles.

\begin{theorem} \label{th:mult}
	Let $ Z_1 \in \mathcal Z(\infty)_{-\nu_1}, Z_2 \in \mathcal Z(\infty)_{-\nu_2} $. 
In the algebra $ \ON $, we have
$$
b_{Z_1}b_{Z_2} = \sum_{Z} i \bigl(Z, \pi^{-1}(0) \cdot \overline{Z_1 \times Z_2 \times U})  b_Z
$$
where the sum ranges over $ Z \in \mathcal Z(\infty)_{-\nu_1 - \nu_2}$.
\end{theorem}

\part{Measures}
\label{pa:Measures}

\section{Measures}
\label{se:Measures}
All the objects defined in this section depend on the choice of a
principal nilpotent element $\dot e\in\mathfrak n$ and we write $ \dot e = \sum \dot e_i $, where each $ \dot e_i $ is a nonzero root vector of weight $ \alpha_i $.  These $\dot e_i $ are 
 a priori unrelated to the
choice of simple root vectors $e_i$ made in \S \ref{ss:Notation}.

\subsection{The elements $n_x$}
\label{ss:EltsNx}
We denote the set of regular elements in $\mathfrak t$ by $\ftreg$.

For each $x\in\ftreg$, the subset $x+\mathfrak n$ of $\mathfrak g$ is
a single orbit under the adjoint action of the group $N$, by
\cite{Bourbaki}, chap.~8, \S11, no.~1, lemme~2. Further, the centralizer
of $x$ in $G$, namely $T$, meets $N$ trivially, so the action of $N$ on
$x+\mathfrak n$ is simply transitive. Therefore, there is a unique
element $n_x\in N$ such that $\Ad_{n_x}(x)=x+\dot e$. Examining the proof
in~\cite{Bourbaki}, one further notes that $x\mapsto n_x$ is a regular
map $\ftreg\to N$.

We thus get an algebra map
$\barD:\ON\to\mathbb C[\ftreg]$
defined by $\barD(f)(x)=f(n_x)$,
where $f\in\ON$ and $x\in\ftreg$.

A major goal of this section is to understand the map $\barD$ and to put it in a wider setting. We first study how $n_x$
varies when the Weyl group acts on~$x$. We denote by $N_-$ the unipotent
radical of the Borel subgroup opposite to $B$ with respect to $T$.
Recall that $\overline w$ denotes a lift to the normalizer $N_G(T)$ of
an element $w$ in the Weyl group.
\begin{proposition}
\label{pr:BehNx}
Let $x\in\ftreg$ and $w\in W$. Then there exists $(y,t)\in N_-\times T$
such that
$$n_{wx}=y\,n_x\,\overline w^{-1}\,t.$$
\end{proposition}
\begin{proof}
By induction on the length of $w$, we can reduce to the case where $w$
is a simple reflection $s_i$.  Choose $\dot f_i$ such that $(\dot e_i,h_i,\dot f_i)$ is an $\mathfrak{sl}_2$ triple.
Let $x_i$ and $y_i$ be the additive one-parameter subgroups of $G$ given
by $x_i(b)=\exp(b \dot e_i)$ and $y_i(b)=\exp(b\dot f_i)$ for $b\in\mathbb C$. Set
$a=\langle\alpha_i,x\rangle$; there exists an element $t\in T$ such that
$$x_i(1/a)\,y_i(-a)\,x_i(1/a)=\overline s_i\,t.$$

Direct calculations give
\begin{align*}
\Ad_{y_i(a)}(x)&=\exp(a\ad_{\dot f_i})(x)=x+a[\dot f_i,x]=x+a^2\dot f_i,\\
\Ad_{y_i(a)}(\dot e)&=\exp(a\ad_{\dot f_i})(\dot e)=\dot e+a[\dot f_i,\dot e_i]+\frac{a^2}2[\dot f_i,[\dot f_i,\dot e_i]]
=\dot e-ah_i-a^2\dot f_i.
\end{align*}
Noting that $s_ix=x-\langle\alpha_i,x\rangle h_i$, we then get
$$\Ad_{y_i(a)}(x+\dot e)=s_ix+\dot e.$$
Since $t$ acts trivially on $\mathfrak t$ and $\overline s_i$ acts by the
simple reflection $s_i$, we deduce that
\begin{equation}
\label{eq:BehNx}
\Ad_{(y_i(a)\,n_x\,\overline s_i\,t)}
(s_ix)=\Ad_{(y_i(a)\,n_x)}(x)=\Ad_{y_i(a)}(x+\dot e)=s_ix+\dot e.
\end{equation}

On the other hand, let $P'$ be the (minimal parabolic) subgroup of $G$
generated by the Borel $B$ and the image of the one-parameter subgroup
$y_i$. We denote the unipotent radical of $P'$ by $N'$ and the Lie
algebra of $N'$ by $\mathfrak n'$. Noting that
$[x+\dot e_i,\mathfrak n']=\mathfrak n'$ and that $\dot e-\dot e_i\in\mathfrak n'$,
we can apply Bourbaki's lemme~2 quoted above and find $n'\in N'$ such
that $\Ad_{n'}(x+\dot e_i)=x+\dot e$. Since
$$\Ad_{x_i(-1/a)}(x)=x+\dot e_i,$$
we see that the adjoint action of $n'\,x_i(-1/a)$ brings $x$ to $x+\dot e$,
and thus $n_x=n'\,x_i(-1/a)$.

Since $y_i(-a)\,x_i(1/a)$ belongs to $P'$ and hence normalizes $N'$,
we can find $n''\in N'$ such that
$$y_i(a)\,n_x\;\overline s_i\,t
=y_i(a)\,n'\,y_i(-a)\,x_i(1/a)
=y_i(a)\,y_i(-a)\,x_i(1/a)\,n''
=x_i(1/a)\,n''.$$
Thus, the product $y_i(a)\,n_x\;\overline s_i\,t$ belongs
to $N$, and by \eqref{eq:BehNx} it acts on $s_ix$ in the same way as
$n_{s_ix}$. We conclude that $n_{s_ix}=y_i(a)\,n_x\;\overline s_i\,t$,
which is of the desired form.
\end{proof}

\subsection{Sequences and shuffles}\label{ssec:sequences}
Our next task is to find an expansion of $n_x$ and $n_x^{-1}$ as an
infinite linear combination of Chevalley monomials.

We need some notation concerning finite
sequences $ \vi = (i_1, \dots, i_p) $ drawn from the set $I$.

\begin{definition}
\begin{enumerate}
\item
We denote by $\Seq$ the set of all such sequences $\vi$, and
for $ \nu \in Q_+ $, we put
$$ \Seq(\nu) := \{ \vi = (i_1, \dots, i_p) : \alpha_{i_1} + \cdots + \alpha_{i_p} = \nu \} .$$
\item
A \textbf{shuffle} of two sequences $ \vj $ and $ \vk $ is a sequence $ \vi $ produced by shuffling together the sequences $ \vj$ and $ \vk $, maintaining the same relative order among the elements of $ \vj $ and $ \vk $.  We write $ \vj \shin \vk $ to denote this set of shuffles.  Thus, if $ \vj $ has length $ p $ and $ \vk $ has length $ q $, then $ \vj \shin \vk $ has $ \binom{p + q}{p} $ elements.
\item
To a sequence $\vi=(i_1,\ldots,i_p)$ in $\Seq$, we associate the weights
$$\beta_0^\vi = 0,\quad
\beta_1^\vi = \alpha_{i_1},\quad
\beta_2^\vi = \alpha_{i_1} + \alpha_{i_2},\quad\ldots,\quad
\beta_p^\vi = \alpha_{i_1} +\dots+ \alpha_{i_p}.$$
\end{enumerate}
\end{definition}

Consider the free Lie algebra $ \mathfrak f $ on the set $ \{\hat e_i:i\in I\} $ and its universal enveloping algebra $ U (\mathfrak f) $ (identified with the free associative algebra on this set).  This algebra $ U (\mathfrak f) $ is graded by $ Q_+ $.  For each $ \nu \in Q_+$, we have a basis $ \{ \hat e_\vi:= \hat e_{i_1} \cdots \hat e_{i_p} \}_{\vi \in Seq(\nu)} $ for $ U(\mathfrak f)_\nu $.

The algebra $U(\mathfrak f) $ is in fact a graded Hopf algebra with finite dimensional components, so its graded dual $ (U(\mathfrak f))^* $ is also a Hopf algebra.  For each $ \nu \in Q_+$, we consider the basis $ \{ \hat e_\vi^* \}_{\vi \in Seq(\nu)} $ for $ (U(\mathfrak f))^*_{-\nu} $ dual to the above basis $ \{ \hat e_\vi \}_{\vi \in Seq(\nu)} $ for $ U(\mathfrak f)_\nu $.  From the definition of the coproduct on $ U(\mathfrak f) $, we get the following shuffle identity in $ (U(\mathfrak f))^*$:
\begin{equation} \label{eq:shuffleUf}
\hat e_\vj^* \, \hat e_\vk^* = \sum_{\vi \in \vj \shin \vk} \hat e_\vi^*
\end{equation}

 There is a unique Hopf
algebra map $U(\mathfrak f)\to U(\mathfrak n)$ that sends $\hat e_i$ to
$\dot e_i$ for each $i\in I$. The dual map is an inclusion of algebras
$\ON\hookrightarrow (U(\mathfrak f))^*$. (This inclusion was previously
studied by various authors, including in \cite[\S 8]{GeissLeclercSchroer}.)
Each sequence $\vi=(i_1,\ldots,i_p)$ defines a monomial
$\dot e_\vi=\dot e_{i_1}\cdots \dot e_{i_p}$ in $U(\mathfrak n)$.

The following well-known functional identity seems to be related to
cat chasing and moulds \cite{MO} and to an identity of
Littlewood~\cite[Lemma p.~149]{DKnutson}.
\begin{lemma}
\label{le:RatFnId}
Given $p$ complex numbers $a_1$, \dots, $a_p$, define
$$f_p(a_1,\ldots,a_p)=\frac1{a_1(a_1+a_2)\cdots(a_1+a_2+\cdots+a_p)}$$
whenever it makes sense. Let $p$ and $q$ be positive integers and let
$\mathrm{Sh}(p,q)$ denote the set of all permutations
$\sigma\in\mathfrak S_{p+q}$ such that
$$\sigma(1)<\sigma(2)<\cdots<\sigma(p)\quad\text{and}\quad
\sigma(p+1)<\sigma(p+2)<\cdots<\sigma(p+q).$$
Then for any complex numbers $a_1$, \ldots, $a_{p+q}$, we have
$$f_p(a_1,\ldots,a_p)\,f_q(a_{p+1},\ldots,a_{p+q})=
\sum_{\sigma\in\mathrm{Sh}(p,q)}f_{p+q}(a_{\sigma^{-1}(1)},\ldots,
a_{\sigma^{-1}(p+q)})$$
whenever both members make sense.
\end{lemma}
\begin{proof}
By analytic continuation, we can deduce the general result from the case
where all $a_i$ have positive real part. In this particular case,
$$f_p(a_1,\ldots,a_p)=\int_{C_p}e^{-(a_1x_1+\cdots+a_px_p)}\;dx_1\;\cdots
\;dx_p$$
where $C_p=\{(x_1,\ldots,x_p)\in\mathbb R^p\mid x_1>\cdots>x_p>0\}$.
The proposition follows by writing $C_p\times C_q$ as the disjoint union
of cones
$$\sigma^{-1}\cdot C_{p+q}=\{(x_{\sigma(1)},\ldots,x_{\sigma(p+q)})\mid
x_1>\cdots>x_{p+q}>0\}$$
for $\sigma\in\mathrm{Sh}(p,q)$, up to a nullset.
\end{proof}

For a sequence $\vi=(i_1,\ldots,i_p)$, we define
$$\barD_\vi=\prod_{k=0}^{p-1}\frac1{\beta^\vi_k-\beta^\vi_p}.$$

These rational functions $ \barD_\vi $ can be evaluated on
any $x\in\ftreg$ that satisfies $\langle\beta,x\rangle\neq0$
for all $\beta\in Q_+\setminus\{0\}$.

\begin{proposition}
\label{pr:ExpanCh}
Let $x\in\ftreg$ such that $\langle\beta,x\rangle\neq0$ for all
$\beta\in Q_+\setminus\{0\}$.Then
$$f(n_x)=\sum_{\vi\in\Seq}\; \langle \dot e_\vi, f \rangle \;\,\barD_\vi(x) $$
for all $f\in\ON$.
\end{proposition}
These linear combinations appearing in this statement are infinite only
in appearance, for $U(\mathfrak n)$ acts locally nilpotently on $\ON$.
The proposition says that the morphism $\barD$  can be
expanded as $\mathbb C(\mathfrak t)$-linear combinations of Chevalley
monomials:
\begin{equation*}
\label{eq:ExpanBarD}
\barD=\sum_{\vi\in\Seq}\barD_\vi\;\dot e_\vi.
\end{equation*}
\begin{proof}
To prove this formula, we need to show that as linear forms on $\ON$
\begin{equation}
\label{eq:ExpanNx}
n_x=\sum_{\vi\in\Seq}\barD_\vi(x)\;\dot e_\vi.
\end{equation}
We first note that Lemma~\ref{le:RatFnId} implies that
$$\barD_\vj(x)\;\barD_\vk(x)=\sum_{\vi\in\vj\shin\vk}\barD_\vi(x)$$
for all sequences $\vj$ and $\vk$. Comparing with \eqref{eq:shuffleUf},
it follows that
$$\sum_{\vi\in\Seq}\barD_\vi(x)\;\dot e_\vi$$
is an algebra map $(U(\mathfrak f))^*\to\mathbb C$. Thus, the right hand
side of \eqref{eq:ExpanNx} is an algebra map $\ON\to\mathbb C$, so
is the evaluation at an element $n\in N$.

Let us compute how this $n$ acts on $x$ in the adjoint representation of
$N$ on $\mathfrak g$. Since $\barD_{(k)}(x)=-1/\alpha_k(x)$, we have
$$\sum_{k\in I}\barD_{(k)}(x)\;\ad_{\dot e_k}(x)=
\sum_{k\in I}\left(\frac{-[\dot e_k,x]}{\alpha_k(x)}\right)=\dot e.$$
Each sequence of length greater than $2$ can be written as a concatenation
$(\vi,j,k)$ with $\vi\in\Seq$ (possibly empty) and $(j,k)\in I^2$.
Denoting by $p$ the length of $\vi$, we compute
\begin{align*}
\barD_{(\vi,j,k)}(x)\;\ad{\dot e_{(\vi,j,k)}}(x)
&=\left(\prod_{\ell=0}^p\frac1{\beta^\vi_\ell-(\beta^\vi_p+
\alpha_j+\alpha_k)}\times\frac{-1}{\alpha_k}\right)(x)\;
\ad_{\dot e_\vi}([\dot e_j,[\dot e_k,x]])\\
&=\left(\prod_{\ell=0}^p\frac1{(\beta^\vi_\ell-\beta^\vi_p)-
(\alpha_j+\alpha_k)}\right)(x)\;\ad_{\dot e_\vi}([\dot e_j,\dot e_k]).
\end{align*}
Summing these elements with $(j,k)$ running over $I^2$ gives zero, since
terms pairwise cancel by antisymmetry of the Lie bracket $[\dot e_j,\dot e_k]$.
Taking the sum over $\mathbf i$ then yields the equality
$$\sum_{\mathbf i\neq\varnothing}\barD_\vi(x)\;\ad_{
\dot e_\vi}(x)=\dot e$$
and we conclude that $\Ad_n(x)=x+\dot e$. (Note that the above sum makes
sense since it is in fact finite.) As this is the definition of
$n_x$, this completes the proof of~\eqref{eq:ExpanNx}.
\end{proof}

\subsection{Measures from simplices}\label{ssec:measures}
In the rest of \S \ref{se:Measures}, we explain that $\barD$ is the shadow of a measure-valued morphism that carries
more information. We start with its construction.

Consider the vector space of $\C$-valued compactly supported
distributions on $ \tR $.  It forms an algebra under convolution, the pushforward 
along the addition map $\tR \times \tR \xrightarrow{+} \tR$.
Define $\PP$ to be the subspace spanned by those distributions equal to linear combinations of
piecewise-polynomial functions times Lebesgue measures on
(not necessarily full-dimensional) polytopes whose vertices lie in the weight lattice $ P $;
it is a subalgebra.
All the distributions we will consider live in $\PP$.

Let $ \Delta^p := \{(c_0, \dots, c_p) \in \mathbb R^{p+1} : \text{ each }c_i\geq0,\ c_0 + \cdots + c_p = 1 \} $ be the standard $p$-simplex.
For $ \vi \in \Seq $ of length $p$, we define the linear map
$ \pi_\vi : \R^{p+1} \rightarrow \tR $ by
$$ \pi_\vi(c_0,\dots,c_p) = -\sum_{k = 0}^p c_k\,\beta_k^\vi$$

We define the measure $ D_\vi $ on $\tR$ by $ D_\vi := (\pi_\vi)_*(\delta_{\Delta^p}) $, the push-forward of Lebesgue measure on the $p$-simplex.  Note that the total mass of $ D_\vi $ is $1/p!$.

\begin{lemma} \label{le:shuffle}
The measures $ D_\vi $ satisfy the shuffle identity
$$
D_\vj * D_\vk = \sum_{ \vi \in \vj \shin \vk} D_\vi.
$$
\end{lemma}

\begin{proof}
Let $p$ and $q$ be the lengths of $\vj$ and $\vk$, respectively, and
consider the composite map
$$\pi_\vj+\pi_\vk:\ \R^{p+1} \times \R^{q+1} \xrightarrow{\pi_\vj \times \pi_\vk}
\tR \oplus \tR \stackrel{+}{\to} \tR.$$
Then the left side of the
desired equality is exactly $(\pi_\vj + \pi_\vk)_*(\delta_{\Delta^p \times \Delta^q})$.
To get the right side, we triangulate the product $\Delta^p \times\Delta^q$
in one of the standard ways (see e.g.\ \cite[pp.~277--278]{Hatcher})
with one simplex for each shuffle.
\end{proof}

Comparing  Lemma \ref{le:shuffle} with \eqref{eq:shuffleUf}, we deduce that
there is an algebra morphism $ (U (\mathfrak f))^* \rightarrow \PP $
taking $ \hat e_\vi^* $ to $ D_\vi $. Composing with the inclusion of algebras
$\ON\hookrightarrow (U(\mathfrak f))^*$ from \S \ref{ssec:sequences}, we get an
algebra map $ D : \ON \rightarrow \PP $. Unpacking the above definitions,
we see that for any $ f \in \C[N] $,
\begin{equation}
\label{eq:Df} D(f) = \sum_{\vi \in \Seq} \langle \dot e_\vi, f\rangle\; D_\vi.
\end{equation}

\subsection{The Fourier Transform}\label{ssec:FT}
For each weight $\beta\in P$, we define $e^\beta$ to be the function $x\mapsto e^{\langle \beta, x \rangle} $ on $ \ft_\C $. Let $\PP'$ be the space of meromorphic functions on $ \ft_\C $ spanned by these exponentials over the field $\C(\ft)$ of rational functions. The \textbf{Fourier Transform} is defined to be the map
\begin{equation*}
\begin{aligned}
FT : \ \PP  &\rightarrow \PP' \\
\mu &\mapsto \left(x \mapsto \int_{\beta\in\ft_\R^*}
e^{\langle \beta, x \rangle}\;d\mu\right)
\end{aligned}
\end{equation*}

\begin{lemma}
The Fourier transform is one to one and satisfies
\begin{enumerate}
\item $FT(a * b) = FT(a)\, FT(b)$ for all $a,b\in \PP$.
\item Let $\beta\in P$. Denoting by $\delta_\beta$ the point measure
at $\beta$, we have $FT(\delta_\beta) = e^{\beta}$.
\end{enumerate}
\end{lemma}

\begin{lemma}
\label{le:StructDi}
  For a sequence $\vi = (i_1, \dots, i_p)$, the Fourier transform
  $ FT(D_\vi) $ is given by
  $$
  FT(D_\vi) = \sum_{j = 0}^p \frac{e^{-\beta^\vi_j}}{\prod\limits_{k \ne j} (\beta^\vi_k  - \beta^\vi_j)}
  $$
\end{lemma}

\begin{proof}
The Fourier Transform for the Lesbesgue measure on a polytope is well-known (see for example \cite[Proposition 5.3]{Br3}).  The current result then follows from the compatibility between pullback of functions and pushforward of measures.
\end{proof}

The exponentials $e^{-\beta^\vi_k}$ can be regarded as regular functions
on the torus $T$. On the other hand, the denominators $\prod_{k\neq j}
(\beta^\vi_k-\beta^\vi_j)$ belong to the multiplicative subset
$ S \subset \C[\ft] $ generated by the set $ Q \smallsetminus \{0\} $.
From the Lemmas, we immediately obtain the following.

\begin{corollary}
The composition $ FT \circ D $ defines an algebra morphism
$$
\C[N] \rightarrow S^{-1} \C[\ft] \otimes \C[T]
$$
\end{corollary}

Thus, the map $FT\circ D$ can geometrically be viewed as a rational map
$\ft\times T\to N$. (Note here that $S$ can be replaced by a finitely
generated semigroup, because $\ON$ is finitely generated.)

\begin{remark} \label{rem:NotInjective}
No open subset of $\ft\times T$ maps dominantly to $N$ if
$\dim N>\dim(\ft\times T)$, so $FT\circ D$ cannot be injective if
the number of positive roots exceeds twice the rank of $G$. Since FT
is one-to-one, this means that $D$ is not injective in general.
\end{remark}

\begin{remark} \label{rem:exponentials}
For any $f\in\ON$, we can write
$$FT\circ D(f)=\sum_{\vi\in\Seq}\langle \dot e_\vi, f \rangle\;FT(D_\vi).$$
If $f\in\ON_{-\nu}$, then the sum can be restricted to sequences in
$\Seq(\nu)$, and we see from Lemma~\ref{le:StructDi} that the
exponentials $e^{-\beta}$ that appear in $FT\circ D(f)$ satisfy
$0\leq\beta\leq\nu$. 

Further, comparing with Proposition~\ref{pr:ExpanCh},
we see that $\barD(f)$ is the coefficient
of $e^{-\nu}$.  On the other hand, it is not difficult to show (using Remark~\ref{re:StarInv}) that the map $ x \mapsto n_x^{-1} $, corresponds to the coefficient of $e^0$ in $FT\circ D(f)$.
\end{remark}

\begin{theorem}
\label{th:GeomTFD}
Let $x\in\ftreg$, let $t\in T$, and let $f\in\ON$. Then $FT\circ D(f)$,
viewed as a rational function on $\ft\times T$, can be evaluated at
$(x,t)$, and we have
$$FT\circ D(f)(x,t)=f(t^{-1}n_xtn_x^{-1}).$$
\end{theorem}

\begin{proof}
We first consider the particular case where $\langle\beta,x\rangle\neq0$
for all $\beta\in Q\setminus\{0\}$. Given $\vi\in\Seq$ of length $p$ and
$\ell\in\{0,\ldots,p\}$, we set
$$A^\vi_\ell(x,t)=\frac{t^{-\beta^\vi_\ell}}{\prod\limits_{\substack{m=0\\
m\neq\ell}}^p(\beta^\vi_m-\beta^\vi_\ell)(x)}$$
where $t^{-\beta^\vi_\ell}$ means the evaluation at $t^{-1}$ of the weight
$\beta^\vi_\ell$. In view of Lemma~\ref{le:StructDi}, we want to prove
that the linear form
\begin{equation}
\label{eq:GeomTFD}
\sum_{\mathbf i\in\Seq}\left(\sum_{\ell=0}^pA^\vi_\ell
(x,t)\right)\;\dot e_\vi
\end{equation}
on $\ON$ is the evaluation at the point $t^{-1}n_xtn_x^{-1}$.

We first note that the linear form \eqref{eq:GeomTFD} is an algebra map
$\ON\to\mathbb C$, because it is the composition of the algebra map
$FT\circ D$ with the evaluation at $(x,t)$. Therefore it is the
evaluation at a point $n\in N$.

By construction, the element $t^{-1}n_xtn_x^{-1}$ is the unique
element of $N$ that brings $x+\dot e$ to $\Ad_{t^{-1}}(x+\dot e)=x+\Ad_{t^{-1}}(\dot e)$.
Let us show that $n$ fulfills this task.

\smallskip
Since $A^{(k)}_0(x,t)+A^{(k)}_1(x,t)=(1-t^{-\alpha_k})/\alpha_k(x)$, we have
$$\sum_{k\in I}\Bigl(A^{(k)}_0(x,t)+A^{(k)}_1(x,t)\Bigr)\,\ad_{\dot e_k}(x)=
\sum_{k\in I}(t^{-\alpha_k}-1)\dot e_k=\Ad_{t^{-1}}(\dot e)-\dot e.$$
Moreover, for each sequence $\mathbf i$ (possibly empty) of length $p$
and each pair $(j,k)$ of elements from~$I$, we have
$$A^{(\vi,j)}_{p+1}(x,t)-A^{(\vi,j,k)}_{p+1}(x,t)\times\alpha_k(x)=0$$
and therefore
\begin{multline*}
\left(\sum_{\ell=0}^{p+1}A^{(\mathbf i,j)}_\ell(x,t)\right)\,
\ad_{\dot e_{(\mathbf i,j)}}(\dot e_k)
+\left(\sum_{\ell=0}^{p+2}A^{(\mathbf i,j,k)}_\ell(x,t)\right)\,
\ad_{\dot e_{(\mathbf i,j,k)}}(x)=\\
\left(\sum_{\ell=0}^p\Bigl(A^{(\vi,j)}_\ell(x,t)-A^{(\vi,j,k)}_\ell(x,t)
\times\alpha_k(x)\Bigr)-A^{(\vi,j,k)}_{p+2}(x,t)\times\alpha_k(x)\right)\,
\ad_{\dot e_\vi}([\dot e_j,\dot e_k]).
\end{multline*}
Summing these elements with $(j,k)$ running over $I^2$ therefore gives
zero since terms pairwise cancel; indeed
$$A^{(\vi,j)}_\ell(x,t)-A^{(\vi,j,k)}_\ell(x,t)\times\alpha_k(x)=
\frac{t^{-\beta^\vi_\ell}}{\Biggl(\prod\limits_{\substack{m=0\\m\neq\ell}}^p
(\beta^\vi_m-\beta^\vi_\ell)\times(\beta^\vi_p-\beta^\vi_\ell+
\alpha_j+\alpha_k)\Biggr)(x)}$$
and
$$-A^{(\vi,j,k)}_{p+2}(x,t)\times\alpha_k(x)=
\frac{t^{-\beta^\vi_\ell-\alpha_j-\alpha_k}}{\Biggl(\prod
\limits_{m=0}^p(\beta^\vi_m-\beta^\vi_p-\alpha_j-\alpha_k)\Biggr)(x)}$$
are symmetric in $(j,k)$, while the Lie bracket $[\dot e_j,\dot e_k]$ is antisymmetric.
Taking the sum over all $\mathbf i\in\Seq$ then yields the equality
$$\Ad_n(x+\dot e)=x+\Ad_{t^{-1}}(\dot e),$$
which completes the proof of the equality $n=t^{-1}n_xtn_x^{-1}$.

The Theorem is thus established when $\langle\beta,x\rangle\neq0$ for
all $\beta\in Q\setminus\{0\}$. The general case then follows from the
regularity of the map $(x,t)\mapsto f(t^{-1}n_xtn_x^{-1})$ on
$\ftreg\times T$.
\end{proof}

\subsection{Universal centralizer interpretation} \label{se:univcent}
We now give a reinterpretation of Theorem \ref{th:GeomTFD} using a version of the universal centralizer.

For any $ y \in \mathfrak g $, we write $ C_G(y) = \{ g \in G : \Ad_g(y) = y \} $ for its centralizer in $ G $.  This is an algebraic group of dimension at least $r$, the rank of $ G $.  An element $ y  $ is said to be \textbf{regular} if $ \dim C_G(y) =  r $.  It is well-known that regular elements form a
non-empty open subset of $\mathfrak g$ in the Zariski topology.

We will need a lemma. For $k\geq1$, denote by $\mathscr C^k$ the $k$-th
term in the lower central series of the nilpotent Lie algebra
$\mathfrak n$; hence $\mathscr C^1=\mathfrak n$ and $\mathscr C^k=0$ for
$k$ large enough. For $x\in\ft$, denote by $\mathfrak n^x=\{f\in\mathfrak
n\mid [x,f]=0\}$ the set of elements in $\mathfrak n$ that centralize $x$.

\begin{lemma}
\label{le:UCS}
Let $k\geq1$ be an integer and let $(x,a,b)\in\ft\times\mathfrak n^x\times
\mathscr C^k$. Then there is $(m,a')\in N\times\mathfrak n^x$ such
that $\Ad_m(x+a+b)\in x+a'+\mathscr C^{k+1}$. 
\end{lemma}
\begin{proof}
The linear map $\ad_x:\mathfrak g\to\mathfrak g$ is semisimple and leaves
the subspace $\mathscr C^k$ stable. We can thus decompose into direct
sums $\mathfrak g=\im\ad_x\oplus\ker\ad_x$ and
$\mathscr C^k=\ad_x(\mathscr C^k)\oplus(\mathscr C^k\cap\ker\ad_x)$.
Let us write $b=[x,u]+b'$, with $u$ and $b'$ in $\mathscr C^k$ and
$[x,b']=0$. Then $m=\exp u$ satisfies
$\Ad_m(x+a+b)\in x+a+b+[u,x]+\mathscr C^{k+1}$, and one can simply take
$a'=a+b'$.
\end{proof}

\begin{proposition}
For any $ x\in\ft$, the element $ x+\dot e $ is regular and  $C_G(x+\dot e) \subset B $.
\end{proposition}

\begin{proof}
First, we prove that $ x+\dot e $ is regular.  Consider the action of $ G \times \C^\times $ on $ \fg $, where $ G $ acts by the adjoint action and $ \C^\times $ acts by scaling.  Define $ \C^\times \rightarrow G \times \C^\times $ by $ s \mapsto (s^{-\rho^\vee},s) $.  Then $ s \cdot (x+\dot e) = sx + \dot e$.  So $ \lim_{s\rightarrow 0} s \cdot (x+\dot e) = \dot e $.   The set of regular elements is open in $ \fg $ and is invariant under the action of $ G \times \C^\times $.  Since the limit point $\dot e $ is regular, we conclude that $ x +\dot e $ is regular.

Now we prove that $C_G(x+\dot e)\subset B$.
Starting with $a=0$ and $b=\dot e$, we apply Lemma~\ref{le:UCS} several
times, taking successively $k=1,2, \dots$ . Composing all the maps $\Ad_m$
obtained in the process, we eventually find elements $m\in N$ and $f\in\mathfrak n^x$
such that $\Ad_m(x+\dot e)=x+f$. Note that $x$ and $f$ are
the components of the Jordan--Chevalley decomposition of $x+f$.

On the other hand, Theorem~2.2 in \cite{Humphreys95} states that the
centralizer $L$ of $x$ is a reductive group with maximal torus $T$
and root system
$\Phi_L=\{\alpha\text{ root of }G\mid\langle\alpha,x\rangle=0\}$.
(As a matter of fact, the statement in \cite{Humphreys95} deals with
the centralizer of a semisimple element in $G$ and not of an element
$x\in\mathfrak g$, but the proof can be adapted to our situation.)
Further $L$ is connected \cite[Theorem~2.11]{Humphreys95}. The intersection $B_L:=B\cap L$
is a Borel subgroup of $L$. We also note that $f$ belongs to the
centralizer of $x$ in $\mathfrak g$, that is, the Lie algebra of $L$.

By the uniqueness of the Jordan--Chevalley decomposition, the centralizer
$C_G(x+f)$ is the joint centralizer of $x$ and $f$ in $G$, so is the
centralizer $C_L(f)$. From the fact that $x+f$ is regular in the Lie
algebra $\mathfrak g$, we then deduce that $f$ is a regular nilpotent
element in the Lie algebra of $L$. 
Let $g\in C_L(f)$. Following the proof of \cite[Lemma~4.3]{Springer},
we write $g=b\,\overline w\,b'$ in the Bruhat decomposition of $L$, where
$b$ and $b'$ are in $B_L$ and $w$ is in the Weyl group of $L$. Both
$\Ad_{b'}(f)$ and $\Ad_{b^{-1}}(f)$ are regular nilpotent elements in
the Lie algebra of $B_L$, so are linear combinations of positive roots
vectors, each simple root in $\Phi_L$ occurring with a nonzero coefficient
(\cite{Bourbaki}, chap.~8, \S11, no.~4, proposition~10).
Since $\Ad_{\overline w}$ maps $\Ad_{b'}(f)$ to $\Ad_{b^{-1}}(f)$, this
implies that $w$ maps each simple root in $\Phi_L$ to a positive root,
and it follows that $w=1$ and $g\in B_L$. We conclude that
$C_L(f)\subset B_L$. Thus,
$$C_G(x+\dot e)=\Ad_{m^{-1}}(C_G(x+f))=\Ad_{m^{-1}}(C_L(f))
\subset\Ad_{m^{-1}}(B_L)\subset B.$$
\end{proof}

We define the \textbf{universal centralizer space} to be
$$
C := \bigl\{ (x, b) \in \ft \times B \bigm| b \in C_G(x+\dot e) \bigr\}
$$
\begin{remark}
Our space $ C$ is the base change over $ \ft \rightarrow \ft/W $ of the usual universal centralizer, as defined in for example \cite[\S 2.2]{BF}.
\end{remark}

From the definition, we have maps $ C \rightarrow \ft,\ C \rightarrow T, \ C \rightarrow N $, that send a pair $ (tn, x) \in C $ to $ x, t, n $, respectively (where $ x \in \ft, t \in T, n \in N $).   We will be particularly interested in the map $\psi : C \rightarrow N $.

\begin{proposition}\label{th:FTDpsi}
\begin{enumerate}
\item The above maps $ C \rightarrow \ft, C \rightarrow T $ restrict to an isomorphism $ C \times_\ft \ftreg = \ftreg \times T $.  
\item With respect to the isomorphism in (i), the map $ \psi $ restricts to
$$ \psi_{\mathrm{reg}} : \ftreg \times T \rightarrow N \quad (x, t) \mapsto t^{-1} n_x t n_x^{-1} $$
\item The resulting algebra morphism $ \psi_{\mathrm{reg}}^* : \C[N] \rightarrow \C[\ftreg \times T] $ agrees with $ FT \circ D $.
\end{enumerate}
\end{proposition}
\begin{proof}
\begin{enumerate}
\item Since $ Ad(n_x)(x) = x+\dot e $,  $ C_G(x+\dot e) = n_x C_G(x) n_x^{-1} $.  Since $ x \in \ftreg $, $ C_G(x) = T $ and so $ C_G(x+\dot e) = \{ n_x t n_x^{-1} : t \in T \} $.  The map $(x,t)\mapsto(n_xtn_x^{-1},x)$ from $\ftreg\times T$ to $C$ is the converse of the desired isomorphism. 
\item Since $ n_x t n_x^{-1} = t t^{-1} n_x t n_x^{-1} $ and $t^{-1} n_x t n_x^{-1} \in N $, the result follows.
\item Given the previous result, this is just a restatement of Theorem \ref{th:GeomTFD}.
\end{enumerate}
\end{proof}

\section{Generalities on Duistermaat--Heckman measures} \label{se:DH}
\subsection{Duistermaat--Heckman measures}\label{ssec:DH}
We will now define Duistermaat--Heckman measures algebraically,
following Brion--Procesi \cite{BrionProcesi}.  In this section, we work in a general context of a projective variety with the action of a torus.  Later, we will apply these ideas to the case of an MV cycle with the action of $T^\vee$.

Let $ V $ be a (possibly infinite-dimensional) vector space with a linear action of a torus $ T $.  Let $P $ be the weight lattice of $ T $ and let $P^\vee $ be its coweight lattice.  Let $ X \subset \mathbb P V $ be a finite-dimensional $T$-invariant closed subscheme of the projectivization of $ V $.
Let $ \mathcal O(n) $ denote the usual line bundle on $ \mathbb P(V) $.  Since $ T $ acts linearly on $ V$, $ \mathcal O(n) $ carries a natural $ T$-equivariant structure.

We do not assume that the torus $ T $ acts effectively on $ X $ (or even on $ V$).  We write $ T' $ for the quotient of $ T $ acting effectively on $ X $.

On the other hand, we do assume that $ X^T $ is finite.  For each $ p \in X^T$, let $ \Phi_T(p) $ be the weight of the action of $T $ on the fibre
of $ \mathcal O(1) $ at the point $ p$.  Equivalently, $ p = [v] $ for some weight vector $ v \in V $ and $\Phi_T(p) $ is negative the weight of $ v $.

Define the {\bf moment polytope} $ \Pol(X) $ to be
$$ \Pol(X) := \Conv(\Phi_T(p) : p \in X^T)
$$
If $X$ is connected (e.g. irreducible), then $ \Pol(X) $ is
contained in a translate of $ (\mathfrak t')^*_{\mathbb R} \subset \tR $.

In fact, it is easy to see that $ \Pol(X) $ is the convex hull of all
negatives of weights of the smallest linear subspace of $ V $
containing $ X $.

The torus $ T $ acts on the space of sections $ \Gamma(X, \mathcal O(n))$.  We consider $ [\Gamma(X, \mathcal O(n))] $, the class of $ \Gamma(X, \mathcal O(n)) $ in $ R(T) $, the complexified representation ring of $ T $.  We can embed $ R(T) $ into the space of distributions on $ \tR $ by setting
$$
[U] \mapsto \sum_{\mu \in P} \dim U_\mu \, \delta_\mu
$$
Let $ \tau_n : \tR \rightarrow \tR $ be the automorphism given by scaling by $ \frac{1}{n}$.

\begin{definition}
  The \emph{Duistermaat--Heckman measure of the triple $ (X, T, V) $}
  is defined to be the weak limit $ DH(X) = \displaystyle{\lim_{n \to \infty}} \frac{1}{n^{\dim X}} (\tau_n)_*[\Gamma(X, \mathcal O(n))] $
  within the space of distributions on $ \tR $.
\end{definition}

Note that each $ (\tau_n)_*[\Gamma(X, \mathcal O(n))] $ is supported on $ \Pol(X) $, and hence so is $ DH(X)$.
In fact, we have the following result of Brion--Procesi \cite{BrionProcesi}.
\begin{proposition}
The measure $DH(X) $ is well-defined, has support $ \Pol(X) $, and is piecewise polynomial of degree $ \dim X - \dim T' $.
\end{proposition}

\subsection{Fourier transform of DH measures and equivariant multiplicities}\label{ssec:FTDH}
 For this section, assume that each fixed point $ p \in X^T $ is non-degenerate and attracting.  This means that for each $ p $, there exists $ \gamma \in P^\vee$ such that if $ \mu $ is a weight of $ T $ acting on $ T_p X$, then $ \langle \gamma, \mu \rangle > 0 $.  We write $ \sigma_p^0 $ for the set of such $ \gamma $ (this is the intersection of $ P^\vee $ with an open cone in $\tR $ and thus is Zariski dense in $ \mathfrak t^* $).

We can compute the Fourier transform (as defined in \S \ref{ssec:FT}) of DH measures with the help of
localization in equivariant $K$-theory and equivariant homology.
Let $\hat{S}$ be the multiplicative set in $ R(T) $ generated by
$ 1 - \delta_\mu $ for $ \mu \in P\setminus\{0\} $
and let $ S $ be the multiplicative set in
$ H^\bullet_T = \Sym \ft^* = \C[\mathfrak t]$ generated by
$ \mu \in P  \setminus \{0\}$.

Let $ K^T(X) $ denote the Grothendieck group of $T$-equivariant
coherent sheaves on $ X $.  This is a module over $ R(T)$.  The
following result is due to Thomason \cite[Th\'eor\`eme 2.1]{Thomason}, for
$K$-theory, and many authors independently (such as Brion \cite[Lemma 1]{Br2} or Evens and Mirkovi\'c \cite[Theorem B.2]{EM}), for homology.
\begin{theorem}
The inclusion $X^T \rightarrow X $ induces isomorphisms:
$$ \hat S^{-1} K^T(X^T) \xrightarrow{\sim} \hat S^{-1} K^T(X)
\quad \text{ and }  \quad
S^{-1} H^T_\bullet(X^T) \xrightarrow{\sim} S^{-1} H^T_\bullet(X).$$
\end{theorem}
Because of this theorem, we can write
$$
[\mathcal O_X] = \sum_{p \in X^T} \hat \varep^T_p(X) [\mathcal O_p] \qquad
[X] = \sum_{p \in X^T} \varep^T_p(X)[\{p\}]
$$
for unique $ \hat \varep^T_p(X) \in \hat S^{-1} R(T) $ and
$ \varep^T_p(X) \in S^{-1}\C[\ft] $.  Following Brion \cite{Br}, we
call $ \varep_p^T(X)$ the \textbf{equivariant multiplicity} of $ X $
at $p $.

One advantage of these equivariant multiplicities is that they can be computed locally.  Let $ X_p^\circ $ denote an affine open $T$-invariant neighbourhood of $p $ in $ X $.  (In \cite[Proposition 4.4]{Br}, Brion observes that the only such $X_p^\circ $  is the attracting set of $p$.)

We will need the following preliminary definition.  Let $ A = \oplus_{m=0}^\infty A_m $ be an $\mathbb N$-graded finitely-generated commutative algebra of Krull dimension $ d $.  Then the multiplicity of $ A $ is defined to be $ (d-1)! \displaystyle{ \lim_{m \to \infty} } \frac{ \dim A_m}{m^{d-1}} $ (this limit exists and is always non-zero).
\begin{proposition} \label{pr:localcalc}
\begin{enumerate}
 \item
 $[\C[X_p^\circ]]$ is a well-defined element of $ \hat S^{-1}R(T) $ and we have $\hat \varep^T_p(X) = [\C[X_p^\circ]] $ in $ \hat S^{-1}R(T)$.
 \item
 We have $ \varep^T_p(X) [\{p \}] = [X_p^\circ]$ in $S^{-1}H_\bullet^T(X_p^\circ) $.
 \item
 For any $ \gamma \in \sigma_p^0$, the rational function $ \varep^T_p(X) $ is well-defined at $ \gamma $ and $ \varep^T_p(X)(\gamma) $ is the multiplicity of the algebra $ \C[X_p^\circ] $ graded with respect to $ - \gamma $.
\end{enumerate}

\end{proposition}

\begin{proof}
$[\C[X_p^\circ]]$ is well-defined because of the assumption of attractiveness (see \cite[Prop 6.6.6]{CG}).  The rest of part (i) and (ii) then follow immediately from pullback to the open set $X_p^\circ$.

Part (iii) is due to Brion \cite[Prop 4.4]{Br}.
\end{proof}

To facilitate further computation of $ \varep^T_p(X) $, suppose that we have a representation $W$ of $ T $, all of whose weights are non-zero,
and assume we are given a $T$-equivariant closed embedding $ X_p^\circ \rightarrow W $.  The \textbf{multidegree} $ \mdeg[W](X_p^\circ) \in H_T^\bullet$ is defined by the equation $\mdeg[W](X_p^\circ) [W] = [X_p^\circ] $ in $ H^T_\bullet(W) $.  This notion of multidegree is useful, since it can be computed using the methods of commutative algebra (see for example \cite[\S 1.5]{KZJ}).  On the other hand, the multidegree of $ X_p^\circ $ determines the equivariant multiplicity of $ X $ at $p $ as follows.

\begin{proposition} \label{prop:mdegeqvtmult}
 With the above setup, we have
 $$ \varep^T_p(X) = \frac{ \mdeg[W](X_p^\circ)}{\displaystyle{\prod_{\mu \text{ wt of } W}} \mu} $$
\end{proposition}

\begin{proof}
We know that $ [X_p^\circ] = \varep^T_p(X)[\{p\}] $ in $H^T_\bullet(W) $.  Since $ \mdeg[W](X_p^\circ) [W] = [X_p^\circ] $ and $ (\displaystyle{\prod_{\mu \text{ wt of } W}} \mu) [W] = [\{p \}] $ in $ H^T_\bullet(W) $, the result follows.
\end{proof}

We are ready to relate the
Duistermaat--Heckman measure to equivariant multiplicity.

\begin{theorem} \label{th:FTDH}
 We have
 $$FT(DH(X)) = \sum_{p \in X^T} \varep^T_p(X) e^{\Phi(p)}
 $$
\end{theorem}

\begin{proof}
  For sufficiently large $ n$, $H^i(X, \mathcal O(n)) = 0 $ for
  $i > 0$.  Thus for sufficiently large $ n$,
  $[ \Gamma(X, \mathcal O(n))] $ equals the integral of
  $ [\mathcal O_X \otimes \mathcal O(n)]$ in equivariant $K$-theory.
  Hence from
$$ [\mathcal O_X] = \sum_{p \in X^T} \hat \varep^T_p(X) [\mathcal O_p] $$
we deduce that for sufficiently large $ n $, we have
 $$
 [\Gamma(X, \mathcal O(n))] = \sum_{p \in X^T} \hat \varep^T_p(X) \delta_{n\Phi(p)}
 $$
Let $ d = \dim X $.   We see that
 \begin{align*}
 FT(DH(X)) &= FT \left(\lim_{n\to\infty}  \frac{1}{n^d}
(\tau_n)_*\left( \sum_p \hat \varep^T_p(X) \delta_{n\Phi(p)}\right) \right)\\
 &= \sum_p \lim_{n\to\infty}  \frac{1}{n^d}
FT\left((\tau_n)_*\left( \hat \varep^T_p(X)\right)\right) e^{\Phi(p)}
 \end{align*}
 So it suffices to show that for each $ p$, we have
$$ \lim_{n\to\infty}  \frac{1}{n^d} FT((\tau_n)_* (\hat \varep^T_p(X))) = \varep^T_p(X) $$
 Now pick $ \gamma \in \sigma_p^0 $ and let $ c $ be the multiplicity of the algebra $ \C[X_p^\circ] $ graded with respect to $ - \gamma $.  By Proposition \ref{pr:localcalc}(iii), it suffices to show that $ \displaystyle{\lim_{n\to\infty}}  \frac{1}{n^d} FT((\tau_n)_*( \hat \varep^T_p(X))) (\gamma) = c$.
Using Proposition \ref{pr:localcalc}(i) we see that
 $$ \lim_{n\to\infty}  \frac{1}{n^{\dim X}} FT((\tau_n)_*(\hat \varep^T_p(X))) (\gamma) = \lim_{n \to \infty} \frac{ \sum_{m=0}^\infty a_m e^{\frac{-m}{n}}}{n^d} $$
 where $ a_m = \displaystyle{\sum_{\langle \mu, -\gamma \rangle = m}} \dim\C[X_p^\circ]_\mu $.  Using the fact that $ a_m = \frac{c}{(d-1)!} m^{d-1} + \cdots$, by elementary calculus, we compute that this limit equals $ c $ as desired.
\end{proof}

  \begin{remark}
The above theorem holds without the assumption that the fixed points are attractive (though this assumption suffices for our purposes).  In fact, the only place that attractiveness assumption is used in this section is in Proposition \ref{pr:localcalc} part (i).  We can avoid using attractiveness by using degeneration to normal cones of the fixed points (thereby making them attractive
with respect to the attendant new circle action).
\end{remark}

\subsection{An extension to coherent sheaves}\label{ssec:supp}
We continue in the above setup, but we consider a $T$-equivariant coherent sheaf $ \mathcal F $ on $ X $.  Following \cite[Def. 5.9.4]{CG}, we define the \textbf{support cycle} of $ \mathcal F $ as
$$
[\supp \mathcal F] := \sum_{S} mult_S(\mathcal F) [S] \in H_\bullet(X)
$$
where the sum ranges over all maximal dimensional irreducible components of the support of $ \mathcal F $.  These components are necessarily $T$-invariant.

We define the Duistermaat--Heckman measure of $ \mathcal F $ by
$$ DH(\mathcal F) := \displaystyle{\lim_{n \to \infty}} \frac{1}{n^{\dim \supp \mathcal F}} (\tau_n)_*[\Gamma(X, \mathcal F \otimes \mathcal O(n))] $$

The main result of this section is that $ DH(\mathcal F) $ only depends on the support cycle.
\begin{theorem} \label{th:limCoherent}
We have the following expansion
 $$ DH(\mathcal F) = \sum_S mult_S(\mathcal F) DH(S)
 $$
 where the sum ranges over all maximal dimensional irreducible components of the support of $ \mathcal F $.
\end{theorem}

\begin{proof}
As in the previous section, we can consider the expansion of $ [\mathcal F] \in \hat S^{-1} K^T(X) $  in terms of the fixed points and we define $ \hat \varep_p^T(\mathcal F) $ by
$$
[\mathcal F] = \sum_p \hat \varep_p^T(\mathcal F) [\mathcal O_p]
$$
Then proceeding as in the proof of Theorem \ref{th:FTDH}, we obtain that
$$
FT(DH(\mathcal F)) = \ \sum_p \lim_{n\to\infty}  \frac{1}{n^d} FT(\tau_n( \hat \varep^T_p(\mathcal F))) e^{\Phi(p)}
$$
where $ d = \dim \supp \mathcal F $.

Now, as above $ \hat \varep^T_p(\mathcal F) = [\Gamma(X_p^\circ, \mathcal F)] $ in $ \hat S^{-1}R(T)$.  Thus we are reduced to a local computation on each open set $ X_p^\circ$.  The desired equation then follows from a result in Chriss--Ginzburg \cite[Theorem 6.6.12]{CG}.
\end{proof}

\subsection{A formula using BB strata}\label{ssec:BB}
The third author has given a formula in \cite{Knutson} for computing
this DH measure using components of Bia\l ynicki--Birula strata.  We
now recall this formula, and for simplicity assume $X$ is irreducible.

Choose a coweight $ \nu $ of $ T $ and thus an embedding
$ \Cx \hookrightarrow T $.  Assume that $ X^{\Cx} = X^T $ and moreover that
the composite map
$X^{\Cx} \xrightarrow{\Phi_T} P \xrightarrow{\nu} \Z $ is injective.
(If $\Phi_T$ is injective, then it is easy to construct such $\nu$,
and otherwise it is of course impossible.)
By this assumption, $ X^{\Cx} $ acquires the structure of a
totally-ordered set (very much depending on the choice of $\nu$).
We write $ p_{top}, p_{bot} $ for the maximal and minimal points in this set.

\begin{definition}
  For each $ p \in X^{\Cx} $, let $ X_p $ denote the Bia\l ynicki--Birula
  stratum of $ p$, defined by
  $$
  X_p := \{ x \in X : \lim_{s \to 0} s x = p \}
  $$
  An irreducible component $ Y $ of $ \overline{X_p} $ will be a
  called a \defn{BB cycle} in $ X $ based at $ p $.
\end{definition}
Note that the only BB cycle of weight $ \Phi(p_{top}) $ is $X$
and the only BB cycle of weight $ \Phi(p_{bot}) $ is  $\{ p_{bot} \} $.

For each $ p \in X^{\Cx}$, fix $ f_p \in V^* $ such that $ f_p $ is a $T$-weight vector and $ f_p(p) \ne 0 $.  Note that this implies that $ f_p $ vanishes at every other fixed point (since they live in different weight spaces in $ \mathbb P (V) $).  The section $f_p$ defines a $T$-invariant Cartier divisor $ D_p \subset X $.

For any two BB cycles $ Y_1, Y_2 $ based at $p_1 < p_2$,
we define $ c(Y_1, Y_2) := i(Y_1, D_{p_2} \cdot Y_2) $.

\begin{definition}
A \defn{chain} of BB cycles is a sequence $ Y_\bullet = (Y_0, Y_1, \dots, Y_m) $ such that $ Y_0 = \{p_{bot} \} $, $ Y_m = X$, and $ c(Y_{k-1}, Y_k) \ne 0 $ for all $ k $. A \defn{weight chain} is a sequence $\mu_\bullet =  (\mu_0, \dots, \mu_m) $ such that there exists a chain of BB cycles $ Y_\bullet $ such that the weight of $ Y_k $ is $\mu_k $ for all $ k $.
\end{definition}
Note that if $ Y $ is a BB cycle based at $ p $, then $ \dim D_p \cap Y = \dim Y - 1$, so the length of any chain is the dimension of $ X $.

For any weight chain $ \mu_\bullet$, define
$$
c(\mu_\bullet) = \sum_{Y_\bullet} c(Y_0, Y_1) \cdots c(Y_{m-1}, Y_m)
$$
where the sum is over all chains of BB cycles $Y_\bullet $ of weight $ \mu_\bullet $.

As before, let $ \Delta^m$ be the standard $m$-simplex.

For any weight chain $ \mu_\bullet $,  define $ \pi_{\mu_\bullet} : \R^{m+1} \rightarrow \tR $  by $ \pi_{\mu_\bullet}(a_0, \dots, a_m) = \sum_{k = 0}^m a_k \mu_k.$
The following result (due to the third author \cite{Knutson}) explains how these weight chains can be used to compute the DH measure.
\begin{theorem} \label{th:Allen}
$$DH(X) = \sum_{\mu_\bullet} c(\mu_\bullet) (\pi_{\mu_\bullet})_*(\delta_{\Delta^m}) $$
where the sum is over all weight chains.
\end{theorem}

\section{Measures from MV cycles}

Recall that at the beginning of \S \ref{se:Measures}, we fixed a choice of regular nilpotent, $ \dot e $.  From this point on, we specialize $ \dot e $ = $ \overline e $ where $ \overline e $ is the principal nilpotent coming from the geometric Satake correspondence \S \ref{ss:PrincNilp}.  In particular, we have $ \dot e_i = q(\alpha_i) e_i $.

\subsection{DH measures of MV cycles}
Let $ \nu \in Q_+$ and let $ Z $ be a stable MV cycle of weight $ -\nu$.  $ Z $ is a $T^\vee$-invariant subvariety of the affine Grassmannian and we have a projective embedding $ \Upsilon : \Gr \rightarrow \mathbb P(V) $, from \S \ref{ss:PrincNilp}.

We will apply the constructions of the last two sections to the triple $ (Z, T^\vee, V) $.  In particular, in \S \ref{ssec:DH} we defined a polytope $ \Pol(Z)$; note that this polytope lives in $ \mathfrak t_{\mathbb R} = (\mathfrak t^\vee_{\mathbb R})^*$.  On the other hand, in \S \ref{ss:MVfromMV} we defined the MV polytope (also denoted $ \Pol(Z)$) which lives in $ \tR$.

Recall that we have the map $ \iota : P \rightarrow \ft_\R $ which we can extend to a linear bijection $ \iota : \ft^*_\R \rightarrow \ft_\R $.

\begin{lemma}
 For any stable MV cycle $Z$, we have an equality $ \iota(\Pol(Z)) = \Pol(Z) $.
\end{lemma}

\begin{proof}
The action of $ T^\vee $ on $ \Gr $ has fixed points $\{ L_\mu\ :\ \mu \in P\}$.
Moreover the weight of $ T^\vee $ acting on the fibre of $ \mathcal O(1) $ at the point $L_\mu $ is $\iota(\mu) $ by Proposition \ref{pr:KacFormula}.
\end{proof}

Similarly, by definition, $ DH(Z) $ is a distribution on $ \mathfrak t_{\mathbb R} = (\mathfrak t^\vee_{\mathbb R})^*$ and we also have the distribution $ D(b_Z) $ on $ \tR $.

\begin{theorem} \label{th:DHequalsD}
 For any stable MV cycle $Z$, we have an equality $ \iota_* (D(b_Z)) = DH(Z) $.
\end{theorem}

\begin{proof}
We choose a generic dominant embedding $ \C^\times \hookrightarrow T^\vee$.
Then the BB cycles inside $ Z $ are simply the 
stable MV cycles contained in $ Z $.  Moreover, for each weight chain $ \mu_\bullet $, there exists a unique $\vi \in \Seq(\nu) $
 such that $ \mu_k = \alpha_{i_1} + \cdots + \alpha_{i_k} $
 for all $ k$.

Thus by Theorem \ref{th:Allen}, we conclude that upon identifying $ \mathfrak t_{\mathbb R} $ with $ \tR$ using $ \iota $, we have
\begin{equation} \label{eq:DHSeqnu}
DH(Z) = \sum_{\vi \in \Seq(\nu)} c(\vi) D_\vi, \qquad c(\vi) = \sum_{Y_\bullet} c(Y_0, Y_1) \cdots c(Y_{m-1}, Y_m) .
\end{equation}
where the second sum ranges over the set of sequences $ Y_\bullet $ where $ Y_k $ is a stable MV cycle of weight $ -(\alpha_{i_1} + \cdots + \alpha_{i_k}) $ and where
$$ c(Y_k, Y_{k+1}) =  i(Y_k, D_{\mu_{k+1}} \cdot Y_{k+1})
$$
Here $ D_{\mu_{k+1}} $ is the divisor in $ \Gr $ coming from a vector $ f \in V^* $ such that $ f(L_{\mu_{k+1}}) \ne 0 $ and $ f $ vanishes on all other weight spaces. 

On the other hand, we begin with (\ref{eq:Df})
$$ D(b_Z) = \sum_{\vi \in \Seq(\nu)} \langle \dot e_{i_1} \cdots \dot e_{i_m} , b_Z\rangle D_\vi = \sum_{\vi \in \Seq(\nu)} (\dot e_{i_1} \cdots \dot e_{i_m} b_Z)(1_N) D_\vi  $$  
Then we expand out the right hand side using (\ref{eq:ActionStableMVcycles}) from the proof of Lemma \ref{le:UEquivar} and we reach exactly the same formula (\ref{eq:DHSeqnu}).
\end{proof}

\subsection{Reformulation and alternate proof of Theorem \ref{th:DHequalsD}}
\label{ssec:alternate}
We apply the Fourier Transform and then $ (\iota^{-1})^* $ to the equality in Theorem \ref{th:DHequalsD}  and we obtain
$$
FT( D(b_Z)) = (\iota^{-1})^* FT(DH(Z))
$$
This is an equality of analytic functions on $ \ft_\C $; more precisely by the analysis in \S \ref{ssec:FT}, this is an equality in the space $ S^{-1}\C[\ft]\otimes \C[T] $.

Now, by Proposition \ref{th:FTDpsi} (iii), we have that $
FT(D(b_Z)) = \psi_{\mathrm{reg}}^*(b_Z) $.

On the other hand, by Theorem \ref{th:FTDH}, we know that  $ FT(DH(Z)) $ is closely related to the expansion of the homology class $ [Z] $ in the fixed point basis under the isomorphism $\iota(S)^{-1}H_\bullet^{T^\vee}(\Gr) \cong \iota(S)^{-1}H_\bullet^{T^\vee}(\Gr^{T^\vee}) $.  In order to avoid confusion, in this section we will write $ [Z]_{T^\vee} $ for the image of $[Z] $ in $ \iota(S)^{-1}H_\bullet^{T^\vee}(\Gr^{T^\vee}) $.

The fixed point set $\Gr^{T^\vee} $ is in correspondence with the weight lattice $ P $ and thus has a group structure.  Hence $  H_\bullet^{T^\vee}(\Gr^{T^\vee}) $ is an algebra and we have an obvious isomorphism  $$\iota(S)^{-1} \C[\ft^*] \otimes \C[T] \cong H_\bullet^{T^\vee}(\Gr^{T^\vee}) $$ 

From these observations (and the injectivity of the Fourier Transform), we see that Theorem \ref{th:DHequalsD} is equivalent to the following statement. 
\begin{theorem} \label{th:DHequalsD2}
	Let $ Z $ be a stable MV cycle.  In the algebra $S^{-1}\C[\ft]\otimes \C[T] $, we have an equality 
	$$
	\psi_{\mathrm{reg}}^*(b_Z) = (\iota^{-1})^*([Z]_{T^\vee})
	$$
\end{theorem}

We will now give an alternate proof of Theorem \ref{th:DHequalsD2} using results of Yun--Zhu \cite{YunZhu}.  These authors define a commutative convolution algebra structure on $ H_\bullet^{T^\vee}(\Gr) $ and describe this algebra using the geometric Satake correspondence.

To formulate their result, recall the universal centralizer space $ C $ and its morphisms $ C \rightarrow \ft \times T $ and $ \psi : C \rightarrow N $ from \S \ref{se:univcent}.  The following theorem follows from combining Propositions 3.3, 5.7 and Remark 3.4 from \cite{YunZhu}.

\begin{theorem} \label{th:HTGr}
	There is an isomorphism $\theta : \C[C] \rightarrow H_\bullet^{T^\vee}(\Gr) $ making the following diagram commute
	$$
	\xymatrix{  
		\C[\ft] \otimes \C[T] \ar^{\iota\otimes 1}[r] \ar[d] & \C[\ft^*] \otimes \C[T] \ar[r] & H_\bullet^{T^\vee}(\Gr^{T^\vee}) \ar[d] \\
		\C[C] \ar^\theta[rr] && H_\bullet^{T^\vee}(\Gr)
	}
$$
\end{theorem}
		
In order to apply this result, we will need to setup a bit more notation.

For any $ \lambda \in P_+$, we have a map
\begin{equation} \label{eq:tauDef}
\begin{aligned}
 \tau_\lambda : H_{T^\vee}^\bullet(\Gr, \IC_\lambda) 
 \rightarrow H^{T^\vee}_{d-\bullet}(\overline{\Gr^\lambda}) 
 \rightarrow H^{T^\vee}_{d-\bullet}(\Gr)
 \end{aligned}
 \end{equation}
 where the first arrow comes from (\ref{eq:ICtoD}) (with $ d = \rho^\vee(\lambda)$). We will also make use of the inclusion $$ H^\bullet(\Gr, \IC_\lambda) \hookrightarrow H_{T^\vee}^\bullet(\Gr, \IC_\lambda) $$
coming from the isomorphism $ H_{T^\vee}^\bullet(\Gr, \IC_\lambda) \cong H^\bullet(\Gr, \IC_\lambda) \otimes \C[\ft^*] $ given in \cite[Lemma 2.2]{YunZhu}.

Also recall the map 
$$
\Psi_\lambda : H^\bullet(\Gr, \IC_\lambda) = L(\lambda) \hookrightarrow \C[N]
$$
from \S \ref{ss:BasesRep}.

In this section, we will also need to consider the automorphism $ t^{-\lambda} $ of $ \Gr $ and we write $ t^{-\lambda}_* :  H^{T^\vee}_\bullet(\Gr) \rightarrow  H^{T^\vee}_\bullet(\Gr) $ for the resulting map on equivariant homology.

\begin{lemma} \label{le:tau}
	For any $ v \in H^\bullet(\Gr, \IC_\lambda) $, we have
	$$
	\theta(\psi^*(\Psi_\lambda(v))) = (t^{-\lambda})_* \tau_\lambda(v)
	$$
\end{lemma}

\begin{proof}
To begin, fix the isomorphism of varieties $  T \times N \cong B $ given by $ (t,n) \mapsto tn $.  Thus we get an isomorphism on coordinate rings $ \C[B] \cong \C[T] \otimes \C[N] $.  

Recall the linear map $ v_\lambda^* : L(\lambda) \rightarrow \C $ defined in \S \ref{ss:BasesRep}.  We extend $ v_\lambda^*$ to a  $\C[\ft^*]$-linear map $ v_\lambda^* : H^\bullet_{T^\vee}(\Gr, \IC_\lambda) \rightarrow \C[\ft^*] $ using the isomorphism $ H^\bullet_{T^\vee}(\Gr, \IC_\lambda) \cong H^\bullet(\Gr, \IC_\lambda) \otimes \C[\ft^*] $.  Note that $ v_\lambda^* = p_* \circ \tau_\lambda $, where $ \tau_\lambda $ in defined in (\ref{eq:tauDef}) and where $ p_* : H_\bullet^{T^\vee}(\Gr) \rightarrow H_\bullet^{T^\vee}(pt) $ is the proper push-forward in Borel--Moore homology.

Given $ v \in L(\lambda)$, we define $ f_v \in \C[B] $ by $ f_v(b) = v_\lambda^*(bv) $. Under the isomorphism $ \C[B]\cong \C[T] \otimes \C[N] $, we find that $ f_v = \lambda \otimes \Psi_\lambda(v) $.  

We can factor $ \psi: C \rightarrow N $ into $ C \rightarrow B \rightarrow N $ and which gives rise to a factorization of $ \psi^* $ as $ \C[N] \hookrightarrow \C[B] \xrightarrow{\psi_B^*} \C[C] $.  Following \cite[\S 3]{YunZhu}, we will now describe $ \theta \circ \psi_B^* $.  

Let $ \{h^i\} , \{h_i \} $ be dual bases of the free $ \C[\ft^*] $ modules $ H_{T^\vee}^\bullet(\Gr) $ and $ H^{T^\vee}_\bullet(\Gr) $. By the definition of $ \sigma_{can} $ from \cite[\S 3]{YunZhu}, for any $ v \in L(\lambda) = H^\bullet(\Gr, \IC_\lambda) $, we have
$$
\theta(\psi_B^*(f_v)) = \sum_i v_\lambda^*(h^i v) h_i \in H_\bullet^{T^\vee}(\Gr)
$$
where we use the action of $ H_{T^\vee}^\bullet(\Gr) $ on $ H^\bullet_{T^\vee}(\Gr, \IC_\lambda)$.

Thus,
$$
\theta(\psi_B^*(f_v)) = \sum_i p_*(\tau_\lambda(h^i v))h_i  = \sum_i  p_*(h^i \tau_\lambda(v)) h_i
$$
since  $\tau_\lambda $ is a $ H^\bullet_{T^\vee}(\Gr) $-module morphism. 

Now, since the pairing between $ H^\bullet_{T^\vee}(\Gr) $ and $ H_\bullet^{T^\vee}(\Gr) $ is given by $ a \otimes h \mapsto p_*(ah) $, we conclude that
\begin{equation} \label{eq:fvtau}
\theta(\psi_B^*(f_v)) = \tau_\lambda(v)
\end{equation}

The commutative diagram in Theorem \ref{th:HTGr} implies that the isomorphism $ \theta : \C[C] \rightarrow H_\bullet^{T^\vee}(\Gr) $ is a $\C[T]$-module isomorphism, where $ \C[T]$ acts on $ \C[C] $ using $ C \rightarrow T $ and acts on $ H_\bullet^{T^\vee}(\Gr) $ using the maps $(t^\mu)_* $.  Thus we see that (\ref{eq:fvtau}) along with $ f_v = \lambda \otimes \Psi_\lambda(v) $ implies the desired result.
\end{proof}

Finally, here is our promised alternate proof.
\begin{proof}[Alternate proof of Theorem \ref{th:DHequalsD2}]
	Let $ Z $ be a stable MV cycle.  Choose $ \lambda \in P_+ $ such that $ t^\lambda Z \subset \overline{\Gr^\lambda}$.  Then $ t^\lambda Z $ is an MV cycle of type $ \lambda $.  We consider its class $ [t^\lambda Z]_{IH} \in H^k(\Gr, \IC_\lambda) $.  By the definition of $ \tau_\lambda $, we have that $ \tau_\lambda([t^\lambda Z]_{IH}) = [t^\lambda Z] $ in $ H_{d-k}^{T^\vee}(\Gr) $ where $ d = 2\rho^\vee(\lambda)$. 
	
	Thus by Lemma \ref{le:tau}, we have $ \theta(\psi^*(b_Z)) = [Z]$ in $ H^{T^\vee}_{d-k}(\Gr)  $.
	
	By Theorem \ref{th:HTGr}, the map $H^{T^\vee}_\bullet(\Gr^{T^\vee}) \rightarrow H^{T^\vee}_\bullet(\Gr) $ is dual to $ C \rightarrow \ft^* \times T $.  Thus passing to $ \ftreg $ and inverting this map, we obtain the desired equality.
\end{proof}

\subsection{Proof of Muthiah's conjecture}
Let $ Z $ be a stable MV cycle.  As a corollary of Theorem \ref{th:DHequalsD2}, it is easy to see that $ \barD(b_Z) $ is given by equivariant multiplicities at the bottom of $ Z $.

More precisely, applying Remark \ref{rem:exponentials} and Theorem \ref{th:FTDH}, we immediately deduce the following.

\begin{corollary} \label{co:barDEM}Let $ Z $ be a stable MV cycle of weight $ -\nu $.
	We have the following equality in $ \C[\ftreg] $.
	$$ \barD(b_Z) = (\iota^{-1})^* \varep_{L_{-\nu}}^{T^\vee}(Z)
	$$
\end{corollary}

Now, we are in a position to prove Mutiah's conjecture, Theorem \ref{th:Dinakar}.  We begin by recalling the setup.  Let $ \lambda \in P_+ \cap Q $ and let $ Z $ be an MV cycle of type $ \lambda $ and weight $0 $, so $ t^{-\lambda} Z $ is a stable MV cycle of weight $ -\lambda $.

Note that we have an equality $\varep_{L_{-\lambda}}^{T^\vee}(t^{-\lambda}Z) =\varep_{L_{0}}^{T^\vee} (Z) $.  Thus, in light of Corollary \ref{co:barDEM}, in order to establish Theorem \ref{th:Dinakar} we are left to prove the following result.

\begin{theorem} \label{th:MuthiahMain}
	The map $L(\lambda)_0 \rightarrow \C[\ftreg] $ defined by $ v \mapsto \barD(\Psi_\lambda(v)) $ is $ W$-equivariant.
\end{theorem}

\begin{proof}
	Let $ w \in W $ and $ x \in \ftreg$.  Then by Proposition \ref{pr:BehNx}, there exist $ (y,t) \in N_-\times T $ such that
	$$ n_{w^{-1} x} = y n_x \overline w t.$$
	Hence, we have
$$
\barD(\Psi_\lambda(v))(w^{-1} x) = v_\lambda^*(n_{w^{-1} x} v) 
= v_\lambda^*(y n_x \overline w t v) 
= v_\lambda^*(n_x \overline w v) 
= \barD(\Psi_\lambda(\overline w v))(x)
$$
where we used that $ t v = v $, since $ v $ is of weight 0, and that $  v_\lambda^*(y u) = v_\lambda^*(u)$, since $ y \in N_- $.

Since this holds for all $w \in W$ and $ x \in \ftreg$, we conclude that $ v \mapsto \barD(\Psi_\lambda(v))$ is $W$-equivariant.
\end{proof}

\section{Preprojective algebra modules}
From this point on, we assume that $ G $ is simply-laced. In particular, this means that $ q(\alpha_i) = 1 $ for all $ i $ and thus $ \iota(\alpha_i) = \alpha^\vee_i $, so $\iota $ corresponds to the usual identification of the root and coroot lattices.  Thus, we can (and will) drop $ \iota $ from our notation without possibility of confusion.

\subsection{Preprojective algebras and their modules}
\label{ssec:preproj}
 Let $ H $ denote the set of oriented edges of the Dynkin diagram (so $ (i,j), (j,i) \in H $ whenever $ i,j $ are connected in $ I$).  If $ h = (i,j) $, write $ \bar{h} = (j,i) $.  Fix a map $ \tau : H \rightarrow \{1,-1\} $ such that for each $ h $, $ \tau(h) + \tau(\bar{h}) = 0 $ (such a $\tau $ corresponds to an orientation of each edge of the Dynkin diagram).

Define the \defn{preprojective algebra} $ \Lambda $ to be the quotient of the
path algebra of $ (I,H) $ by the relation $\sum_{h \in H} \tau(h) h \bar{h} = 0 $.

So a $ \Lambda$-module $ M $ consists of vector spaces $ M_i $ for $ i \in I $ and linear maps $ M_h : M_i \rightarrow M_j $ for each $ h = (i,j) \in H $, such that
\begin{equation} \label{eq:preproj}
\sum_{h \in H} \tau(h) M_h M_{\bar{h}} = 0 .
\end{equation}
Given a $ \Lambda$-module $ M$, we define its dimension vector by
$$
\dimvec M = \sum_{i \in I} (\dim M_i)\, \alpha_i
$$

We write $ S_i $ for the simple module at vertex $ i $, i.e.\ the module with $ \dimvec S_i = \alpha_i $.
The map $ M \mapsto \dimvec M $ gives rise to an isomorphism $ K(\Lambda\text{-mod}) \cong Q $.

For each $ \nu = \sum_{i \in I} \nu_i \alpha_i  \in Q_+ $, we consider the affine variety of $ \Lambda$-module structures on $ \oplus_{i \in I} \C^{\nu_i} $.  More precisely, we define
$$
\Lambda(\nu) \subset \bigoplus_{(i,j) \in H} \Hom(\C^{\nu_i}, \C^{\nu_j})
$$
to be the subvariety defined by the equation (\ref{eq:preproj}).

\subsection{The dual semicanonical basis}\label{ssec:dualsemi}
Let $ M$ be a $ \Lambda $-module with dimension vector~$ \nu $.
Following Lusztig \cite{Lusztig00} and Geiss--Leclerc--Schr\"oer \cite[\S 5]{GeissLeclercSchroer}, we define an element
$ \xi_M \in \C[N]_{-\nu} $ as follows.

First, for each $ \vi \in \Seq(\nu) $, we define the projective variety of composition series of type $ \vi$,
$$
F_\vi(M) = \bigl\{ 0=M^0 \subseteq M^1 \subseteq \dots \subseteq M^m = M \bigm| M^k/M^{k-1} \cong S_{i_k} \text{ for all } k \bigr\}
$$
and then we define $\xi_M \in \C[N]_{-\nu} $ by requiring that
$$ \langle e_{i_1} \cdots e_{i_m} , \xi_M \rangle = \chi(F_\vi(M))$$
for any $\vi \in \Seq(\nu)$, where $ \chi $ denotes topological Euler characteristic.
(Note: Lusztig and Geiss--Leclerc--Schr\"oer consider decreasing composition
series, whereas we chose to use increasing ones. Our choice accounts for
the use of the dual setup and ensures that the crystal structure on the
dual semicanonical basis coincides with the crystal structure defined in
\cite[\S 5]{KashiwaraSaito}.)

With this definition, the following result is immediate (see \cite[Lemma 7.3]{GeissLeclercSchroer}).  
\begin{lemma} \label{le:multXi}
For any $ \Lambda $-modules $M, N $, we have $ \xi_{M \oplus N} = \xi_M \xi_N $.
\end{lemma}

This map $ M \mapsto \xi_M $ is constructible and so for any component $ Y \subset \Lambda(\nu) $, we can define $ c_Y \in \ON_{-\nu} $ by setting $ c_Y = \xi_M $, for $ M $ a general point in $ Y $.

The following result is due to Lusztig \cite{Lusztig00}.
\begin{theorem} \label{th:LusztigPerfect}
\begin{enumerate}
\item
For each $ \nu \in Q_+$, $ \{ c_Y \mid Y \in \Irr \Lambda(\nu) \} $ is a basis for~$ \ON_{-\nu}$.
\item
The union of these bases forms a biperfect basis of $ \ON $.
\end{enumerate}
\end{theorem}
\begin{proof}
Statement (i) is a direct consequence of Theorem~2.7 in \cite{Lusztig00}.
From \S 2.9 in \textit{loc.\ cit.} and the definition of the bicrystal
structure on the set of irreducible components of the nilpotent varieties
(see \S 5 in \cite{KashiwaraSaito}), we deduce that for any irreducible
component $Y\subset\Lambda(\nu)$ and any $i\in I$, if we set
$p=\varepsilon_i^*(Y)$ and $Y'=(\tilde e_i^*)^p\;Y$, then
$c_Y\cdot e_i^{(p)}=c_{Y'}$. With these notations, write
$$c_Y\cdot e_i=\sum_{Y''}\alpha_{Y''}\;c_{Y''}.$$
Routine arguments show then that
$$\alpha_{Y''}=\begin{cases}
p&\text{if }Y''=\tilde e_i^*Y,\\
0&\text{if }\varepsilon_i^*(Y'')\geq{p-1}\text{ and }Y''\neq\tilde e_i^*Y,
\end{cases}$$
which proves the half of the statement (ii) related to the right action
of $N$ on $\ON$. The other half can be proved analogously or deduced
from Theorem~3.8 in \cite{Lusztig00}.
\end{proof}

This basis for $ \ON $ is called the dual semicanonical basis. By \S\ref{ss:UniqCrys}, it carries a bicrystal structure isomorphic to $B(\infty)$. Thus, a MV polytope is uniquely associated to each element in the dual semicanonical basis.

On the other hand, if $ M $ is a $\Lambda $-module, then we define its
\defn{Harder--Narasimhan polytope} to be
$$
\Pol(M) := \Conv\bigl\{-\dimvec N \bigm| N \subseteq M \text{ is a submodule}\bigr\}.
$$
This map $ M \mapsto \Pol(M) $ is constructible and so for any component $ Y \subset \Lambda(\nu) $, we can define $ \Pol(Y) := \Pol(M) $ for $ M $ a general point in $ Y$.  (We added the $ - $ sign into the definition since $ \xi_M $ has weight $ -\dimvec M $.)

The following result was obtained by the first two authors with Tingley (see \cite[\S 1.3]{BaumannKamnitzerTingley}).
\begin{theorem}\label{thm:BKT}
  Let $ Y $ be a component of $ \Lambda(\nu) $.  Then $\Pol(Y)$ is the
  MV polytope of the basis vector $c_Y$.
\end{theorem}

\subsection{Measures from $ \Lambda$-modules}\label{ssec:measmod}
Let $ M $ be a $ \Lambda $-module of dimension vector $ \nu $.  By the definition of $\xi_M $ and the map $ D : \ON \rightarrow \PP $, we have that
$$
D(\xi_M) =  \sum_{\vi \in \Seq(\nu)} \chi(F_\vi(M))\, D_\vi
$$
Note that the measure $ D(\xi_M) $ is supported on the polytope $ \Pol(M) $.

In the previous section (Theorem \ref{th:DHequalsD}), we showed that the measure $ D(b_Z) $ of an MV basis vector equals the push-forward of the Duistermaat--Heckman measure $ DH(Z) $ of the corresponding MV cycle. The Duistermaat--Heckman measure is defined as the asymptotics of sections of line bundles on $ Z$.  In a similar fashion, we will now explain that $ D(\xi_M) $ can also be regarded as an asympototic.

We define $ F_n(M) $ to be the space of $(n+1)$-step chains of submodules of $ M$, so
$$
F_n(M) = \{ 0=M^0 \subseteq M^1 \subseteq \dots \subseteq M^{n+1} = M \}
$$
and let $ F_{n,\mu}(M) $ denote the locus in $ F_n(M) $ where $ \dimvec M^1 + \dots + \dimvec M^n = \mu $.

We will record the information of the Euler characteristics of $ F_{n,\mu}(M) $ as an element of $ R(T) $ by
$$ [H^\bullet(F_n(M))] = \sum_{\mu \in Q_+} \chi(F_{n,\mu}(M)) \; \delta_{-\mu} $$
Note that $[H^\bullet(F_n(M))]$ is supported
on the polytope $ n\Pol(M) $.

\begin{theorem} \label{th:limLambda}
For any $ \Lambda$-module $ M$, with $ \nu = \dimvec M$, we have $$ D(\xi_M) = \lim_{n \rightarrow \infty} \frac{1}{n^{\rho^\vee(\nu)}} (\tau_n)_* [H^\bullet(F_n(M))] $$
\end{theorem}

\begin{proof}
  Let $ m = \dim M = \rho^\vee(\nu)$.

  The proof is largely parallel to that of \cite[\S 3]{LittDH}.
  We begin by defining the locally constant function
  \begin{align*}
  \mu_\bullet: F_n(M) &\rightarrow Q_+^{\ n} \\
M^\bullet  &\mapsto (\dimvec (M^k/M^{k-1}))_{k=1,\ldots,n}
\end{align*}

  For any $ M^\bullet \in F_n(M) $, the number of non-zero terms in $\mu_\bullet(M^\bullet) $ is at most $ m $.  Given a sequence $\gamma \in Q_+^{\ n}$,
  let $\gamma^{\neq 0}$ be the sequence with its $0$s removed. 
  Thus we can decompose
  $$ [H^\bullet(F_n(M))] = 
  \sum_{\ell=0}^{m} 
  \sum_{\theta \in (Q_+ \smallsetminus \{0\})^\ell}
  \sum_{\gamma \in Q_+^n \atop \gamma^{\neq 0} = \theta}
  \chi\left( \{M^\bullet \in F_n(M)\, :\, \mu_\bullet(M^\bullet) = \gamma \} \right)
  \, \delta_{-\sum_{k=1}^n (n+1-k)\, \gamma_k} $$
  Now observe that if $ \theta = \gamma^{\neq 0}$, we have an obvious isomorphism
  $$\{M^\bullet \in F_n(M)\ :\ \mu_\bullet(M^\bullet) = \gamma \} \ \cong\ 
  \{M^\bullet \in F_\ell(M)\ :\ \mu_\bullet(M^\bullet) = \theta  \}$$
  We let $ \chi_\theta $ denote the Euler characteristic 
  of this space.  Thus,
  $$ [H^\bullet(F_n(M))] = 
  \sum_{\ell=0}^{m} \ 
  \sum_{\theta \in (Q_+ \smallsetminus \{0\})^\ell} \chi_{\theta}
  \sum_{\gamma \in Q_+^n \atop \gamma^{\neq 0} = \theta}
  \delta_{-\sum_{k=1}^n (n+1-k)\, \gamma_k} $$
  The number of $\gamma$ with $\gamma^{\neq 0}=\theta$ is plainly ${n \choose \ell} = O(n^\ell)$.
  We now rescale
  $$ \frac{[H^\bullet(F_n(M))]}{n^{\rho^\vee(\nu)}} =
  \sum_{\ell=0}^{m}\ 
  \sum_{\theta \in (Q_+ \smallsetminus \{0\})^\ell} \chi_\theta \
  \frac{1} {n^m}
  \sum_{\gamma \in Q_+^n \atop \gamma^{\neq 0} = \theta}   \delta_{-\sum_{k=1}^n (n+1-k)\, \gamma_k}
  \qquad \text{using $ \rho^\vee(\nu) = m$}, $$
  and let $$ D_{\theta, n} := \frac{1}{n^{m}} \sum_{\gamma:\ \gamma^{\neq 0} = \theta}\, \delta_{-\sum_{k=1}^n (n+1-k)\, \gamma_k} $$
  This term has total mass $O(n^{\ell-m})$. In the limit $n\to \infty$,
  we can therefore neglect all $\ell < m$, and independently
  we apply $(\tau_n)_*$:
  $$ \lim_{n\to\infty} (\tau_n)_* \frac{[H^\bullet(F_n(M))]}{n^{\rho^\vee(\nu)}} =
  \sum_{\theta \in (Q_+ \smallsetminus \{0 \})^{m}} \chi_\theta   \lim_{n\to\infty}  (\tau_n)_* D_{\theta,n}
  $$
  Since $ \theta \in (Q_+ \smallsetminus \{0 \})^{m} $ (i.e. $ \theta $ has the same length as $ \dim M $), for the locus $\{ M^\bullet \in F_\ell(M)\ :\ \mu_\bullet(M^\bullet) = \theta \}$ to be non-empty, we need each $\theta_k $ to be a simple root and that $ \sum \theta_k = \nu $.  Such a
  $\theta$ uniquely determines (and is determined by) a sequence $\vi \in \Seq(\nu)$ with $\theta_k = \alpha_{i_k}$ for all $ k $.  Moreover, we have $\chi_\theta = \chi(F_\vi(M))$.

  Now that $\ell = m$, the term $ D_{\theta,n} $ has mass $1/m!$ as 
  $ n\to \infty$, the volume of the $m$-simplex.  
  We proceed to determine how that mass is distributed.

  To index the $n\choose m$ terms in the $D_{\theta,n} $ sum,
  we count how long the individual strings of $0$s are between
  the $m$ non-zero terms in $\gamma$: there are $1+m$ strings of $0$s,
  of total length $n - m$.  Thus the set of all $ \gamma $ such that $ \gamma^{\neq 0} = \theta $ are naturally in bijection with lattice points in
  the dilated simplex $(n-m)\Delta^{m}$.
  (We will soon apply $(\tau_n)_*$, resulting in the nearly-standard
  simplex $(1-\frac{m}{n})\Delta^{m}$.)
  The $j$th vertex of this simplex corresponds to the case that 
  $\gamma$ has $j$ nonzero terms up front, its $n-m$ zeros all in
  the middle, and $m-j$ nonzero terms at the end (its nonzero terms $\theta$
  determined by $\gamma$).

  Recall the linear map $ \pi_\vi : \R^{m+1} \to \tR $ which takes the $(k+1)$st standard basis vector  to the negative partial sum $ -\beta^\vi_k = -(\theta_1 + \dots + \theta_k) $.  The map $ \pi_\vi $ intertwines the above bijection with the map $ \gamma \mapsto -\sum_{k=1}^n (n+1-k)\, \gamma_k + \zeta $ where $ \zeta:= \sum_{k=1}^m (m+1-k)\, \theta_k$ is a shift which is independent of $ n$.  Thus, we obtain that
  $$
 D_{\theta, n} = \sum_{x \in (n-m)\Delta^m \cap \Z^{m+1}} \delta_{\pi_\vi(x) - \zeta}
 $$
 and thus $ \lim_{n\to \infty}  (\tau_n)_* \frac{1}{n^m} D_{\theta, n} = D_\vi $ as desired.

\end{proof}

\section{Comparing measures from MV cycles and
  from $\Lambda$-modules}
\subsection{From measures to sections} \label{se:measec}
Let $ Z $ be a stable MV cycle of weight $ \nu$.  Let $ Y$ be an irreducible component of $ \Lambda(\nu) $.  We say that $Y $ and $ Z$ \textbf{correspond} if $ \Pol(Y) = \Pol(Z)$; in other words, if the corresponding basis elements $c_Y$ and $b_Z$ give the same element in the bicrystal $B(\infty)$.

For the remainder of this section, fix a pair $ Y, Z $ which correspond.  Recall that the measures $ D(c_Y) $ and $ D(b_Z) $ are both supported on $  \Pol(Y) = \Pol(Z)$.  Thus an enhancement of the equality of polytopes would be the equality of measures.  Note that the equality of basis vectors $ c_Y =  b_Z $ would imply the equality of measures $ D(c_Y) = D(b_Z) $, but not vice versa (because of Remark \ref{rem:NotInjective}).

By Theorems \ref{th:DHequalsD} and \ref{th:limLambda}, we see that each of the measures $ D(c_Y) $ and $ D(b_Z) $ are the limits of (scalings of) measures $ [H^\bullet(F_n(M))] $ (where $ M$ is a general point of $ Y $) and $ [\Gamma(Z, \mathscr L^{\otimes n})] $.  This motivates the following definition.

\begin{definition}
We say that $ Y $ and $ Z $ are \textbf{extra-compatible} if for all $ n \in \N $ and $ \mu \in Q_+$, we have
$$
\dim \Gamma(Z, \mathscr L^{\otimes n})_{-\mu} = \chi(F_{n,\mu}(M))
$$
where $M $ is a general point of $ Y $.
\end{definition}

The following is clear from the above results.
\begin{proposition} \label{prop:4statements}
Consider the following four statements
\begin{enumerate}
\item \label{st1} $ Y $ and $ Z$ are extra-compatible.
\item \label{st2} $c_Y = b_Z $.
\item \label{st3} $ D(c_Y) = D(b_Z) $.
\item \label{st4} $ \barD(c_Y) = \barD(b_Z) $.
\end{enumerate}
We have the implications
$$
(\ref{st1}) \Rightarrow (\ref{st3}) \quad (\ref{st2}) \Rightarrow (\ref{st3}) \Rightarrow (\ref{st4})
$$
\end{proposition}

\subsection{General conjecture} \label{se:conj}
We can translate the Euler characteristic of $ F_{n, \mu}(M) $ into the Euler characteristic of another variety.  Consider the algebra $ \Lambda[t] := \Lambda \otimes_\C \C[t] $.  We define
$$
\mathbb G_{\mu}(M[t]/t^n) = \bigl\{ N \subset M \otimes \C[t]/t^n \bigm| N \text{ is a $\Lambda[t]$-submodule, }\, \dimvec N = \mu \bigr\},
$$

\begin{lemma}
For any $ M, n, \mu $, we have $ \chi \left( \mathbb G_{\mu}(M[t]/t^n) \right) = \chi(F_{n, \mu}(M)) $.
\end{lemma}

\begin{proof}
First, define an inclusion $ F_{n, \mu}(M) \rightarrow \mathbb G_\mu(M[t]/t^n) $ by
$$
(M^1, \dots, M^n) \mapsto M^1\otimes \C t^0 \oplus M^2 \otimes \C t^1 \oplus \cdots \oplus M^n \otimes \C t^{n-1} \subset M\otimes \C[t]/t^n
$$
It is easy to see that the right hand side really is a $ \Lambda[t]$-submodule.

On the other hand, we can define an action of $ \C^\times $ on $Gr_{n,\mu}(M) $ using the action of $ \C^\times $ on $ \C[t]/t^n $ given by $ s \cdot p(t) = p(st) $.  It is easy to see that the above map gives an isomorphism
$$F_{n, \mu}(M) \rightarrow  \mathbb G_\mu(M[t]/t^n)^{\C^\times} $$
and so the result follows.
\end{proof}

Thus we deduce that $Y $ and $ Z$ are extra-compatible if for all $n \in \N,\mu \in Q_+$ we have
$$
\dim \Gamma(Z, \mathcal L^{\otimes n})_{-\mu} = \chi \left( \mathbb G_\mu(M[t]/t^n) \right)
$$
where $ M $ is a general point of $ Y $.

If we assume that the odd cohomology of $\mathbb G_\mu(M[t]/t^n) $ vanishes, this implies that there is an equality of $ T^\vee$-representations,
\begin{equation} \label{eq:AllenConj}
\Gamma(Z, \mathscr L^{\otimes n}) = H^\bullet(\mathbb G(M[t]/t^n)), \quad \text{ for all } n \in \N,
\end{equation}
where $T^\vee $ acts on the right hand side through the decomposition $ \mathbb G(M[t]/t^n) = \sqcup \mathbb G_\mu(M[t]/t^n) $.

\begin{remark}
The right hand side of (\ref{eq:AllenConj}) carries a cohomological grading.  We expect that (up to appropriate shift) this matches the $ \Z $-grading on the left hand side which comes from the loop rotation $ \C^\times $ action on $ Z$.
\end{remark}

The left hand sides of (\ref{eq:AllenConj}) form the components of a graded algebra, so it is natural to search for a similar structure on the right hand side.  After studying this question for some time, we are pessimistic about finding this algebra structure.  On the other hand, $\displaystyle{ \bigoplus_n } \Gamma(Z, \mathscr L^{\otimes n}) $ is also a graded module over the ring
$$ \bigoplus_n \Gamma\left(\overline{S_+^0 \cap S^{-\nu}_-}, \mathscr L^{\otimes n}\right) $$
We believe that such a module structure naturally exists for the direct sums of the right hand side of (\ref{eq:AllenConj}).  In fact, we believe that this structure is present for any module $M $, not just general modules corresponding to extra-compatible components.

\begin{conjecture}\label{conj:sheaf}
  For any preprojective algebra module $ M $ of dimension vector $ \nu $, the $ \mathbb Z_{\geq0} \times P $-graded vector space
  $$
  \bigoplus_{n \in \N}  H^\bullet(\mathbb G(M[t]/t^n))
  $$
  carries the structure of a $T^\vee$-equivariant graded $  \bigoplus_n \Gamma\left(\overline{S_+^0 \cap S^{-\nu}_-}, \mathcal L^{\otimes n}\right) $-module.
\end{conjecture}

If we assume this conjecture, then we get a  coherent sheaf $ \mathcal F_M$  on $ \overline{S_+^0 \cap S^{-\nu}_-} $ such that for large enough $n $,
$$
 \Gamma\left(\overline{S_+^0 \cap S^{-\nu}_-}, \mathcal F_M \otimes \mathscr L^{\otimes n}\right) \cong  H^\bullet(\mathbb G(M[t]/t^n)).
$$
as $ T^\vee$-representations.  Assuming the vanishing of odd cohomology, Theorem \ref{th:limLambda} implies that 
$ DH(\mathcal F_M) =  D(\xi_M) $.

On the other hand, we have the support cycle $ [\supp \mathcal F_M] = \sum a_k[Z_k] $ where $ Z_k $ ranges over the stable MV cycles of weight $ -\nu $.  By Theorem \ref{th:limCoherent}, we know that $ DH(\mathcal F_M) = \sum a_k DH(Z_k) $.  Thus by Theorem \ref{th:DHequalsD}, we reach $
\sum a_k D(b_{Z_k}) = D(\xi_M) $.

This suggests that $ \xi_M = \sum a_k b_{Z_k} $ in $\C[N]$.  In conclusion, the expansion of $ \xi_M $ in the MV basis should be given by the support cycle of $ \mathcal F_M$.

The conjecture extends to direct sums $M = \oplus_{k=1}^d M_k$ of
$\Lambda$-modules. Such an $M$ carries a $(\C^\times)^d$-action with
$\mathbb G(M[t]/t^n)^{(\C^\times)^d} \iso \prod_{k=1}^d \mathbb G(M_k[t]/t^n)$.  We will explain a conjectural relation between the sheaf $ \mathcal F_M $ and the sheaves $ \mathcal F_{M_k} $.  For this we will use the Beilinson--Drinfeld Grassmannian.  In \S \ref{ssec:BD}, we recalled the definition of $ \Gr_\A $, a family over $ \A $.  In a similar fashion, there is the BD Grassmannian $ \Gr_{\A^d} $ defined by $G$-bundles trivialized away from a collection of $ d $ points.  (There is a small difference: in \S \ref{ssec:BD}, we fixed one of the points to be 0; here we let all the points vary.)  As in \S \ref{ssec:BD}, the fibre of $ \Gr_{\A^d} \rightarrow \A^d $ over a point $ (x_1, \dots, x_d) $ is isomorphic to a product of copies of $ \Gr $, indexed by the set $ \{x_1, \dots, x_d\} $.

\begin{conjecture}\label{conj:BDsheaf}
  For any tuple $(M_k)_{k=1,\ldots,d}$ of $\Lambda$-modules,
  there is a $T^\vee $-equivariant sheaf $\mathcal F_{(M_k)}$ on $ \Gr_{\A^d} $,
  flat over $\A^d$, that
  \begin{enumerate}
  \item has global sections $ \Gamma( \Gr_\A, \mathcal F_{(M_k)} \otimes \mathscr L^{\otimes n}) \cong H^\bullet_{(\C^\times)^d} (\mathbb G(\oplus_k M_k[t]/t^n))$ as representations of $ T^\vee $ and as modules over $ \C[\A^d] = H^\bullet_{(\C^\times)^d}(pt) $,
  \item over points on the diagonal, restricts to the sheaf $ \mathcal F_{\oplus_k M_k} $ from Conjecture \ref{conj:sheaf}
    associated to $\oplus_k M_k$
  \item over general points of $\A^d$, restricts to
    the $\boxtimes$ product $ \boxtimes_k \mathcal F_{M_k} $ of the sheaves from Conjecture \ref{conj:sheaf}
    associated to the individual $M_k$.
  \end{enumerate}
\end{conjecture}

The simplest case is that each $M_k$ is one-dimensional and so there exists $ \vi \in \Seq(\nu) $ such that $ M_k = S_{i_k} $.  In this case, we expect that $ \mathcal F_{M_k} = \mathcal O_{\overline{S_+^0 \cap S^{-\alpha_{i_k}}_-}} $ for each $k $ and that $ \mathcal F_{(M_k)} = \mathcal O_{\overline{Z^\nu}}$, the structure sheaf of the compatified Zastava space.  In fact, we will explain in \cite{HKW}, that in this case, the conjecture follows from Remark 3.7 in Braverman--Finkelberg--Nakajima \cite{BFN}.

\subsection{Shifting MV cycles}
In order to compute the sections of line bundles over stable MV cycles, it will sometimes be useful to shift them, since they often appear more naturally as MV cycles in some $ \overline{\Gr^\lambda} $.  So the following result will be helpful for us.

\begin{proposition} \label{pr:shiftWeights}
	Let $ Z$ be any MV cycle.  Let $ \nu \in P $ and $ n \in \N $.   We have an isomorphism of $ T^\vee $-representations
	$$
	\Gamma(Z, \mathscr L^{\otimes n}) =  \Gamma(t^\nu Z, \mathscr L^{\otimes n}) \otimes \C_{n \iota(\nu)}
	$$
\end{proposition}
\begin{proof}
	We can lift $ t^\nu $ to an element of $ E(G^\vee(\mathcal K)) $.  Using this lift, the isomorphism $ Z \rightarrow t^\nu Z $ extends to an isomorphism of line bundles and thus an isomorphism of sections.  However, this isomorphism is not $ T^\vee$-equivariant, because $ T^\vee $ and $ t^\nu $ do not commute inside $ E(G^\vee(\mathcal K)) $. 
	
	Let $ \mu \in P $ and let $ a\in \C^\times $.  In $ E(G^\vee(\mathcal K)) $, the commutator $ t^\nu a^\mu t^{-\nu} a^{-\mu} $ lies in the central $ \mathbb C^\times $ and equals $ \iota(\nu)(\mu) $.  Since this central $ \mathbb C^\times $ acts by weight $ n $ on $  \mathscr L^{\otimes n} $ the result follows.
\end{proof}

\subsection{Spherical Schubert varieties}\label{ssec:spherical}
Let $ \lambda \in P_+$  and consider the spherical Schubert variety $ \overline{\Gr^\lambda}$.  We shift it to form the stable MV cycle $ t^{-\lambda} \overline{\Gr^\lambda} $, a component of $ \overline{S_+^0 \cap S_-^{-\lambda + w_0 \lambda}} $.  The MV polytope of $t^{-\lambda} \overline{\Gr^\lambda} $ is the shifted Weyl polytope $$ \Conv( w \lambda \mid w \in W) - \lambda. $$

The corresponding $\Lambda$-module is injective.  More precisely, for each $ i \in I $, let $ I(\omega_i)$ be the injective hull of the simple module $ S_i $.  Let $ I(\lambda) = \oplus_i I(\omega_i)^{\oplus \langle \alpha_i^\vee, \lambda \rangle} $.  This is a rigid module and so the closure of the corresponding locus (consisting of those module structures isomorphic to $M$) in $ \Lambda(\lambda - w_0 \lambda) $ is an irreducible component $ Y(\lambda) $.

We have equalities of basis vectors
$$
b_{t^{-\lambda}\overline{\Gr^\lambda}}= \Psi_\lambda(v_{w_0 \lambda}) 
= c_{Y(\lambda)}
$$
as both are flag minors (see Remark \ref{re:flag}).

We conjecture that $ t^{-\lambda}\overline{\Gr^\lambda} $ and $ Y(\lambda) $ are extra-compatible in the above sense, in other words for all $n \in N $ and $ \mu \in Q_+$, we have
$$
\dim \Gamma(t^{-\lambda}\overline{\Gr^\lambda}, \mathscr L^{\otimes n})_{-\mu} = \chi(F_{n, \mu}(I(\lambda)))
$$

We will now prove a stronger version of this statement when $ n = 1$.

Consider the Nakajima quiver variety $ \mathcal{M}(\mathbf{w}) $, defined using the framing vector $ \mathbf{w} $, with $ w_i = \langle \alpha_i^\vee, \lambda \rangle $.  This quiver variety has a ``core'' $ \mathcal L(\mathbf w) $ and there is a homotopy retraction of $ \mathcal M(\mathbf w) $ onto $ \mathcal L(\mathbf w)$.  We have the following result of Shipman \cite[Corollary 3.2]{Ship}.

\begin{theorem}
There is an isomorphism $ \mathcal L(\mathbf w) \cong F_1(I(\lambda))$.
\end{theorem}

Thus we get a chain of isomorphisms
$$
H^\bullet(\mathcal M(\mathbf w)) \cong H^\bullet(\mathcal L(\mathbf w)) \cong H^\bullet(F_1(I(\lambda))
$$

By work of Varagnolo, there is an action of $\mathfrak g^\vee[[t]] $ on $ H^\bullet(\mathcal M(\mathbf w))$.  Since $ \Gr^\lambda $ is an orbit of $ G^\vee(\mathcal O) $ we also have a $ \mathfrak g^\vee[[t]]$ action on $  \Gamma(\overline{\Gr^\lambda}, \mathscr L)$.  The following result is essentially due to Kodera--Naoi \cite{KN} and Fourier--Littelmann \cite{FL} (see \cite[Theorem 8.5]{KTWWY}).

\begin{theorem}
  $ \Gamma(\overline{\Gr^\lambda}, \mathscr L) \cong  H^\bullet(F_1(I(\lambda))$
  as representations of $ \mathfrak g^\vee[[t]]$.
\end{theorem}

 Unpacking the weight spaces on both sides (and using the odd cohomology vanishing established by Nakajima \cite[Prop. 7.3.4]{Nak}), the above theorem implies the $n =1 $ condition appearing in the definition of extra-compatibility.  (Note that on both sides the weight spaces get shifted by $ \lambda $.  On the left hand side, this is because of Proposition \ref{pr:shiftWeights} and on the right hand side, this is because of the definition of the action of the Cartan in the Varagnolo action.)
 \begin{corollary}
$$
\dim \Gamma(t^{-\lambda}\overline{\Gr^\lambda}, \mathscr L)_{-\mu} = \chi(F_{1,\mu}(I(\lambda)))
$$
\end{corollary}

More generally, it seems reasonable that $ H^\bullet(F_n(I(\lambda)) $ should carry an action of $ \mathfrak g^\vee[[t]]$, extending the torus action (which comes from the decomposition $ F_n(I(\lambda)) = \sqcup_\mu F_{n,\mu}(I(\lambda)) $.

More generally, we expect that the extra compatibility extends to those MV cycles and the quiver variety components which represent the flag minors from Remark \ref{re:flag}.

\subsection{Schubert varieties inside cominuscule Grassmannians} \label{se:Schubert}
We now examine a class of MV cycles which we can prove are extra-compatible: Schubert varieties inside cominuscule Grassmannians.  These represent flag minors from minuscule representations.

Let $ \omega_i \in P_+ $ be {\bf cominuscule}, meaning, $ \omega_i $ is a minimal element of $ P_+ \setminus \{0\}$ with respect to the dominance order.  In this case, it is easy to see that $ \Gr^{\omega_i} $ is closed and that the action of $ G^\vee $ on $ \Gr^{\omega_i} $ gives rise to an isomorphism $ G^\vee / P_i \cong \Gr^{\omega_i} $ where $ P_i $ is a maximal parabolic subgroup of $ G^\vee$.

Moreover, the MV cycles of type $\omega_i $ are the Schubert varieties in $G^\vee/P_i $. For each $ \gamma \in W \omega_i $, we have  a Schubert variety $ \Gr^{\omega_i} \cap \overline{S_-^\gamma} $.  We can then translate to obtain a stable MV cycle $ Z(\gamma) := t^{-\omega_i} (\Gr^{\omega_i} \cap \overline{S_-^\gamma)}$ of weight $ \gamma - \omega_i$.

We introduce a partial order (the Bruhat order) on $ W \omega_i $ by $ \tau \le \gamma $ if $ \tau - \gamma \in Q_+ $ (caution: this is opposite to our convention for dominance order).  This partial order corresponds to the order on MV cycles of type $ \omega_i $ by containment.  The minimal element of this partial order is $ \omega_i $ and the maximal element is $ w_0 \omega_i $.  We will be particularly interested in intervals in this poset of the form
$$ [\omega_i, \gamma] = \{ \tau \in W \omega_i : \omega_i \le \tau \le \gamma \} $$
The MV polytope $P(\gamma) $ of $Z(\gamma) $ is easily described using this order (see for instance \cite[Proposition~2.5.1]{KatoNaitoSagaki}) as
$$ P(\gamma) = \Conv \{ -\omega_i + \tau : \tau \in [\omega_i, \gamma] \}$$

On the other hand, for each $ \gamma \in W \omega_i$, there is a unique (up to isomorphism) $\Lambda$-module $ N(\gamma) $ with dimension-vector $ \omega_i - \gamma $ and socle $ S_i $ (except if $\gamma = \omega_i $ in which case $ N(\gamma) = 0 $).   We have a corresponding component $ Y(\gamma) $ of $ \Lambda(\omega_i - \gamma) $.  Note that $ Y(\gamma) $ and $ Z(\gamma) $ correspond because they both represent the unique elements of $ B(\infty) $ of weight $ \gamma - \omega_i $ which lies in the image of $ \Psi_{\omega_i} $.

As a special case, we have $ N(w_0\, \omega_i) = I(\omega_i) $, the injective hull of $ S_i $.  In fact,
for each $ \gamma \in W \omega_i$, $ N(\gamma) $ occurs once as a submodule of $ I(\omega_i) $ and these are all the submodules of $ I(\omega_i) $.  The map $ \gamma \mapsto N(\gamma) $ is an isomorphism of posets between $ W \omega_i $ and the poset of submodules of $ I(\omega_i) $ (under inclusion).

This implies that $ \Pol(N(\gamma)) = P(\gamma) $.

\begin{theorem} \label{th:cominusExtra}
Let $ \gamma \in W \omega_i $.  The pair $ Z(\gamma), Y(\gamma) $ is extra-compatible.
\end{theorem}

\begin{proof}
For each $ n \in \mathbb N $, let
$$ F_n(\gamma) := \{ (\tau_1, \dots, \tau_n) : \omega_i \le \tau_1 \le \cdots \le \tau_n \le \gamma \}$$
be the set of chains in the poset $ [\omega_i, \gamma] $.

From the above discussion, we see that $ \tau \mapsto N(\tau) $ gives a bijection $ F_n(\gamma) = F_n(N(\gamma)) $.

On the other hand, Seshadri \cite{Se} has given a bijection between $ F_n(\gamma) $ and the standard monomial basis for $ \Gamma(Z(\gamma), \mathscr L^{\otimes n}) $ (see also \cite[Theorem 8.1.0.2]{LR}) .

\end{proof}

\begin{remark}
 More generally, we can consider a pair $ \tau, \gamma \in W \omega_i$ such that $ \tau \le \gamma $.  Then we can consider the translated Richardson variety
 $$ Z(\tau, \gamma) := t^{-\tau} (\Gr^{\omega_i} \cap \overline{S^\tau} \cap \overline{S_-^\gamma)}) $$
 This will also be an MV cycle. (In the case of type A, these MV cycles were studied by Anderson--Kogan \cite[\S 2.6]{AndersonKogan}.)

 The corresponding $\Lambda$-module is $ N(\gamma) / N(\tau) $ which underlies a component $ Y(\tau, \gamma) $ of $ \Lambda(\tau - \gamma)$.  The above analysis carries over to this case by considering chains in the interval $[\tau, \gamma] $.  Thus we see that the pair $Z(\tau, \gamma), Y(\tau, \gamma) $ is also extra-compatible.
\end{remark}

\begin{appendix}
\newpage

\section{Extra-compatibility of MV cycles and preprojective algebra modules and non-equality of bases, by Anne Dranowski, Joel Kamnitzer, and Calder Morton-Ferguson}
\subsection{Tableaux, MV cycles, and preprojective algebra modules}
In this appendix, we will work with $ G = SL_m, G^\vee = PGL_m$ (and in fact with $ m = 5,6$).  Our goal is to compare the MV and dual semicanonical basis elements $ b_Z $ and $ c_Y $ where $ Z, Y $ correspond in the sense of \S \ref{se:measec}.  In order to construct compatible pairs $ Z, Y $, we will work with the combinatorics of semistandard Young tableaux.

\subsubsection{Tableaux and Lusztig data}
We identify $ P = \Z^m / \Z(1, \dots, 1) $ and $P_+ = \{ \lambda = (\lambda_1, \dots, \lambda_m) \in \Z^m \mid \lambda_1 \ge \cdots \ge \lambda_m = 0 \} $ and we identify $ P_+ $ with the
set of Young diagrams having at most \(m-1\) rows.  We also identify $ Q = \{ (\nu_1, \dots, \nu_m) \in \Z^m \mid \nu_1 + \cdots + \nu_m = 0 \} $ and we have the simple roots $ \alpha_i = \varep_i - \varep_{i+1} $.

We will write the positive roots for $ SL_m $ as $ \Delta_+ = \{ \varep_i - \varep_j \mid 1 \le i < j \le m \} $.
For any choice of convex order for $ \Delta_+ $, there is a bijection $ \psi : B(\infty) \rightarrow \mathbb N^{\Delta_+} $ constructed by Lusztig (see \cite[\S 2]{Lusztig}). In this appendix, we will work with one choice of convex order
\begin{equation}
\label{eq:order}
\varep_1 - \varep_2 \le  \dots \le \varep_1 - \varep_m \le \varep_2 - \varep_3 \le \cdots \le  \varep_2 - \varep_m \le  \cdots \le \varep_{m-1} - \varep_m
\end{equation}
(this convex order corresponds to the reduced word $ s_1 \cdots s_{m-1} \cdots s_1 s_2 s_1 $ for $w_0$).
A \textbf{Lusztig datum} is an element $ n_\bullet \in \mathbb N^{\Delta_+}$.  The \textbf{weight} of $ n_\bullet $ is defined to be $ \sum_{\beta \in \Delta_+} n_\beta \beta $.

Fix $\lambda \in P_+$ and let $ N = \sum \lambda_i $.  Let \(YT(\lambda)\) be the set of semi-standard Young tableaux of shape \(\lambda\). We say that a tableau \(\tau\in YT(\lambda)\) has weight \(\mu \in \Z^m \)
if \(\tau\) has \(\mu_i\) boxes numbered \(i\) for \(1\le i\le m\).  Let \(YT(\lambda)_\mu \subset YT(\lambda) \) denote the tableaux having weight \(\mu\).
For \(1\le i\le m\), let \(\tau^{(i)}\) denote the restricted tableau obtained from \(\tau\) by deleting all boxes numbered \(j\), for \(j>i\). Let $ \sh(\tau^{(i)}) \in \N^m $ denote the shape of \(\tau^{(i)}\).

Given a tableau \(\tau\in YT(\lambda)_\mu\) we define its Lusztig datum \(n(\tau)_\bullet\)  by setting
\[
n(\tau)_{\varep_i - \varep_j}:= \sh(\tau^{(j)})_i - \sh(\tau^{(j-1)})_i = \text{ number of boxes on the $i$-row filled with $ j$}
\]
Note that $ n(\tau)_\bullet $ has weight $ \lambda - \mu $.
\subsubsection{Generalized orbital varieties}
\label{ss:ss11}
We will study MV cycles by identifying open affine subsets of them with generalized orbital varieties.  This identification comes from the Mirkovi\'c--Vybornov isomorphism, which we now recall.  Fix $ \lambda \in P_+$ and let $ N = \sum \lambda_i $.  Fix $ \mu \in P_+ $ such that $ \lambda - \mu \in Q_+ $ (in particular $ N = \sum \mu_i$). We will work inside the space of $N\times N $ matrices and we consider such matrices in block form, where we have $m \times m $ blocks, with the $ i,j$ block of size $ \mu_i \times \mu_j $.  

We will need the following spaces of matrices.
\begin{align*}
\mathfrak n &= \text{ upper triangular matrices } \\
\mathbb O_\lambda &= \text{ matrices of Jordan type $\lambda$ } \\
\mathbb T_\mu &= \{J_\mu + x : \text{all entries of $ x $ are 0, except the } \\
&\qquad \text{first \(\min(\mu_i,\mu_j)\) columns of the last row of each
\( i,j \) block} \} 
\end{align*}
where $ J_\mu $ is the Jordan form matrix of type $ \mu $.

\begin{example}
	\label{ex2}
	If \(\mu = (3,2)\) then an element of \(\mathbb T_\mu\) takes the form
	\[
	\left[
	\begin{BMAT}(e)[5pt]{ccc|cc}{ccc|cc}
	& 1 &  &   &   \\
	&  & 1 &   &   \\
	*  & * & * & * & *  \\
	&   &   &  & 1  \\
	* & * &  & * & * \\
	\end{BMAT}
	\right]
	\]
	where all blank entries are 0.
\end{example}
We will study \(\overline{\mathbb O_\lambda}\cap \mathbb T_\mu \) and \(\overline{\mathbb O_\lambda}\cap \mathbb T_\mu \cap \mathfrak n\). The irreducible components of this latter variety are called \textbf{generalized orbital varieties}.  Note that if $ \mu = (1, \dots, 1)$, then $ \mathbb T_\mu $ is the space of all matrices and so we recover the usual orbital varieties.

Given \(\tau\in YT(\lambda)_\mu\) define

\begin{equation}
\begin{aligned}
\label{eq:zo}
\Zoo_\tau &= \bigl\{A\in \overline{\mathbb O_\lambda}\cap \mathbb T_\mu \cap \mathfrak n \bigm| A\big|_{\C^{\mu_1 +\cdots + \mu_i}} \in \mathbb O_{\sh(\tau^{(i)})} \textrm{ for } 1\le i\le m \bigr\} 
\end{aligned}
\end{equation}
Here and elsewhere, \(\C^{k}\) denotes the subspace of \(\C^N\) spanned by the first \(k\) standard basis elements.

The following result is due to the first appendix author.
\begin{theorem}[\cite{D18}]
	\label{thm:s1t1}
	\begin{enumerate}
		\item For each $ \tau $, $ \Zoo_\tau $ has a unique irreducible component of dimension $ \rho^\vee(\lambda - \mu) $ and all other components have smaller dimension.
		\item Letting $ \Zo_\tau $ denote the closure of this component of maximal dimension, the map $ \tau \mapsto \Zo_\tau $ gives a bijection between $ YT(\lambda)_\mu $ and the set of irreducible components of $ \overline{\mathbb O_\lambda}\cap \mathbb T_\mu \cap \mathfrak n$.
	\end{enumerate}
\end{theorem}

The affine space $ \mathbb T_\mu $ (and the subvariety $ \overline{\mathbb O_\lambda}\cap \mathbb T_\mu \cap \mathfrak n$) carries an action of the maximal torus $ T^\vee \subset PGL_m $ where $ [t_1, \dots, t_m] \in T^\vee $ acts by conjugating by the $N \times N $ diagonal matrix whose entries $ t_1, \dots, t_1, \dots,t_m, \dots, t_m $ are constant in each block.

In this appendix, we will work with polynomials in the weights of $T^\vee $.  In particular, we introduce the notation
$$
p(\mu) := \prod_{1\le i < j \le m}(\varep_i - \varep_j )^{\mu_j} =  \prod_{1\le i < j \le m}(\alpha_i + \cdots + \alpha_{j-1} )^{\mu_j} 
$$
for the product of the weights of $ T^\vee $ acting on the affine space $ \mathbb T_\mu \cap \mathfrak n $.

\subsubsection{Generalized orbital varieties and MV cycles} \label{se:Ztau}

We will now recall the Mirkovi\'c--Vybornov isomorphism \cite{MVy} in the form of \cite{CK18}.

Given a nilpotent $ N \times N $ matrix $ A \in \mathbb T_\mu $, we define $ \varphi(A) \in G^\vee_1[t^{-1}] t^\mu $ (where $ G^\vee = PGL_m $ and $ G^\vee_1[t^{-1}] $ denotes the kernel of $ G^\vee[t^{-1}] \rightarrow G^\vee $) by the formulae
\begin{equation*}
\begin{aligned}
\varphi(A) &= t^\mu I_m + (a_{ij}(t)) \\
a_{ij}(t) &= -\sum_{k=1}^{\mu_i} A_{ij}^k t^{k-1} \\
A_{ij}^k &= k\textrm{th entry on the last row of the } j,i \textrm{ block of }A
\end{aligned}
\end{equation*}

The following result follows from Theorem 3.1 of \cite{CK18} (except we have applied transpose to the domain).

\begin{theorem} \label{thm:iso}
	The map $ A \mapsto [\varphi(A)] $ gives an isomorphism
	\[
	\varphi:\overline{\mathbb O_\lambda} \cap \mathbb T_\mu \rightarrow \overline{\Gr^{\lambda}} \cap G^\vee_1[t^{-1}]L_\mu
	\]
\end{theorem}
Note that  $ A \mapsto [\varphi(A)] $  is $ T^\vee $-equivariant with respect to the above defined action of $T^\vee $ on $ \mathbb T_\mu $ and the usual $ T^\vee $ action on $ \Gr $.

In \cite{D18}, the first appendix author restricted this isomorphism to the generalized orbital varieties and obtained the following result.

\begin{theorem}
	\label{thm:s1t2}
	\begin{enumerate}
		\item  $ \varphi $ restricts to an isomorphism $$ \varphi : \overline{\mathbb O_\lambda} \cap \mathbb T_\mu \cap \mathfrak n \rightarrow \overline{\Gr^{\lambda}} \cap S^\mu_- $$
		\item For any tableau $ \tau \in YT(\lambda)_\mu $, the closure of the image of the generalized orbital variety, $ Z_\tau := \overline{\varphi(\Zo_\tau)} $, is the MV cycle whose Lusztig datum equals $ n(\tau)_\bullet $.  In other words, the following diagram commutes
		\begin{equation*}
		\xymatrix{ YT(\lambda)_\mu \ar[r] \ar[dd] &\mathcal Z(\lambda)_\mu \ar[d] \\
			& \mathcal Z(\infty) \ar[d] \\
			\mathbb N^{\Delta_+} & B(\infty) \ar[l]
		}
		\end{equation*}
		where the top horizontal arrow is $ \tau \mapsto Z_\tau$, the left vertical arrow is $ \tau \mapsto n(\tau)_\bullet $, the right vertical arrows are constructed in \S \ref{se:MVBasisAlg} and the bottom arrow is Lusztig's bijection.
	\end{enumerate}
	
\end{theorem}

We will use this theorem to compute the value of $ \barD$ on an MV basis vector, with the aid of the following Proposition. 

\begin{proposition} \label{pr:Dmdeg}
	We have
	$$
	\barD(b_{t^{-\lambda} Z_\tau}) = \varep_{L_{\mu - \lambda}}(t^{-\lambda}Z_\tau) = \varep_{L_\mu}(Z_\tau) = \frac{\mdeg[\mathfrak n \cap \mathbb T_\mu] \Zo_\tau}{p(\mu)}
	$$
\end{proposition}
\begin{proof}
	The first equality is Corollary \ref{co:barDEM} and the third comes from Proposition \ref{prop:mdegeqvtmult}.
\end{proof}

\subsubsection{A Pl\"ucker embedding}
\label{ss:ss12}
We will now explain how the knowledge of $ \Zo_\tau $ allows us to compute the sections of line bundles over $ Z_\tau $.

Recall that in the lattice model for \(\Gr\) the orbit \(\Gr^\lambda\) can be described as
\[
\Gr^\lambda = \bigl\{L\in\Gr \bigm| L\subset L_0= \C\xt^m \textrm{ and } t\big|_{L_0/L}\textrm{ has Jordan type }\lambda\bigr\}
\]
Let \(p\ge\lambda_1\).  Thus \(L\in\overline{\Gr^\lambda}\) contains \(t^p L_0\) and the quotient $ L / t^p L_0 $ has dimension $ mp - N $.

\begin{proposition}
	The map
	\[
	\overline{\Gr^\lambda}\to \Grf(mp-N,L_0 / t^p L_0) \quad L\mapsto L/t^pL_0
	\]
	is a closed embedding.  Moreover, the standard determinant line bundle on $ \Grf(mp-N,L_0 / t^p L_0) $ restricts to the line bundle $ \mathscr L $ on $\overline{\Gr^\lambda} $.
\end{proposition}

Now assume that \(N = m, \mu = (1, \dots, 1) \) and \(p = 2\), so that $ mp - N = m $.  We consider the Pl\"ucker embedding $ \Grf(m, L_0/t^2 L_0)  \rightarrow \mathbb P(\bigwedge^m L_0/t^2 L_0) $.  Our aim is to study the chain of maps connecting $ \overline{\mathbb O_\lambda} \cap \mathfrak n $ with this projective space.

Let $ A \in  \overline{\mathbb O_\lambda} \cap \mathfrak n $.  In this case, the Mirkovi\'c--Vybornov isomorphism is simply given by $ \varphi(A) = [t{\mathrm{Id}} - A^{tr}]$ and the corresponding lattice is \[
L_A = (t\mathrm{Id} - A^{tr}) L_0 = \Span_{\C\xt}\{e_it-A^{tr}e_i \mid 1\le i\le m \}\,.
\]

Fix the basis
$$(v_1, \ldots,  v_m,v_{\bar 1}, \ldots, v_{\bar m}) \equiv ([e_1 t], \ldots, [e_{m}t],[-e_1], \ldots, [-e_{m}]) \text{ for } L_0 / t^2 L_0
$$
We write $ S = (1, \dots, m, \bar 1, \dots, \bar m) $ for the index set of this basis.  We have a resulting basis for $ \bigwedge^m L_0/t^2 L_0 $ indexed by subsets $ C \subset S $ of size $ m$.

In this basis,
\( L_A / t^2 L_0 \) is the row space of the $m\times 2m$ matrix
\(
\widetilde{A} = \begin{bmatrix} I & A \end{bmatrix}_\beta
\)
For any subset $ C \subset S $ of size $m $, we can consider the minor $ \Delta_C(\widetilde A) $ of this matrix using the columns $ C $.
Thus, we have established the following result.

\begin{proposition}
	Under the chain of maps
	\begin{equation}
	\overline{\mathbb O_\lambda} \cap \mathfrak n \rightarrow \overline{\Gr^\lambda} \cap S_-^0 \rightarrow \Grf(m, L_0 / t^2 L_0) \rightarrow \mathbb P(\bigwedge\nolimits^m L_0/t^2 L_0)
	\end{equation}
	$A $ is sent to $ [ \Delta_{C}(\widetilde A) ]_{C \in \binom{S}{m}} $.
\end{proposition}
We also note these maps are $T^\vee $-equivariant (in fact $ PGL_m$-equivariant) where $ T^\vee $ acts on $ \mathbb P(\bigwedge^m L_0/t^2 L_0) $ using the natural $GL_m $ action on $  L_0/t^2 L_0 $.  In particular, the basis vectors $ v_i $ and $ v_{\bar i} $ both have weight $ \varep_i $.

From this Proposition, it is immediate that $ \overline{\mathbb O_\lambda} \cap \mathfrak n $ is mapped into the affine space $ \mathbb A^{\binom{2m}{m} -1} \subset \mathbb P(\bigwedge^m L_0/t^2 L_0) $ defined by the condition that $ \Delta_{1,\dots, m} \ne 0 $ and moreover that for any $ \tau $, $ \Zo_\tau $ is the intersection of $ Z_\tau $ with this open affine space.  Thus we deduce the following corollary.

\begin{corollary} \label{co:IdealMV}
	The ideal of $ Z_\tau \subset \mathbb P(\bigwedge^m L_0/t^2 L_0) $ equals the homogenization of the kernel of the map $ \C[\{\Delta_C \mid C \in  \binom{S}{m}\}] \rightarrow \C[\Zo_\tau] $.
\end{corollary}

\subsubsection{Preprojective algebra modules}
\label{ss:ss13}
Fix $ \nu = Q_+ $.  We have a bijection $ \Irr \Lambda(\nu) \rightarrow B(\infty)_{-\nu} $.  We will now recall the composition $ \cup_\nu \Irr \Lambda(\nu) \rightarrow B(\infty) \rightarrow \N^{\Delta_+} $ which was studied in \cite{BaumannKamnitzerTingley}.

The convex order (\ref{eq:order}) on $ \Delta_+$ determines a sequence of indecomposable bricks $(B_\beta)_{\beta \in \Delta_+} $ labelled by the positive roots.  For this convex order, 
$$B_{\varep_i - \varep_j} = i \leftarrow i+1 \leftarrow \cdots \leftarrow j-1 $$
Let \(M \) be a $ \Lambda$-module. The \textbf{Harder--Narasimhan filtration} of \(M\) is the unique decreasing filtration $ (M_\beta)_{\beta \in \Delta_+} $ such that $ M_\beta / M_{\beta'} \cong B_\beta^{\oplus n(M)_\beta} $ where $ \beta < \beta' $ is a consecutive pair in the convex order on $ \Delta_+ $.  We define the Lusztig datum $n(M)_\bullet $ of $ M $ to be these multiplicities.

We define $ n(Y)_\bullet $ for any component $ Y \subset \Lambda(\nu) $ by setting $ n(Y)_\bullet := n(M)_\bullet $, where $ M $ is a general point of $ Y $.  The following result is contained in \cite{BaumannKamnitzerTingley} (see Remark 5.25 (ii)).

\begin{theorem} \label{th:fromBKT}
	The map $ \cup_\nu \Irr \Lambda(\nu) \rightarrow \N^{\Delta_+}  $ defined by $ Y \mapsto n(Y)_\bullet $ is a bijection.  Moreover, this bijection agrees with the composite $ \cup_\nu \Irr \Lambda(\nu) \rightarrow B(\infty) \rightarrow \N^{\Delta_+} $ (where the second map is Lusztig's bijection).
\end{theorem}

From Theorem \ref{thm:s1t2} and \ref{th:fromBKT}, we deduce the following.
\begin{corollary} \label{co:Compatible}
	Let $ \tau \in YT(\lambda)_\mu$ be a Young tableau and let $ \nu = \lambda - \mu $.  Let $ Z_\tau $ be the MV cycle constructed in \S \ref{se:Ztau}.  Let $Y_\tau \subset \Lambda(\nu) $ be component such that  $ n(Y)_\bullet = n(\tau)_\bullet $.
	
	Then the stable MV cycle $ t^{-\lambda} Z_\tau $ and the component $ Y_\tau $ correspond in the sense of \S \ref{se:measec}.
\end{corollary}
\subsubsection{Evidence for extra-compatibility in an $A_4$ example}
\label{ss:A4}
We take $ m = N = 5, \lambda = (2,2,1), \mu = (1,1,1,1,1)$ and consider
\[\tau = \young(12,34,5)
\]

Let $ Z = Z_\tau $ be the MV cycle defined from $ \tau $ in \S \ref{se:Ztau}.  Let $ Y = Y_\tau \subset \Lambda(\nu) $ and let $ M $ be a general point of $ Y_\tau $, i.e. a general module with $ n(M) = n(\tau)$.  This gives the simplest indecomposable $ M $ not covered by the analysis in \S \ref{se:Schubert}.

The pair $ t^{-\lambda}Z, Y $ are compatible, by Corollary \ref{co:Compatible}.  Moreover, $ b_{t^{-\lambda}Z} = c_Y$, as can be seen by a variant of the analysis from \S \ref{ss:NonUniq}.  In fact, we expect that $ t^{-\lambda}Z, Y $ are extra-compatible.  We will now prove the following result which gives evidence in this direction.  

\begin{theorem} \label{th:A4}
	\begin{enumerate}
		\item For all $ n \in \N $, we have $\dim \Gamma(t^{-\lambda} Z, \mathscr L^{\otimes n}) = \dim H^\bullet(F_n(M))$.
		\item For $ n = 1, 2 $ and all $ \mu $, we have $ \dim \Gamma(t^{-\lambda} Z, \mathscr L^{\otimes n})_\mu = \dim H^\bullet(F_{n,\mu}(M))
		$.
		\item We have $ \barD(b_{t^{-\lambda}Z}) = \barD(\xi_M) $.
	\end{enumerate}
\end{theorem}

\subsubsection{Generalized orbital variety and multidegree}
Let $ A \in \Zoo_\tau $ and write
\[A =\begin{bmatrix}
0 & a_1 & a_2 & a_3 & a_4 \\
0 & 0   & a_5 & a_6 & a_7 \\
0 & 0   & 0   & a_8 & a_9 \\
0 & 0   &   0 & 0   & a_{10} \\
0 & 0 & 0 & 0 & 0
\end{bmatrix}
\]

We apply (\ref{eq:zo})  to find the following conditions along with the equations they impose on the matrix entries of \(A\)
\[
\begin{aligned}
\begin{bmatrix}
0 & a_1\\
0 & 0
\end{bmatrix} &\in \mathbb O_{(2)}   &a_1 \ne 0 \\
\begin{bmatrix}
0 & a_1 & a_2\\
0 & 0 & a_5\\
0 & 0 & 0
\end{bmatrix}  &\in \mathbb O_{(2,1)}  &a_5 = 0\\
\begin{bmatrix}
0 & a_1 & a_2 & a_3\\
0 & 0 & 0 & a_6 \\
0 & 0 & 0 & a_8\\
0 & 0 & 0 & 0
\end{bmatrix} &\in \mathbb O_{(2,2)}  &\phantom{hello}a_1a_6 + a_2a_8 = 0, a_8 \ne 0 \\
 A  &\in\mathbb O_\lambda  &a_{10} = \det\begin{bmatrix}
a_6 & a_7 \\
a_8 & a_9
\end{bmatrix} &= a_1a_7 + a_2a_9 = 0
\end{aligned}
\]
Altogether we find that $ \mathring Z = \Zo_\tau$ is the vanishing locus of
\[
I = (a_5, a_{10}, a_1a_6 + a_2a_8, a_7a_8 - a_6a_9, a_1a_7 + a_2a_9)
\]
and we verified using a computer that this ideal is prime.

Applying the algorithm from \cite[\S 1.5]{KZJ} (or by computer), we obtain the following. 
\begin{equation}
\label{eq:mdeg}
\mdeg[\mathfrak n] \mathring Z = {\left(\alpha_{1} \alpha_{2} + \alpha_{2}^{2} + \alpha_{2} \alpha_{3}\right)} \alpha_{4}^{2} + {\left(\alpha_{1} \alpha_{2}^{2} + \alpha_{2}^{3} + \alpha_{2} \alpha_{3}^{2} + 2 \, {\left(\alpha_{1} \alpha_{2} + \alpha_{2}^{2}\right)} \alpha_{3}\right)} \alpha_{4}
\end{equation}
\subsubsection{The MV cycle and its sections}
In this example, the desciption of the MV cycle \(Z\) from Corollary \ref{co:IdealMV}
can be simplified, since \(\mathring Z\) is contained in the subspace defined by the vanishing of $ a_5 $ and $ a_{10} $. By ignoring minors which are forced to be zero among the set of ${10 \choose 5}$ total possibilities, we can exhibit \(Z\) as a subvariety of $\mathbb{P}^{16}$ using the following set of minors.
$$
\begin{aligned}
u &= \Delta_{12345} & b_1 &= \Delta_{1345\overline{1}} & b_2 &= \Delta_{1245\overline{1}} & b_3 &= \Delta_{1235\overline{1}} & b_4 &= \Delta_{1234\overline{1}} & b_5 &= \Delta_{1235\overline{2}} \\
b_6 &= \Delta_{1234\overline{2}} & b_7 &= \Delta_{1235\overline{3}} & b_8 &= \Delta_{1234\overline{3}} & b_9 &= \Delta_{123\overline{12}} & b_{10} &= \Delta_{124\overline{12}} & b_{11} &= \Delta_{124\overline{13}} \\
b_{12} &= \Delta_{125\overline{12}} & b_{13} &= \Delta_{125\overline{13}} & b_{14} &= \Delta_{123\overline{13}} & b_{15} &= \Delta_{134\overline{13}} & b_{16} &= \Delta_{135\overline{13}}
\end{aligned}$$
Let \(P = \mathbb{C}[b_1, \dots, b_{16}, u]\) and
$$
\begin{aligned}
J_1 &= (b_9 - b_3b_6 + b_4b_5, b_{10} - b_2b_6, b_{11} - b_2b_8, \\
&\qquad \qquad b_{12} - b_2b_5,b_{13} - b_2b_7, b_{14} - b_3b_8 + b_4b_7, b_{15} - b_1b_8, b_{16} - b_1b_7) \\
J_2 &= (b_1b_5 + b_2b_7, b_6b_7 - b_5b_8, b_1b_6 + b_2b_8)
\end{aligned}
$$
where $ J_1 $ just comes from $ \C[\mathfrak n] $ and $ J_2 $ represents the additional relations coming from $ I $.
Thus $Z = \text{Proj}(P/J^h)$ for $J = J_1 + J_2$, where $J^h$ is the homogenization of $J$ with respect to $u$.
Using Macauley2, we obtain the following expression for the space of sections.
\begin{equation}
\label{eq:Hilb}
\dim\Gamma(Z, \mathcal{O}(n)) = \frac{(n+1)^2(n+2)^2(n + 3)(5n + 12)}{144}
\end{equation}

\subsubsection{The preprojective algebra module}
A general module $M $ with Lusztig data $ n(\tau) = (1,0,0,0,1,1,0,0,1,0)$ is 
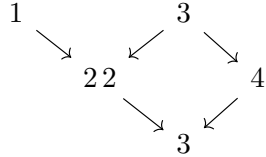
\[
\begin{tikzcd}[row sep = small,column sep = small]
1\ar[dr] & & 3\ar[dl]\ar[dr] \\
& 2\,2\ar[dr] & & 4\ar[dl] \\
& & 3
\end{tikzcd}
\]
with the maps chosen such that $\ker(M_2 \to M_3)$, $\text{im}(M_3 \to M_2)$ and $\text{im}(M_1\to M_2)$ are all distinct.

In fact, \(M\) has HN filtration with subquotients
\[
\begin{tabular}{c}
\(1\)\\
\hline
\(2\)\\
\hline
\(3\to 2\) \\
\hline
\(4\to 3\)
\end{tabular}
\]

\subsubsection{Flags of submodules}

We will now outline a recursive method for computing $\dim H^\bullet (F_{n}(M))$.
Given a dimension vector $\nu \in Q_+$, let
$F_n(M)^\nu$
be the component of $F_{n}(M)$ consisting of $(n+1)$-step flags for which the $n$th submodule in the chain has dimension vector $\nu$. Note that for each $ \nu $, all submodules of $M $ of dimension $ \nu $ are isomorphic (this is a special property of $M$).  This means that for any $\nu$, there exists a submodule $N$ such that
$$F_n(M)^\nu \cong F_{n-1}(N)\times F_{1,\nu}(M).$$
Using this recursive definition, we compute $\dim H^\bullet(F_n(M)^\nu)$ in Table \ref{tab:A4exgr1}.

\begin{table}
	\caption{Spaces of (chains of) submodules of $M$.}\label{tab:A4exgr1}
	\begin{tabular}{ |c|c|c|c| }
		\hline
		$\nu$ & $F_1(M)^\nu$ & $\dim H^\bullet(F_1(M)^\nu)$ & $\dim H^\bullet(F_n(M)^\nu)$\\
		\hline
		$(0, 0, 0, 0)$ & Point & 1 & $1$\\
		$(0, 1, 0, 0)$ & Point & 1 & $n$\\
		$(0, 0, 1, 0)$ & Point & 1 & $n$\\
		$(0, 0, 1, 1)$ & Point & 1 & $\frac{1}{2}n(n+1)$\\
		$(0, 1, 1, 0)$ & $\mathbb{P}^1$ & 2 & $\frac{1}{2}n(3n+1)$\\
		$(0, 1, 1, 1)$ & $\mathbb{P}^1$ & 2 & $\frac{1}{6}n(n+1)(5n+1)$\\
		$(0, 2, 1, 0)$ & Point & 1 & $\frac{1}{2}n^2(n+1)$\\
		$(1, 1, 1, 0)$ & Point & 1 & $\frac{1}{6}n(n+1)(n+2)$\\
		$(0, 1, 2, 1)$ & Point & 1 & $\frac{1}{12}n(n+1)^2(n+2)$\\
		$(0, 2, 1, 1)$ & Point & 1 & $\frac{1}{6}n^2(n+1)(2n+1)$\\
		$(1, 1, 1, 1)$ & Point & 1 & $\frac{1}{24}n(n+1)(n+2)(3n+1)$\\
		$(1, 2, 1, 0)$ & Point & 1 & $\frac{1}{6}n^2(n+1)(n+2)$\\
		$(0, 2, 2, 1)$ & Point & 1 & $\frac{1}{12}n^2(n+1)^2(n+2)$\\
		$(1, 2, 1, 1)$ & Point & 1 & $\frac{1}{24}n^2(n+1)(n+2)(3n+1)$\\
		$(1, 2, 2, 1)$ & Point & 1 & $\frac{1}{144}n^2(n+1)^2(n+2)(5n+7)$\\
		\hline
	\end{tabular}
\end{table}


Since $F_n(M)$ is the disjoint union of each of these varieties, we have $$\dim H^\bullet(F_n(M)) = \sum_{\nu}\dim H^\bullet(F_n(M)^\nu)$$ and so summing the above polynomials we get
\begin{equation}
\label{eq:HVn}
\dim H^\bullet(F_n(M)) = \frac{(n+1)^2(n+2)^2(n + 3)(5n + 12)}{144}.
\end{equation}

Together with (\ref{eq:Hilb}), this establishes Theorem \ref{th:A4}(i).

When $ n = 1, 2 $ we can take this computation further and prove Theorem \ref{th:A4}(ii).  (When $ n = 1$, this can be easily seen by comparing the weights of the variables $ u, b_1, \dots, b_{16} $ with the dimension vectors of submodules of $ M $ from Table \ref{tab:A4exgr1}, taking into account the shifting of weights given by Proposition \ref{pr:shiftWeights}.)

\subsubsection{Computation of the ``flag function''}

By Proposition \ref{pr:ExpanCh} and the definition of $ \xi_M $ from \S \ref{ssec:preproj}, we have that
\[
\barD(\xi_M) =  \sum_{\vi \in \Seq(\nu)} \chi(F_\vi(M))\, \barD_\vi
\]
We call $ \barD(\xi_M) $ the flag function of $ M$.  

For $\vi\in \Seq(\nu)$ among
\[
\begin{aligned}
(3, 4, 2, 3, 2, 1) & & (3, 2, 4, 3, 2, 1) & & (2, 3, 4, 2, 3, 1) & &
(2, 3, 2, 1, 4, 3) \\
(2, 3, 2, 4, 3, 1) & & (2, 3, 4, 2, 1, 3) & & (2, 3, 2, 4, 1, 3)  & & (3, 2, 1, 2, 4, 3) \\
(3, 4, 2, 1, 2, 3) & & (3, 2, 4, 1, 2, 3) & & (3, 2, 1, 4, 2, 3)
\end{aligned}
\]
the variety $F_{\vi}(M)$ is a point, so $\chi(F_{\vi}(M)) = 1$. For $\vi$ among
\[
\begin{aligned}
(3, 4, 2, 2, 3, 1) & & (3, 2, 4, 2, 3, 1) & & (3, 2, 2, 4, 3, 1) & & (3, 4, 2, 2, 1, 3) \\
(3, 2, 4, 2, 1, 3) & & (3, 2, 2, 4, 1, 3) & & (3, 2, 2, 1, 4, 3)
\end{aligned}
\]
we see that $F_{\vi}(M)\cong\mathbb P ^1$, so $\chi(F_{\vi}(M)) = 2$.
For all other values of $\vi$, $ F_\vi(M) = \emptyset$. 

The flag function is a rational function, but we can use $ p(\mu) $ to clear the denominator.  By direct computation, we obtain that the flag function of $ M $ is given by:

\begin{equation*}
\barD(\xi_M)
p(\mu)
= {\left(\alpha_{1} \alpha_{2} + \alpha_{2}^{2} + \alpha_{2} \alpha_{3}\right)} \alpha_{4}^{2} + {\left(\alpha_{1} \alpha_{2}^{2} + \alpha_{2}^{3} + \alpha_{2} \alpha_{3}^{2} + 2 \, {\left(\alpha_{1} \alpha_{2} + \alpha_{2}^{2}\right)} \alpha_{3}\right)} \alpha_{4}
\end{equation*}
Comparing with (\ref{eq:mdeg}) and applying Proposition \ref{pr:Dmdeg}, we obtain the proof of Theorem \ref{th:A4}(iii).
\subsection{Weak evidence for extra-compatibility in an \(A_5\) Example}
\label{ss:ss22}
Let \(\lambda = (2,2,1,1) \), let \(\mu = (1,1,1,1,1,1)\) and consider \[\tau = \young(13,25,4,6)\]

As before, let $ Z = Z_\tau $ be the MV cycle defined from $ \tau $ in \S \ref{se:Ztau}.  Let $ Y = Y_\tau \subset \Lambda(\nu) $.  A general point of $ Y $ is of the form $ M_a $, where $ a \in \C $ is a parameter.  This is the simplest example of a component whose general point is not rigid.

The pair $ t^{-\lambda}Z, Y $ are compatible and $ b_Z = c_Y$, as can be seen by a variant of the analysis from \S \ref{ss:NonUniq}.  We have the following weak evidence for extra-compatibility and for the equality of basis vectors. 
\begin{theorem} \label{th:A5}
	\begin{enumerate}
		\item For all $ \nu \in Q_+ $,  $ \dim \Gamma(t^{-\lambda} Z, \mathscr L)_{-\nu} = \dim H^\bullet(F_{1,\nu}(M_a))
		$.
		\item We have $ \barD(b_{t^{-\lambda}Z}) = \barD(\xi_{M_a}) $.
	\end{enumerate}
\end{theorem}

\subsubsection{Generalized orbital variety and multidegree}
%
We apply (\ref{eq:zo}) with the aid of a computer to find that the generalized orbital variety \(\mathring Z_\tau\) is cut out by the prime ideal
\begin{align*}
I = ({a}_{15}&,{a}_{10},{a}_{1}, {a}_{3} {a}_{6}-{a}_{2} {a}_{7}, {a}_{2} {a}_{12} +{a}_{3} {a}_{14}, {a}_{6} {a}_{12} +{a}_{7} {a}_{14}, {a}_{2} {a}_{11} +{a}_{3} {a}_{13}, \\
 {a}_{6} {a}_{11}& +{a}_{7} {a}_{13}, {a}_{12} {a}_{13} -{a}_{11}{a}_{14}, 
{a}_{5} {a}_{6} {a}_{13}-{a}_{2}{a}_{9} {a}_{13}-{a}_{4} {a}_{6} {a}_{14}+{a}_{2} {a}_{8} {a}_{14}, \\
{a}_{5} {a}_{7} {a}_{13}&-{a}_{3} {a}_{9} {a}_{13}-{a}_{4} {a}_{7} {a}_{14}+{a}_{3} {a}_{8} {a}_{14}, 
{a}_{5}{a}_{7} {a}_{11} -{a}_{3} {a}_{9} {a}_{11}-{a}_{4} {a}_{7})
\end{align*}
where $a_1,\ldots, a_{15}$ are the matrix entries of a $6\times 6$ upper triangular matrix
\[
A = \begin{bmatrix}
0 & a_1 & a_2 & a_3 & a_4 & a_5\\
0 & 0 & a_6 & a_7 & a_8 & a_9\\
0 & 0 & 0 & a_{10} & a_{11} & a_{12}\\
0 & 0 & 0 & 0 & a_{13} & a_{14}\\
0 & 0 & 0 & 0 & 0 & a_{15}\\
0 & 0 & 0 & 0 & 0 & 0
\end{bmatrix}
\]
in \(\mathfrak n\).

From here it is easy to compute $\mdeg[\mathfrak n]\mathring Z$ using a computer.

As in the previous section, we can use the ideal of the orbital variety $ \Zo $ to find the homogeneous ideal of the MV cycle $Z$ and thus to determine $ \Gamma(Z, \mathscr L) $.  We omit the details.

\subsubsection{The preprojective algebra module and its flag function}
In this example \(n(\tau) =(0,1,0,0,0,0,0,1,0,1,0,0,0,1,0)\).  The general module \(M_a \in Y \) is of the form
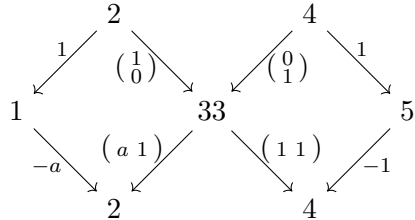
\[
\begin{tikzpicture}
\begin{scope}[scale=1.3,xshift=6cm]
\node (10) at (1,0){$2$};
\node (30) at (3,0){$4$};
\node (01) at (0,1){$1$};
\node (21) at (2,1){$33$};
\node (41) at (4,1){$5$};
\node (12) at (1,2){$2$};
\node (32) at (3,2){$4$};
\draw[->] (12) to node[pos=.3,left]{$\scriptstyle1$} (01);
\draw[->] (12) to node[pos=.6,left]{$\left(\bsm1\\0\esm\right)$} (21);
\draw[->] (32) to node[pos=.6,right]{$\left(\bsm0\\1\esm\right)$} (21);
\draw[->] (32) to node[pos=.3,right]{$\scriptstyle1$} (41);
\draw[->] (01) to node[pos=.6,left]{$\scriptstyle-a$} (10);
\draw[->] (21) to node[pos=.3,left]{$\left(\bsm a&1\esm\right)$} (10);
\draw[->] (21) to node[pos=.3,right]{$\left(\bsm1&1\esm\right)$} (30);
\draw[->] (41) to node[pos=.6,right]{$\scriptstyle-1$} (30);
\end{scope}
\end{tikzpicture}
\]
%
%

It is easy to determine all submodules of $ M_a $ and thus the space $ F_1(M_a)$.  Comparing with the computation of $ \Gamma(Z, \mathscr L) $ yields the proof of Theorem \ref{th:A5}(i).  We omit the details.

We can compute $\barD(\xi_{M_a})$ by enumerating composition series in the same manner as in the $A_4$ example; there are 148 sequences $ \vi $ with $ F_\vi(M_a) = \mathbb P^1 $ and 104 sequences $ \vi $ with $ F_\vi(M_a) = pt $.  Computing in this way, we get that $\barD(\xi_{M_a}){p(\mu)}$ is equal to the multidegree given in the previous section, yielding the proof of Theorem \ref{th:A5}(ii).

\subsection{Non-equality of basis vectors}
\label{ss:ss23}
Let \(\lambda = (4,4,2,2)\), let \(\mu = (2,2,2,2,2,2)\) and consider
\[\tau = \young(1133,2255,44,66)\]
so that \(\nu = 2\alpha_1 + 4\alpha_2 + 4\alpha_3 + 4\alpha_4 + 2\alpha_5\) and \(n(\tau) =  (0,2,0,0,0,0,0,2,0,2,0,0,0,2,0)\). 

As before, we let $ Z = Z_\tau $ and $ Y = Y_\tau $ giving a corresponding pair $ t^{-\lambda} Z $ and $ Y $.  However, we will now prove that $ b_Z \ne c_Y $.  Using Proposition \ref{prop:4statements}, it suffices to prove that $ \barD(b_Z) \ne \barD(c_Y) $.

A general point of $ Y $ is $ M_a \oplus M_{a'} $ with $ a\ne a' $.  Let $ I(\omega_2 + \omega_4) := I(\omega_2) \oplus I(\omega_4) $ be the injective $ \Lambda $ module (using the notation of \S \ref{ssec:spherical}).  We prove the following.

\begin{theorem} \label{th:AppendixMain}
	We have
	$$ \barD(b_Z) = \barD(\xi_{M_a \oplus M_{a'}}) - 2 \barD(\xi_{I(\omega_2 + \omega_4)}) $$
	and in particular, $ \barD(b_Z) \ne \barD(\xi_{M_a \oplus M_{a'}}) $ and thus $ b_Z \ne c_Y $.
\end{theorem}

\begin{proof}
	By Lemma \ref{le:multXi}, the computation of the right hand side is reduced to the previous section (and the easy computation of $ \barD(\xi_{I(\omega_i)}) $).
	
	On the other hand, for the left hand side, we use (\ref{eq:zo}) to give a description of the generalized orbital variety $ \Zo_\tau $.  Using the aid of a computer, we find that it is cut out by a prime ideal $ I $ inside a polynomial ring with 24 generators.
	
		From there, it is easy to compute the multidegree of $ \Zo_\tau $ and thus $ \barD(b_Z) $.

\end{proof}

\end{appendix}

\end{document}